\documentclass[11pt, oneside, article]{memoir}

\usepackage[readonly,extract=no,prefix=memoize/]{memoize}
\usepackage{newpxtext}
\usepackage{amssymb}
\usepackage[varg,bigdelims]{newpxmath}
\let\openbox\relax
\usepackage{amsthm}
\usepackage{amsmath}
\usepackage{mathtools}
\usepackage{aliascnt}
\usepackage{dutchcal}
\usepackage[multiple]{footmisc}
\usepackage{tikz}
\usepackage{xcolor}
\usepackage{enumitem}
\definecolor{darkgreen}{rgb}{0,0.35,0}
\usepackage[colorlinks,citecolor=darkgreen,linkcolor=darkgreen,urlcolor=black]{hyperref}
\usepackage[nameinlink,capitalize]{cleveref}

\crefname{chapter}{Section}{Sections}

\theoremstyle{plain}
\newtheorem{proposition}{Proposition}[chapter]

\theoremstyle{remark}
\newaliascnt{remark}{proposition}
\newtheorem{remark}[remark]{Remark}
\aliascntresetthe{remark}
\newtheorem{convention}[remark]{Convention}
\newtheorem{example}[remark]{Example}

\theoremstyle{definition}
\newaliascnt{definition}{proposition}
\newtheorem{definition}[definition]{Definition}
\aliascntresetthe{definition}

\crefname{proposition}{Proposition}{Propositions}
\crefname{remark}{Remark}{Remarks}
\crefname{definition}{Definition}{Definitions}

\DeclareSymbolFont{stmry}{U}{stmry}{m}{n}
\DeclareMathSymbol\fatsemi\mathop{stmry}{"23}
\newcommand{\then}{\mathbin{\fatsemi}}

\newcommand{\lift}[1]{\emlift{#1}}
\newcommand{\factor}[1]{{#1'}}
\newcommand{\emlift}[1]{\overline{#1}}
\newcommand{\kllift}[1]{\overline{#1}}
\newcommand{\blift}[1]{\overline{#1}}
\newcommand{\cdlift}[1]{\overline{#1}}

\newcommand{\id}{1}
\DeclareMathOperator{\ophelper}{op}
\DeclareMathOperator{\cohelper}{co}
\DeclareMathOperator{\trhelper}{tr}
\newcommand{\op}{^{\ophelper}}
\newcommand{\co}{^{\cohelper}}
\newcommand{\tran}{^{\trhelper}}

\DeclareMathAlphabet{\mathbbold}{U}{bbold}{m}{n}
\newcommand{\arrow}{\mathbf{2}}

\newcommand{\dash}{{\mathord{\text{--}}}}

\newcommand{\ccdots}{\hspace{.04cm}\cdots} 

\newcommand{\namedcat}[1]{\mathbf{#1}}
\newcommand{\cat}[1]{{#1}}
\newcommand{\tcat}[1]{\mathbf{#1}}
\newcommand{\dcat}[1]{\mathbb{#1}}
\newcommand{\ocell}[1]{{#1}}
\newcommand{\mnd}[1]{\mathbf{#1}}
\newcommand{\bim}[1]{{#1}}
\newcommand{\fun}[1]{{#1}}
\newcommand{\tfun}[1]{\mathbf{#1}}

\newcommand{\DD}{\dcat{D}}
\newcommand{\tX}{\tcat{X}}
\newcommand{\tY}{\tcat{Y}}
\newcommand{\tC}{\tcat{C}}
\newcommand{\tDelta}{\tcat{\Delta}}
\newcommand{\mT}{\mnd{t}}
\newcommand{\bM}{\bim{m}}
\newcommand{\fF}{\fun{f}}
\newcommand{\fG}{\fun{g}}
\newcommand{\fK}{\fun{k}}
\newcommand{\fX}{\fun{x}}
\newcommand{\fY}{\fun{y}}
\newcommand{\fR}{\fun{r}}
\newcommand{\fL}{\fun{l}}
\newcommand{\oS}{\ocell{S}}
\newcommand{\oX}{\ocell{X}}
\newcommand{\oY}{\ocell{Y}}
\newcommand{\cS}{\cat{S}}

\newcommand{\lpow}[1]{{{}^{#1}\mspace{-1.5mu}}}
\DeclarePairedDelimiter{\gen}{\langle}{\rangle}
\newcommand{\wcd}[1]{\gen{#1}}
\newcommand{\cd}[1]{\scalebox{0.8}{$\langle$}\mspace{-1.5mu}#1\mspace{-1.5mu}\scalebox{0.8}{$\rangle$}}
\newcommand{\pf}[2]{{{#1}{\#}{#2}}} 
\newcommand{\mX}{\wcd{\fX}}
\newcommand{\wmX}{\wcd{\fX}}
\newcommand{\mTX}{\pf{\mT}{\fX}}

\newcommand{\Set}{\namedcat{Set}}
\newcommand{\SET}{\namedcat{SET}}
\newcommand{\Cat}{\namedcat{Cat}}
\newcommand{\CAT}{\namedcat{CAT}}
\newcommand{\Span}{\namedcat{Span}}
\newcommand{\Mat}{\namedcat{Mat}}

\newcommand{\Loose}{\tcat{Loose}}
\newcommand{\Tight}{\tcat{Tight}}
\newcommand{\Mnd}{\tcat{Mnd}}
\newcommand{\EM}{\namedcat{EM}}
\newcommand{\Kl}{\namedcat{Kl}}
\newcommand{\Mod}{\namedcat{Mod}}

\newcommand{\MMnd}{\dcat{M}\tcat{nd}}
\newcommand{\MMndret}{\dcat{M}\tcat{nd}_\mathrm{ret}}
\newcommand{\EEM}{\dcat{E}\namedcat{M}}
\newcommand{\EEMret}{\EEM_{\mathrm{ret}}}
\newcommand{\SSpan}{\dcat{S}\tcat{pan}}
\newcommand{\MMat}{\dcat{M}\tcat{at}}
\newcommand{\MMod}{\dcat{M}\tcat{od}}
\newcommand{\CCat}{\dcat{C}\tcat{at}}
\newcommand{\SSq}{\dcat{S}\tcat{q}}

\newcommand{\lax}{\mathrm{lax}}
\newcommand{\HHom}{\dcat{H}\namedcat{om}_{\mathrm{co/lax}}}
\newcommand{\grayl}{\otimes_\mathrm{lax}}

\newcommand{\Alg}[2]{#2^{#1}}
\newcommand{\AlgTS}{\Alg{\mT}{\oS}}

\newcommand{\mcar}[1]{{\fR^{#1}}}
\newcommand{\mlcar}[1]{{\fL^{#1}}}
\newcommand{\mcarmod}[1]{{\rho^{#1}}}

\newcommand{\mcarT}{\mcar{\mT}}
\newcommand{\mlcarT}{\mlcar{\mT}}
\newcommand{\mcarmodT}{\mcarmod{\mT}}

\newcommand{\OpAlg}[2]{\lpow{#1}{#2}}
\newcommand{\OpAlgTS}{\OpAlg{\mT}{\oS}}

\newcommand{\mkcar}[1]{\lpow{#1}\fL}
\newcommand{\mrkcar}[1]{\lpow{#1}\fR}
\newcommand{\mkcarmod}[1]{\lpow{#1}\lambda}

\newcommand{\mkcarT}{\mkcar{\mT}}
\newcommand{\mrkcarT}{\mrkcar{\mT}}
\newcommand{\mkcarmodT}{\mkcarmod{\mT}}

\newcommand{\cdmod}[1]{\epsilon}
\newcommand{\pfmod}[1]{\epsilon}
\newcommand{\mXmod}{\cdmod{\mX}}
\newcommand{\mTXmod}{\pfmod{\mTX}}

\newcommand{\NEarrow}{\rotatebox[origin=c]{45}{\(\Rightarrow\)}}
\newcommand{\NWarrow}{\rotatebox[origin=c]{135}{\(\Rightarrow\)}}
\newcommand{\SWarrow}{\rotatebox[origin=c]{-135}{\(\Rightarrow\)}}
\newcommand{\SEarrow}{\rotatebox[origin=c]{-45}{\(\Rightarrow\)}}

\newcommand{\dN}{\mathrm{N}}
\newcommand{\dE}{\mathrm{E}}
\newcommand{\dS}{\mathrm{S}}
\newcommand{\dW}{\mathrm{W}}
\newcommand{\dNE}{\mathrm{NE}}
\newcommand{\dSE}{\mathrm{SE}}
\newcommand{\dNW}{\mathrm{NW}}
\newcommand{\dSW}{\mathrm{SW}}
\newcommand{\dNNE}{\mathrm{NNE}}
\newcommand{\dSSE}{\mathrm{SSE}}
\newcommand{\dNNW}{\mathrm{NNW}}
\newcommand{\dSSW}{\mathrm{SSW}}
\newcommand{\dNEE}{\mathrm{NEE}}
\newcommand{\dSEE}{\mathrm{SEE}}
\newcommand{\dNWW}{\mathrm{NWW}}
\newcommand{\dSWW}{\mathrm{SWW}}

\newcommand{\dL}{\mathrm{L}}
\newcommand{\dR}{\mathrm{R}}

\definecolor{purp}{RGB}{175,0,175}
\definecolor{dpurp}{RGB}{75,0,75}
\definecolor{dgray}{RGB}{40,40,40}
\definecolor{dgreen}{RGB}{0,150,0}
\definecolor{ddgreen}{RGB}{0,130,0}
\definecolor{dorange}{RGB}{150,50,0}
\definecolor{dyellow}{RGB}{220,180,0}
\definecolor{ddyellow}{RGB}{198,162,0}
\definecolor{dddyellow}{RGB}{176,144,0}
\definecolor{dred}{RGB}{150,0,0}
\definecolor{dblue}{RGB}{0,0,150}
\definecolor{pink}{RGB}{255,150,150}
\definecolor{beige}{RGB}{125,125,65}
\definecolor{choc}{RGB}{115,76,38}
\definecolor{bg}{RGB}{230,230,230}
\definecolor{bgd}{RGB}{178,178,178}
\definecolor{bgm}{RGB}{200,200,200}

\usetikzlibrary{patterns.meta,decorations.pathmorphing}

\tikzset{every picture/.style={baseline={([yshift=-3.5pt]current bounding box.center)},scale=.125}}

\pgfdeclaredecoration{teeth}{initial}{
  \state{initial}[width=\pgfdecorationsegmentlength]{
    \pgfpathmoveto{\pgfqpoint{-.75\pgfdecorationsegmentlength}{0pt}}
    \pgfpathlineto{\pgfqpoint{0pt}{\pgfdecorationsegmentamplitude}}
    \pgfpathlineto{\pgfqpoint{.75\pgfdecorationsegmentlength}{0pt}}
  }
}

\tikzset{boundc/.style={rounded corners}}
\tikzset{bound/.style={rounded corners,draw=black!50}}
\tikzset{o/.style={line width=2}}
\tikzset{t/.style={fill}}
\tikzset{tcirc/.style={t,circle,fill=#1,inner sep=2.5}}
\tikzset{tc/.style={t,circle,inner sep=1.75}}
\tikzset{olax/.style={draw=none,postaction={decorate,decoration={curveto,raise=1pt},draw,line width=0.7pt},postaction={decorate,decoration={teeth,raise=.85pt,segment length=3.5pt,amplitude=-4pt},fill}}}
\tikzset{olaxr/.style={draw=none,postaction={decorate,decoration={curveto,raise=-1pt},draw,line width=0.7pt},postaction={decorate,decoration={teeth,raise=-.85pt,segment length=3.5pt,amplitude=4pt},fill}}}
\tikzset{oclax/.style={olaxr}}
\tikzset{oclaxr/.style={olax}}
\tikzset{olaxs/.style={draw,line width=1pt,line cap=round,postaction={decorate,decoration={teeth, segment length=2.5pt,amplitude=-4pt},fill}}}
\tikzset{olaxo/.style={draw=none,postaction={decorate,decoration={curveto,raise=1.3pt},draw,line width=0.7pt},postaction={decorate,decoration={teeth,raise=1.15pt,segment length=3.5pt,amplitude=-4pt},fill}}}
\tikzset{oclaxo/.style={draw=none,postaction={decorate,decoration={curveto,raise=-1.3pt},draw,line width=0.7pt},postaction={decorate,decoration={teeth,raise=-1.15pt,segment length=3.5pt,amplitude=4pt},fill}}}

\tikzset{tmodi/.style={t,fill=white,draw=black,rectangle,minimum width=10,minimum height=5,inner sep=0,rounded corners=1}}
\tikzset{tspec/.style={tmodi}}
\tikzset{tspech/.style={t,rectangle,minimum width=13.5,minimum height=8.5,inner sep=0,rounded corners=2}}
\tikzset{tspechb/.style={t,fill=black,rectangle,minimum width=15.5,minimum height=10.5,inner sep=0,rounded corners=2.75}}
\tikzset{tspechs/.style={t,rectangle,minimum width=12,minimum height=7,inner sep=0,rounded corners=1.75}}
\tikzset{tspechsb/.style={t,fill=black,rectangle,minimum width=14,minimum height=9,inner sep=0,rounded corners=2.25}}
\tikzset{tdmodi/.style={t,fill=white,draw=black,rectangle,minimum width=10,minimum height=10,inner sep=0,rounded corners=1,rotate=45}}
\tikzset{tdspec/.style={tdmodi}}
\tikzset{tdspech/.style={t,rectangle,minimum width=13.5,minimum height=13.5,inner sep=0,rounded corners=2,rotate=45}}
\tikzset{tdmodis/.style={t,fill=white,draw=black,rectangle,minimum width=10,minimum height=10,inner sep=0,rounded corners=1}}
\tikzset{tdspecs/.style={tdmodis}}
\tikzset{tdspecsh/.style={t,rectangle,minimum width=13.5,minimum height=13.5,inner sep=0,rounded corners=2}}
\tikzset{em/.style={pattern={Lines[angle=45, distance=3pt, line width=0.4pt]},pattern color=#1}}
\tikzset{emb/.style={pattern={Lines[angle=45, distance=3pt, line width=1.6pt]},pattern color=#1}}
\tikzset{kl/.style={pattern={Lines[angle=-45, distance=3pt, line width=0.4pt]},pattern color=#1}}
\tikzset{klpl/.style={pattern={Lines[xshift=2.1pt, angle=-45, distance=3pt, line width=0.4pt]},pattern color=#1}}
\tikzset{klplb/.style={pattern={Lines[xshift=2.1pt, angle=-45, distance=3pt, line width=1.6pt]},pattern color=#1}}
\tikzset{emon/.style={pattern={Lines[angle=45, distance=6pt, line width=0.4pt]},pattern color=#1}}
\tikzset{klon/.style={pattern={Lines[angle=-45, distance=6pt, line width=0.4pt]},pattern color=#1}}
\tikzset{emoff/.style={pattern={Lines[xshift=4.2pt, angle=45, distance=6pt, line width=0.4pt]},pattern color=#1}}
\tikzset{kloff/.style={pattern={Lines[xshift=4.2pt, angle=-45, distance=6pt, line width=0.4pt]},pattern color=#1}}
\tikzset{omon/.style={o}}
\tikzset{oem/.style={line width=1.5}}
\tikzset{okl/.style={line width=1.5}}
\tikzset{omod/.style={o,preaction={draw=#1,line width=2.5,decorate,decoration={curveto,raise=1pt},shorten >=.5,shorten <=.5}}}
\tikzset{olmod/.style={o,preaction={draw=#1,line width=2.5,decorate,decoration={curveto,raise=-1pt},shorten >=.5,shorten <=.5}}}
\tikzset{omods/.style={o,transform canvas={xshift=-.75},preaction={draw=#1,transform canvas={xshift=-.75},line width=2.5,decorate,decoration={curveto,raise=1pt},shorten >=.5,shorten <=.5}}}
\tikzset{olmods/.style={o,transform canvas={xshift=.75},preaction={draw=#1,transform canvas={xshift=.75},line width=2.5,decorate,decoration={curveto,raise=-1pt},shorten >=.5,shorten <=.5}}}
\tikzset{obmod/.style={line width=1.5}}
\tikzset{obmodh/.style={line width=4}}
\tikzset{ocdb/.style={line width=3}}
\tikzset{ocd/.style={draw=bgd,line width=1}}
\tikzset{tcdb/.style={t,circle,fill=black,inner sep=1.05}}
\tikzset{tcd/.style={t,circle,fill=bgd,inner sep=.35}}
\tikzset{opfb/.style={line width=4}}
\tikzset{tpfb/.style={t,circle,fill=black,inner sep=1.45}}

\tikzset{tmon/.style={t,circle,inner sep=.75}}
\tikzset{thom/.style={inner sep=4.95,path picture={\draw[olaxs,#1] ([xshift=0,yshift=1.25]path picture bounding box.west) -- ([xshift=4,yshift=1.25]path picture bounding box.east);}}}
\tikzset{tmhomi/.style={tmodi}}
\tikzset{tmhom/.style={tmodi}}
\tikzset{tmhomh/.style={t,rectangle,minimum width=12.5,minimum height=7.5,inner sep=0,rounded corners=1.75}}
\tikzset{tmhomw/.style={t,fill=white,draw=black,rectangle,minimum width=12.5,minimum height=5,inner sep=0,rounded corners=1}}
\tikzset{tmhomwh/.style={t,rectangle,minimum width=15,minimum height=7.5,inner sep=0,rounded corners=1.75}}
\tikzset{tmhomt/.style={t,fill=white,draw=black,rectangle,minimum width=5,minimum height=5,inner sep=0,rounded corners=1}}
\tikzset{tmhomth/.style={t,rectangle,minimum width=7.5,minimum height=7.5,inner sep=0,rounded corners=1.75}}
\tikzset{tmhomsr/.style={t,fill=white,draw=black,rectangle,minimum width=12.5,minimum height=12.5,inner sep=0,rounded corners=1,rotate=45}}
\tikzset{tmhomsrh/.style={t,rectangle,minimum width=15,minimum height=15,inner sep=0,rounded corners=2,rotate=45}}
\tikzset{tmhoms/.style={t,fill=white,draw=black,rectangle,minimum width=12.5,minimum height=12.5,inner sep=0,rounded corners=1}}
\tikzset{tmhomsh/.style={t,rectangle,minimum width=15,minimum height=15,inner sep=0,rounded corners=1.75}}
\tikzset{tmpoint/.style={transform canvas={shift={(-.08pt,.11pt)}},t,circle,inner sep=1.175}}
\tikzset{omlip/.style={transform canvas={shift={(.06,.04)}},omon}}
\tikzset{arken/.style={fill=black,opacity=0}}
\tikzset{darken/.style={fill=black,opacity=.1}}
\tikzset{ddarken/.style={fill=black,opacity=.2}}
\tikzset{dddarken/.style={fill=black,opacity=.3}}
\tikzset{eback/.style={dddarken}}
\tikzset{nback/.style={darken}}
\tikzset{wback/.style={arken}}
\tikzset{sback/.style={ddarken}}
\tikzset{ewire/.style={blue}}
\tikzset{nwire/.style={ddgreen}}
\tikzset{wwire/.style={dddyellow}}
\tikzset{swire/.style={red}}
\tikzset{label/.style={font=\scriptsize}}

\tikzset{ob/.style={font=\small}}
\tikzset{a/.style={->,shorten <=-1,shorten >=-1}}
\tikzset{ab/.style={<-,shorten <=-1,shorten >=-1}}
\tikzset{lon/.style={shorten <=-1.5,shorten >=-1.5}}
\tikzset{long/.style={shorten <=-3,shorten >=-3}}
\tikzset{arr/.style={font=\tiny}}
\tikzset{cell/.style={font=\footnotesize}}
\tikzset{eq/.style={line width=2.4pt,postaction={draw=white, line width=1.6pt, shorten <=-0.05pt, shorten >=-0.05pt}}}
\tikzset{ar/.style={above right=-2}}
\tikzset{al/.style={above left=-2}}
\tikzset{br/.style={below right=-2}}
\tikzset{bl/.style={below left=-2}}
\tikzset{arcl/.style={above right=-3}}
\tikzset{alcl/.style={above left=-3}}
\tikzset{brcl/.style={below right=-3}}
\tikzset{blcl/.style={below left=-3}}
\tikzset{acl/.style={above=-2}}
\tikzset{bcl/.style={below=-2}}
\tikzset{arclcl/.style={above right=-4}}
\tikzset{alclcl/.style={above left=-4}}
\tikzset{brclcl/.style={below right=-4}}
\tikzset{blclcl/.style={below left=-4}}
\tikzset{aclcl/.style={above=-3}}
\tikzset{bclcl/.style={below=-3}}

\newcommand{\preq}{\hspace{35pt}}
\newcommand{\qqq}{\hspace{237.5pt}}
\newcommand{\eqquad}{\hspace{3pt}}
\newcommand{\equad}{\hspace{-2pt}}
\newcommand{\eqq}{\eqquad=\eqquad}
\newcommand{\eqspace}[1]{\hspace{#1}=\hspace{#1}}
\newcommand{\eq}{\equad=\equad}
\newcommand{\afterimage}{\vspace{3pt}}
\newcommand{\mypath}{}
\newcommand{\mypatha}{}
\newcommand{\mypathb}{}
\newcommand{\mypathc}{}
\newcommand{\mypathd}{}
\newcommand{\mypathe}{}
\newcommand{\mypathf}{}
\newcommand{\mypathg}{}

\tikzset{every picture/.style={baseline={([yshift=-3.5pt]current bounding box.center)},scale=.125}}

\setlist[itemize,enumerate]{topsep=3pt,itemsep=0pt,parsep=0pt,partopsep=0pt}

\settrims{0pt}{0pt}
\settypeblocksize{*}{34pc}{*}
\setlrmargins{*}{*}{1}
\setulmarginsandblock{.98in}{.98in}{*}
\setheadfoot{\onelineskip}{2\onelineskip}
\setheaderspaces{*}{1.5\onelineskip}{*}
\checkandfixthelayout

\setsecheadstyle{\normalfont\large\bfseries}
\setsubsecheadstyle{\normalfont\normalsize\bfseries}
\setsubsubsecheadstyle{\normalfont\normalsize\itshape}

\begin{document}

\title{MONADS IN 2-CATEGORIES}
\author{Aaron David Fairbanks}
\date{}

\maketitle

\vspace{-3.5em}
\begin{abstract}
  This is a condensed overview of the formal theory of monads in a
  2-category. Commutative diagrams and string diagrams are given side
  by side. In the string diagrams, each concept is visualized in a way
  that suggests its properties using topological intuition. For
  example, monads themselves appear as channels of fluid, suggesting
  the unit and associativity laws by topological deformation. We cover
  monads, modules, (co)lax 1-cells and the 2-cells between them,
  algebra objects, bimodules, distributive laws, codensity monads, and
  pushforward monads. In addition to the standard 2-categories of
  monads in a 2-category, we also define two double categories of
  monads in a 2-category. For example, applied to spans, this yields
  the two double categories of categories, functors, and retrofunctors
  upon restricting to the 1-cells carried by functions.
\end{abstract}

\tableofcontents*

\pagebreak

\phantomsection
\addcontentsline{toc}{chapter}{Introduction}
\chapter*{Introduction}

This is a condensed overview of the formal theory of monads in a
2-category. The paper is limited to a few definitions and results per
section, originally written as an appendix to
supplement~\cite{fairbanks-carlson-spivak}. We cover monads, modules,
lax and colax 1-cells of monads, algebra objects~\cite{street:monads}
(a.k.a.\ Eilenberg-Moore objects), bimodules~\cite{wood:ii},
distributive laws~\cite{beck:distributive}, codensity
monads~\cite{dubuc}, and pushforward monads~\cite{mateo}.

In \cref{sec:doublemonads} we also draw attention to two double
categories $\MMnd(\tX)$ and $\EEM(\tX)$ of monads in a 2-category
$\tX$, extending the two 2-categories $\Mnd(\tX)$ and $\EM(\tX)$
from~\cite{lack-street}. Both have the same objects (monads) and
arrows (lax and colax 1-cells), but the 2-cells are different. The
first double category --- whose 2-cells are called simply \emph{monad
  2-cells} --- can be viewed as a special case of the double category of
functors, lax and colax transformations, and modifications between two
2-categories from~\cite[Proposition A.8]{fairbanks-shulman}. It can
also be viewed as a special case of the double category $\MMnd(\DD)$
of loose monads in a double category
from~\cite{fiore-gambino-kock}. The second --- whose 2-cells we call
\emph{monad specializations}\footnote{The name is motivated
  by~\cite{garner:ionads,fairbanks-carlson-spivak}, in which this kind
  of 2-cell between colax 1-cells of comonads generalizes the usual
  notion of specialization between points in a topological space.} ---
can be viewed as a special case of the double category
$\EEM(\DD)$ 
of loose monads in a double category $\DD$ mentioned in~\cite[Chapter
7]{clarke}.


Like~\cite{hinze-marsden:string,hinze-marsden}, we use the string
diagram calculus for 2-categories. This is to make definitions and
calculations more transparent by visualizing their structure. For
readers not comfortable with string diagrams, we include pasting
diagrams as well. Also, having the pasting diagrams beside the string
diagrams permits more artistic liberty in the string diagrams with
less danger of confusion. We frequently draw individual regions
(0-cells), strings (1-cells), and nodes (2-cells) with multiple colors
or special textures, suggesting certain properties of the cells they
represent.
To match the diagrams, we write composition in a 2-category in
diagrammatic order $\fF \then \fG$, rather than the usual Leibniz
order $\fG \circ \fF$.

Knowledge of the definition of 2-category (and in
\cref{sec:doublemonads}, double category) is assumed; see
e.g.~\cite{kelly-street}. It is also useful to know about adjunctions
and extensions in 2-categories, though we spell out details when these
come up. The first three sections are prerequisites for the others,
which mostly do not depend on one another.


\pagebreak

\chapter{Monads}\label{sec:monads}

\begin{definition}
  A \emph{monad}\footnotemark\ on an object $\oS$ in a 2-category is a
  1-cell
  \[\preq
    \mathclap{\mT \colon \oS \to \oS}
    \qqq
    \mathclap
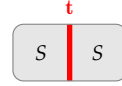
{
      \begin{tikzpicture}[yscale=.5]
        \begin{scope}
          \clip [boundc] (-6,-6) rectangle (6,6);
          \fill [bg] (-6,-6) rectangle (6,6);
          \draw [omon,red] (0,-6) -- (0,6);
        \end{scope}
        \draw [bound] (-6,-6) rectangle (6,6);
        \node [label] at (-3,0) {$\oS$};
        \node [label] at (3,0) {$\oS$};
        \node [label,red,above] at (0,6) {$\mT$};
        \node [label,red,below] at (0,-6) {$\phantom{\mT}$};
      \end{tikzpicture}
    }
  \]
  equipped with a 2-cell $\mu \colon \mT \then \mT \Rightarrow \mT$, called the
  \emph{multiplication}
  \[\preq
    \mathclap
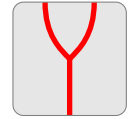
{
      \begin{tikzpicture}
        \node [ob] (s) at (-7,0) {$\oS$};
        \node [ob] (m) at (0,7) {$\oS$};
        \node [ob] (t) at (7,0) {$\oS$};
        \draw [a] (s) -- node[arr,al] {$\mT$} (m);
        \draw [a] (m) -- node[arr,ar] {$\mT$} (t);
        \draw [a] (s) -- node[arr,below] {$\mT$} (t);
        \node [cell] at (0,2.5) {$\mu$};
      \end{tikzpicture}
    }
    \qqq
    \mathclap{
      \begin{tikzpicture}
        \begin{scope}
          \clip [boundc] (-6,-6) rectangle (6,6);
          \fill [bg] (-6,-6) rectangle (6,6);
          \draw [omon,red] (-2.6,6) .. controls +(0,-3.5) and +(-1,1) .. (0,0);
          \draw [omon,red] (2.6,6) .. controls +(0,-3.5) and +(1,1) .. (0,0);
          \draw [omon,red] (0,-6) -- (0,0);
          \node [tmon,red] at (0,0) {};
        \end{scope}
        \draw [bound] (-6,-6) rectangle (6,6);
      \end{tikzpicture}
    }
  \]
  and a 2-cell $\eta \colon \id_{\oS} \Rightarrow \mT$,
  called the \emph{unit}
  \[\preq
    \mathclap
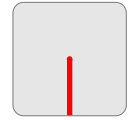
{
      \begin{tikzpicture}
        \node [ob] (s) at (-7,0) {$\oS$};
        \node [ob] (t) at (7,0) {$\oS$};
        \draw [eq] (s) .. controls +(4.5,7.5) and +(-4.5,7.5) .. (t);
        \draw [a] (s) -- node[arr,below] {$\mT$} (t);
        \node [cell] at (0,2.5) {$\eta$};
      \end{tikzpicture}
    }
    \qqq
    \mathclap{
      \begin{tikzpicture}
        \begin{scope}
          \clip [boundc] (-6,-6) rectangle (6,6);
          \fill [bg] (-6,-6) rectangle (6,6);
          \draw [omon,red] (0,-6) -- (0,0);
          \node [tmon,red] at (0,0) {};
        \end{scope}
        \draw [bound] (-6,-6) rectangle (6,6);
      \end{tikzpicture}
    }
  \]
  satisfying the associativity and unit laws
  \[\preq
    \mathclap{
      \begin{tikzpicture}[xscale=.75,yscale=1.125]
        \node [ob] (s) at (-14,0) {$\oS$};
        \node [ob] (mm) at (-9,7) {$\oS$};
        \node [ob] (m) at (2,7) {$\oS$};
        \node [ob] (t) at (7,0) {$\oS$};
        \draw [a] (s) -- node[arr,al] {$\mT$} (mm);
        \draw [a] (mm) -- node[arr,above] {$\mT$} (m);
        \draw [a] (s) -- node[arr,br] {$\mT$} (m);
        \draw [a] (m) -- node[arr,ar] {$\mT$} (t);
        \draw [a] (s) -- node[arr,below] {$\mT$} (t);
        \node [cell] at (.5,2.5) {$\mu$};
        \node [cell] at (-6.75,4.75) {$\mu$};
      \end{tikzpicture}
      \eq
      \begin{tikzpicture}[xscale=.75,yscale=1.125]
        \node [ob] (s) at (-7,0) {$\oS$};
        \node [ob] (m) at (-2,7) {$\oS$};
        \node [ob] (mm) at (9,7) {$\oS$};
        \node [ob] (t) at (14,0) {$\oS$};
        \draw [a] (m) -- node[arr,above] {$\mT$} (mm);
        \draw [a] (mm) -- node[arr,ar] {$\mT$} (t);
        \draw [a] (s) -- node[arr,al] {$\mT$} (m);
        \draw [a] (m) -- node[arr,bl] {$\mT$} (t);
        \draw [a] (s) -- node[arr,below] {$\mT$} (t);
        \node [cell] at (-.5,2.5) {$\mu$};
        \node [cell] at (6.75,4.75) {$\mu$};
      \end{tikzpicture}
    }
    \qqq
    \mathclap{
      \begin{tikzpicture}
        \begin{scope}
          \clip [boundc] (-7,-7) rectangle (7,7);
          \fill [bg] (-7,-7) rectangle (7,7);
          \draw [omon,red] (-4,7) .. controls +(0,-2.5) and +(-1,1) .. (-2,2.5);
          \draw [omon,red] (0,7) .. controls +(0,-2.5) and +(1,1) .. (-2,2.5);
          \draw [omon,red] (4,7) .. controls +(0,-5.5) and +(1,1) .. (0,-2.5);
          \draw [omon,red] (-2,2.5) .. controls +(0,-2.5) and +(-1,1) .. (0,-2.5);
          \draw [omon,red] (0,-7) -- (0,-2.5);
          \node [tmon,red] at (-2,2.5) {};
          \node [tmon,red] at (0,-2.5) {};
        \end{scope}
        \draw [bound] (-7,-7) rectangle (7,7);
      \end{tikzpicture}
      \eqq
      \begin{tikzpicture}
        \begin{scope}
          \clip [boundc] (-7,-7) rectangle (7,7);
          \fill [bg] (-7,-7) rectangle (7,7);
          \draw [omon,red] (-4,7) .. controls +(0,-5.5) and +(-1,1) .. (0,-2.5);
          \draw [omon,red] (0,7) .. controls +(0,-2.5) and +(-1,1) .. (2,2.5);
          \draw [omon,red] (4,7) .. controls +(0,-2.5) and +(1,1) .. (2,2.5);
          \draw [omon,red] (2,2.5) .. controls +(0,-2.5) and +(1,1) .. (0,-2.5);
          \draw [omon,red] (0,-7) -- (0,-2.5);
          \node [tmon,red] at (2,2.5) {};
          \node [tmon,red] at (0,-2.5) {};
        \end{scope}
        \draw [bound] (-7,-7) rectangle (7,7);
      \end{tikzpicture}
    }
  \]
  \vspace{-5pt}
  \[\preq
    \mathclap{
      \begin{tikzpicture}
        \path (0,11) -- (0,-4);
        \node [ob] (s) at (-7,0) {$\oS$};
        \node [ob] (m) at (3,7) {$\oS$};
        \node [ob] (t) at (7,0) {$\oS$};
        \draw [a] (s) -- node[arr,br] {$\mT$} (m);
        \draw [a] (m) -- node[arr,ar] {$\mT$} (t);
        \draw [a] (s) -- node[arr,below] {$\mT$} (t);
        \draw [eq] (s) .. controls +(1,6.5) and +(-6.5,1) .. (m);
        \node [cell] at (2,2.5) {$\mu$};
        \node [cell] at (-3,4.75) {$\eta$};
      \end{tikzpicture}
      \eq
      \begin{tikzpicture}
        \path (0,11) -- (0,-4);
        \node [ob] (s) at (-6,0) {$\oS$};
        \node [ob] (t) at (6,0) {$\oS$};
        \draw [a] (s) .. controls +(3,9) and +(-3,9) .. node[arr,above] {$\mT$} (t);
        \draw [a] (s) -- node[arr,below] {$\mT$} (t);
        \node [cell] at (0,3.25) {$=$};
      \end{tikzpicture}
      \eq
      \begin{tikzpicture}
        \path (0,11) -- (0,-4);
        \node [ob] (s) at (-7,0) {$\oS$};
        \node [ob] (m) at (-3,7) {$\oS$};
        \node [ob] (t) at (7,0) {$\oS$};
        \draw [a] (s) -- node[arr,al] {$\mT$} (m);
        \draw [a] (m) -- node[arr,bl] {$\mT$} (t);
        \draw [a] (s) -- node[arr,below] {$\mT$} (t);
        \draw [eq] (m) .. controls +(6.5,1) and +(-1,6.5) .. (t);
        \node [cell] at (-2,2.5) {$\mu$};
        \node [cell] at (3,4.75) {$\eta$};
      \end{tikzpicture}
    }
    \qqq
    \mathclap{
      \begin{tikzpicture}
        \begin{scope}
          \clip [boundc] (-7,-7) rectangle (7,7);
          \fill [bg] (-7,-7) rectangle (7,7);
          \draw [omon,red] (2.6,7) .. controls +(0,-5.5) and +(1,1) .. (0,-2.5);
          \draw [omon,red] (-2,2.5) .. controls +(0,-2.5) and +(-1,1) .. (0,-2.5);
          \draw [omon,red] (0,-7) -- (0,-2.5);
          \node [tmon,red] at (-2,2.5) {};
          \node [tmon,red] at (0,-2.5) {};
        \end{scope}
        \draw [bound] (-7,-7) rectangle (7,7);
      \end{tikzpicture}
      \eqq
      \begin{tikzpicture}
        \begin{scope}
          \clip [boundc] (-7,-7) rectangle (7,7);
          \fill [bg] (-7,-7) rectangle (7,7);
          \draw [omon,red] (0,-7) -- (0,7);
        \end{scope}
        \draw [bound] (-7,-7) rectangle (7,7);
      \end{tikzpicture}
      \eqq
      \begin{tikzpicture}
        \begin{scope}
          \clip [boundc] (-7,-7) rectangle (7,7);
          \fill [bg] (-7,-7) rectangle (7,7);
          \draw [omon,red] (-2.6,7) .. controls +(0,-5.5) and +(-1,1) .. (0,-2.5);
          \draw [omon,red] (2,2.5) .. controls +(0,-2.5) and +(1,1) .. (0,-2.5);
          \draw [omon,red] (0,-7) -- (0,-2.5);
          \node [tmon,red] at (2,2.5) {};
          \node [tmon,red] at (0,-2.5) {};
        \end{scope}
        \draw [bound] (-7,-7) rectangle (7,7);
      \end{tikzpicture}
    }
  \]
  \afterimage
  
  A \emph{monad map} between monads $\mT_1$ and $\mT_2$ on $\oS$ is a
  2-cell $\phi \colon \mT_1 \Rightarrow \mT_2$
  \vspace{-5pt}
  \[\preq
    \mathclap{
      \begin{tikzpicture}
        \node [ob] (s) at (-7,0) {$\oS$};
        \node [ob] (t) at (7,0) {$\oS$};
        \draw [a] (s) .. controls +(5,3) and +(-5,3) .. node[arr,above] {$\mT_1$} (t);
        \draw [a] (s) .. controls +(5,-3) and +(-5,-3) .. node[arr,below] {$\mT_2$} (t);
        \node [cell] at (0,0) {$\phi$};
      \end{tikzpicture}
    }
    \qqq
    \mathclap{
      \begin{tikzpicture}
        \begin{scope}
          \clip [boundc] (-6,-6) rectangle (6,6);
          \fill [bg] (-6,-6) rectangle (6,6);
          \draw [omon,red] (0,6) -- (0,0);
          \draw [omon,blue] (0,-6) -- (0,0);
          \node [thom=purp] at (0,0) {};
        \end{scope}
        \draw [bound] (-6,-6) rectangle (6,6);
        \node [label,purp,left=4] at (0,0) {$\phi$};
        \node [label,red,above] at (0,6) {$\mT_1$};
        \node [label,blue,below] at (0,-6) {$\mT_2$};
      \end{tikzpicture}
    }
    \vspace{-5pt}
  \]
  satisfying
  
  \[\preq
    \mathclap{
      \begin{tikzpicture}[xscale=.875]
        \path (0,10) -- (0,-10);
        \node [ob] (s) at (-10,0) {$\oS$};
        \node [ob] (m) at (0,7) {$\oS$};
        \node [ob] (t) at (10,0) {$\oS$};
        \draw [a] (s) -- node[arr,br] {$\mT_2$} (m);
        \draw [a] (m) -- node[arr,bl] {$\mT_2$} (t);
        \draw [a] (s) -- node[arr,below] {$\mT_2$} (t);
        \draw [a] (s) .. controls +(-1,6) and +(-6,2) .. node[arr,al] {$\mT_1$} (m);
        \draw [a] (m) .. controls +(6,2) and +(1,6) .. node[arr,ar] {$\mT_1$} (t);
        \node [cell] at (0,2.5) {$\mu$};
        \node [cell] at (-6.5,5) {$\phi$};
        \node [cell] at (6.5,5) {$\phi$};
      \end{tikzpicture}
      \eqq
      \begin{tikzpicture}[xscale=.875]
        \path (0,10) -- (0,-10);
        \node [ob] (s) at (-10,0) {$\oS$};
        \node [ob] (m) at (0,7) {$\oS$};
        \node [ob] (t) at (10,0) {$\oS$};
        \draw [a] (s) -- node[arr,al] {$\mT_1$} (m);
        \draw [a] (m) -- node[arr,ar] {$\mT_1$} (t);
        \draw [a] (s) -- node[arr,above=-2,pos=.3] {$\mT_1$} (t);
        \draw [a] (s) .. controls +(6,-6) and +(-6,-6) .. node[arr,below] {$\mT_2$} (t);
        \node [cell] at (0,2.5) {$\mu$};
        \node [cell] at (0,-2.5) {$\phi$};
      \end{tikzpicture}
    }
    \qqq
    \mathclap{
      \begin{tikzpicture}
        \begin{scope}
          \clip [boundc] (-7,-7) rectangle (7,7);
          \fill [bg] (-7,-7) rectangle (7,7);
          \draw [omon,red] (-2.6,7) -- (-2.6,2.5);
          \draw [omon,red] (2.6,7) -- (2.6,2.5);
          \draw [omon,blue] (-2.6,2.5) .. controls +(0,-3.5) and +(-1,1) .. (0,-2.5);
          \draw [omon,blue] (2.6,2.5) .. controls +(0,-3.5) and +(1,1) .. (0,-2.5);
          \draw [omon,blue] (0,-7) -- (0,-2.5);
          \node [thom=purp] at (2.6,2.5) {};
          \node [thom=purp] at (-2.6,2.5) {};
          \node [tmon,blue] at (0,-2.5) {};
        \end{scope}
        \draw [bound] (-7,-7) rectangle (7,7);
      \end{tikzpicture}
      \eqq
      \begin{tikzpicture}
        \begin{scope}
          \clip [boundc] (-7,-7) rectangle (7,7);
          \fill [bg] (-7,-7) rectangle (7,7);
          \draw [omon,red] (-2.6,7) .. controls +(0,-2.5) and +(-1,1) .. (0,2.5);
          \draw [omon,red] (2.6,7) .. controls +(0,-2.5) and +(1,1) .. (0,2.5);
          \draw [omon,red] (0,2.5) -- (0,-2.5);
          \draw [omon,blue] (0,-7) -- (0,-2.5);
          \node [tmon,red] at (0,2.5) {};
          \node [thom=purp] at (0,-2.5) {};
        \end{scope}
        \draw [bound] (-7,-7) rectangle (7,7);
      \end{tikzpicture}
    }
  \]
  \vspace{-10pt}
  \[\preq
    \mathclap{
      \begin{tikzpicture}[xscale=.875]
        \path (0,8) -- (0,-10);
        \node [ob] (s) at (-10,0) {$\oS$};
        \node [ob] (t) at (10,0) {$\oS$};
        \draw [a] (s) -- node[arr,below] {$\mT_2$} (t);
        \draw [eq] (s) .. controls +(6,8) and +(-6,8) .. (t);
        \node [cell] at (0,2.5) {$\eta$};
      \end{tikzpicture}
      \eqq
      \begin{tikzpicture}[xscale=.875]
        \path (0,8) -- (0,-10);
        \node [ob] (s) at (-10,0) {$\oS$};
        \node [ob] (t) at (10,0) {$\oS$};
        \draw [a] (s) -- node[arr,above=-2,pos=.3] {$\mT_1$} (t);
        \draw [eq] (s) .. controls +(6,8) and +(-6,8) .. (t);
        \draw [a] (s) .. controls +(6,-6) and +(-6,-6) .. node[arr,below] {$\mT_2$} (t);
        \node [cell] at (0,2.5) {$\eta$};
        \node [cell] at (0,-2.5) {$\phi$};
      \end{tikzpicture}
    }
    \qqq
    \mathclap{
      \begin{tikzpicture}
        \begin{scope}
          \clip [boundc] (-7,-7) rectangle (7,7);
          \fill [bg] (-7,-7) rectangle (7,7);
          \draw [omon,blue] (0,-7) -- (0,0);
          \node [tmon,blue] at (0,0) {};
        \end{scope}
        \draw [bound] (-7,-7) rectangle (7,7);
      \end{tikzpicture}
      \eqq
      \begin{tikzpicture}
        \begin{scope}
          \clip [boundc] (-7,-7) rectangle (7,7);
          \fill [bg] (-7,-7) rectangle (7,7);
          \draw [omon,red] (0,2.5) -- (0,-2.5);
          \draw [omon,blue] (0,-7) -- (0,-2.5);
          \node [tmon,red] at (0,2.5) {};
          \node [thom=purp] at (0,-2.5) {};
        \end{scope}
        \draw [bound] (-7,-7) rectangle (7,7);
      \end{tikzpicture}
    }
  \]
  \afterimage
\end{definition}

\begin{convention}
  In string diagrams, we write the unit and multiplication 2-cells as
  unmarked nodes the same color as the string denoting the 1-cell
  $\mT$, as in \cite{willerton}. This is to suggest the associativity
  and unit laws by topological deformation, as if $\mT$ were a fluid
  between surrounding regions of $\oS$. More generally we may write
  $n$-ary multiplication $\mu \colon \mT^n \Rightarrow \mT$ (obtained
  by repeatedly composing the ordinary multiplication) as an unmarked
  node of the same color.

  \[
    \begin{tikzpicture}
      \begin{scope}
        \clip [boundc] (-7,-7) rectangle (7,7);
        \fill [bg] (-7,-7) rectangle (7,7);
        \draw [omon,red] (-4,7) .. controls +(0,-5) and +(-1,1) .. (0,0);
        \draw [omon,red] (0,7) -- (0,0);
        \draw [omon,red] (4,7) .. controls +(0,-5) and +(1,1) .. (0,0);
        \draw [omon,red] (0,-7) -- (0,0);
        \node [tmon,red] at (0,0) {};
      \end{scope}
      \draw [bound] (-7,-7) rectangle (7,7);
    \end{tikzpicture}
    \eqquad\coloneqq\eqquad
    \begin{tikzpicture}
      \begin{scope}
        \clip [boundc] (-7,-7) rectangle (7,7);
        \fill [bg] (-7,-7) rectangle (7,7);
        \draw [omon,red] (-4,7) .. controls +(0,-2.5) and +(-1,1) .. (-2,2.5);
        \draw [omon,red] (0,7) .. controls +(0,-2.5) and +(1,1) .. (-2,2.5);
        \draw [omon,red] (4,7) .. controls +(0,-5.5) and +(1,1) .. (0,-2.5);
        \draw [omon,red] (-2,2.5) .. controls +(0,-2.5) and +(-1,1) .. (0,-2.5);
        \draw [omon,red] (0,-7) -- (0,-2.5);
        \node [tmon,red] at (-2,2.5) {};
        \node [tmon,red] at (0,-2.5) {};
      \end{scope}
      \draw [bound] (-7,-7) rectangle (7,7);
    \end{tikzpicture}
    \eqq
    \begin{tikzpicture}
      \begin{scope}
        \clip [boundc] (-7,-7) rectangle (7,7);
        \fill [bg] (-7,-7) rectangle (7,7);
        \draw [omon,red] (-4,7) .. controls +(0,-5.5) and +(-1,1) .. (0,-2.5);
        \draw [omon,red] (0,7) .. controls +(0,-2.5) and +(-1,1) .. (2,2.5);
        \draw [omon,red] (4,7) .. controls +(0,-2.5) and +(1,1) .. (2,2.5);
        \draw [omon,red] (2,2.5) .. controls +(0,-2.5) and +(1,1) .. (0,-2.5);
        \draw [omon,red] (0,-7) -- (0,-2.5);
        \node [tmon,red] at (2,2.5) {};
        \node [tmon,red] at (0,-2.5) {};
      \end{scope}
      \draw [bound] (-7,-7) rectangle (7,7);
    \end{tikzpicture}
  \]
  \afterimage
\end{convention}

\footnotetext{The definitions of monad and monad map make sense more
  generally in a \emph{virtual 2-category} (following the terminology
  scheme of~\cite{cruttwell-shulman}), which is like a 2-category but
  where 1-cells cannot compose and each 2-cell has a path of
  compatible 1-cells as input but only one 1-cell as output. A monad
  can be defined more concisely as a functor from the terminal virtual
  2-category.}

\pagebreak
\chapter{Modules}

\begin{definition}\label{def:module}
  Let $\mT$ be a monad on $\oS$. A (right)
  \emph{$\mT$-module}\footnote{The definitions of $\mT$-module and
    $\mT$-module map make sense in a virtual 2-category. A module over
    a monad can be defined more concisely as a functor from the
    terminal \emph{virtual actegory}: a virtual 2-category with two
    objects $\oX$ and $\oS$ and no 1-cells $\oX \to \oX$ or
    $\oS \to \oX$.} is a 1-cell
  \[\preq
    \mathclap{\bM\colon \oX \to \oS}
    \qqq
    \mathclap{
      \begin{tikzpicture}[yscale=.5]
        \begin{scope}
          \clip [boundc] (-6,-6) rectangle (6,6);
          \fill [bg] (-6,-6) rectangle (6,6);
          \fill [bgd] (-6,-6) rectangle (0,6);
          \draw [omods=red] (0,6) -- (0,-6);
        \end{scope}
        \draw [bound] (-6,-6) rectangle (6,6);
        \node [label] at (-3.125,0) {$\oX$};
        \node [label] at (3,0) {$\oS$};
        \node [label,above] at (0,6) {$\bM$};
        \node [label,below] at (0,-6) {$\phantom{\bM}$};
      \end{tikzpicture}
    }
  \]
  equipped with a 2-cell
  $\rho \colon \bM \then \mT \Rightarrow \bM$
  \[\preq
    \mathclap{
      \begin{tikzpicture}
        \node [ob] (s) at (-7,0) {$\oX$};
        \node [ob] (m) at (0,7) {$\oS$};
        \node [ob] (t) at (7,0) {$\oS$};
        \draw [a] (s) -- node[arr,al] {$\bM$} (m);
        \draw [a] (m) -- node[arr,ar] {$\mT$} (t);
        \draw [a] (s) -- node[arr,below] {$\bM$} (t);
        \node [cell] at (0,2.5) {$\rho$};
      \end{tikzpicture}
    }
    \qqq
    \mathclap{
      \begin{tikzpicture}
        \begin{scope}
          \clip [boundc] (-6,-6) rectangle (6,6);
          \fill [bg] (-6,-6) rectangle (6,6);
          \renewcommand{\mypath}{(-2.6,6) .. controls +(0,-3.5) and +(-1,1) .. (0,0) -- (0,-6)}
          \fill [bgd] \mypath -- (-6,-6) -- (-6,6) -- cycle;
          \draw [omon,red] (2.6,6) .. controls +(0,-3.5) and +(1,1) .. (0,0);
          \draw [omods=red] \mypath;
        \end{scope}
        \draw [bound] (-6,-6) rectangle (6,6);
      \end{tikzpicture}
    }
  \]
  satisfying the associativity and unit laws
  
  \[\preq
    \mathclap{
      \begin{tikzpicture}[xscale=.75,yscale=1.125]
        \node [ob] (s) at (-14,0) {$\oX$};
        \node [ob] (mm) at (-9,7) {$\oS$};
        \node [ob] (m) at (2,7) {$\oS$};
        \node [ob] (t) at (7,0) {$\oS$};
        \draw [a] (s) -- node[arr,al] {$\bM$} (mm);
        \draw [a] (mm) -- node[arr,above] {$\mT$} (m);
        \draw [a] (s) -- node[arr,br] {$\bM$} (m);
        \draw [a] (m) -- node[arr,ar] {$\mT$} (t);
        \draw [a] (s) -- node[arr,below] {$\bM$} (t);
        \node [cell] at (.5,2.5) {$\rho$};
        \node [cell] at (-7,5) {$\rho$};
      \end{tikzpicture}
      \eqq
      \begin{tikzpicture}[xscale=.75,yscale=1.125]
        \node [ob] (s) at (-7,0) {$\oX$};
        \node [ob] (m) at (-2,7) {$\oS$};
        \node [ob] (mm) at (9,7) {$\oS$};
        \node [ob] (t) at (14,0) {$\oS$};
        \draw [a] (m) -- node[arr,above] {$\mT$} (mm);
        \draw [a] (mm) -- node[arr,ar] {$\mT$} (t);
        \draw [a] (s) -- node[arr,al] {$\bM$} (m);
        \draw [a] (m) -- node[arr,bl] {$\mT$} (t);
        \draw [a] (s) -- node[arr,below] {$\bM$} (t);
        \node [cell] at (-.5,2.5) {$\rho$};
        \node [cell] at (7,5) {$\mu$};
      \end{tikzpicture}
    }
    \qqq
    \mathclap{
      \begin{tikzpicture}
        \begin{scope}
          \clip [boundc] (-7,-7) rectangle (7,7);
          \fill [bg] (-7,-7) rectangle (7,7);
          \renewcommand{\mypath}{(-4,7) .. controls +(0,-2.5) and +(-1,1) .. (-2,2.5) .. controls +(0,-2.5) and +(-1,1) .. (0,-2.5) -- (0,-7)}
          \fill [bgd] \mypath -- (-7,-7) -- (-7,7) -- cycle;
          \draw [omon,red] (0,7) .. controls +(0,-2.5) and +(1,1) .. (-2,2.5);
          \draw [omon,red] (4,7) .. controls +(0,-5.5) and +(1,1) .. (0,-2.5);
          \draw [omods=red] \mypath;
        \end{scope}
        \draw [bound] (-7,-7) rectangle (7,7);
      \end{tikzpicture}
      \eqq
      \begin{tikzpicture}
        \begin{scope}
          \clip [boundc] (-7,-7) rectangle (7,7);
          \fill [bg] (-7,-7) rectangle (7,7);
          \renewcommand{\mypath}{(-4,7) .. controls +(0,-5.5) and +(-1,1) .. (0,-2.5) -- (0,-7)}
          \fill [bgd] \mypath -- (-7,-7) -- (-7,7) -- cycle;
          \draw [omon,red] (0,7) .. controls +(0,-2.5) and +(-1,1) .. (2,2.5);
          \draw [omon,red] (4,7) .. controls +(0,-2.5) and +(1,1) .. (2,2.5);
          \draw [omon,red] (2,2.5) .. controls +(0,-2.5) and +(1,1) .. (0,-2.5);
          \node [tmon,red] at (2,2.5) {};
          \draw [omods=red] \mypath;
        \end{scope}
        \draw [bound] (-7,-7) rectangle (7,7);
      \end{tikzpicture}
    }
  \]
  \vspace{-5pt}
  \[\preq
    \mathclap{
      \begin{tikzpicture}[xscale=.75,yscale=1.125]
        \path (0,11) -- (0,-4);
        \node [ob] (s) at (-14,0) {$\oX$};
        \node [ob] (t) at (7,0) {$\oS$};
        \draw [a] (s) .. controls +(4,8.5) and +(-4,8.5) .. node[arr,above] {$\bM$} (t);
        \draw [a] (s) -- node[arr,below] {$\bM$} (t);
        \node [cell] at (-3.5,3) {$=$};
      \end{tikzpicture}
      \eqq
      \begin{tikzpicture}[xscale=.75,yscale=1.125]
        \path (0,11) -- (0,-4);
        \node [ob] (s) at (-7,0) {$\oX$};
        \node [ob] (m) at (-2,7) {$\oS$};
        \node [ob] (t) at (14,0) {$\oS$};
        \draw [a] (s) -- node[arr,al] {$\bM$} (m);
        \draw [a] (m) -- node[arr,bl] {$\mT$} (t);
        \draw [a] (s) -- node[arr,below] {$\bM$} (t);
        \draw [eq] (m) .. controls +(8,1.5) and +(-1.5,8) .. (t);
        \node [cell] at (-.5,2.5) {$\rho$};
        \node [cell] at (7,5) {$\eta$};
      \end{tikzpicture}
    }
    \qqq
    \mathclap{
      \begin{tikzpicture}
        \begin{scope}
          \clip [boundc] (-7,-7) rectangle (7,7);
          \fill [bg] (-7,-7) rectangle (7,7);
          \fill [bgd] (-7,-7) rectangle (0,7);
          \draw [omods=red] (0,7) -- (0,-7);
        \end{scope}
        \draw [bound] (-7,-7) rectangle (7,7);
      \end{tikzpicture}
      \eqq
      \begin{tikzpicture}
        \begin{scope}
          \clip [boundc] (-7,-7) rectangle (7,7);
          \fill [bg] (-7,-7) rectangle (7,7);
          \renewcommand{\mypath}{(-2.6,7) .. controls +(0,-5.5) and +(-1,1) .. (0,-2.5) -- (0,-7)}
          \fill [bgd] \mypath -- (-7,-7) -- (-7,7) -- cycle;
          \draw [omon,red] (2,2.5) .. controls +(0,-2.5) and +(1,1) .. (0,-2.5);
          \node [tmon,red] at (2,2.5) {};
          \draw [omods=red] \mypath;
        \end{scope}
        \draw [bound] (-7,-7) rectangle (7,7);
      \end{tikzpicture}
    }
  \]
  \afterimage
  
  A \emph{$\mT$-module map} between $\mT$-modules
  $\bM_1, \bM_2 \colon \oX \to \oS$ is a 2-cell
  $\phi \colon \bM_1 \Rightarrow \bM_2$
  \[\preq
    \mathclap{
      \begin{tikzpicture}
        \node [ob] (s) at (-7,0) {$\oX$};
        \node [ob] (t) at (7,0) {$\oS$};
        \draw [a] (s) .. controls +(5,3) and +(-5,3) .. node[arr,above] {$\bM_1$} (t);
        \draw [a] (s) .. controls +(5,-3) and +(-5,-3) .. node[arr,below] {$\bM_2$} (t);
        \node [cell] at (0,0) {$\phi$};
      \end{tikzpicture}
    }
    \qqq
    \mathclap{
      \begin{tikzpicture}
        \begin{scope}
          \clip [boundc] (-6,-6) rectangle (6,6);
          \fill [bg] (-6,-6) rectangle (6,6);
          \node [tmhomh,red] at (0,0) {};
          \fill [bgd] (-6,-6) rectangle (0,6);
          \draw [omods=red] (0,6) -- (0,0);
          \draw [omods=red,white] (0,0) -- (0,-6);
          \node [tmhom] at (0,0) {};
        \end{scope}
        \draw [bound] (-6,-6) rectangle (6,6);
        \node [label,left=3.5] at (0,0) {$\phi$};
        \node [label,above] at (0,6) {$\bM_1$};
        \node [label,below] at (0,-6) {$\bM_2$};
      \end{tikzpicture}
    }
  \]
  satisfying
  
  \[\preq
    \mathclap{
      \begin{tikzpicture}[xscale=.875]
        \path (0,10) -- (0,-10);
        \node [ob] (s) at (-10,0) {$\oX$};
        \node [ob] (m) at (0,7) {$\oS$};
        \node [ob] (t) at (10,0) {$\oS$};
        \draw [a] (s) -- node[arr,brcl,pos=.33] {$\bM_2$} (m);
        \draw [a] (m) -- node[arr,ar] {$\mT$} (t);
        \draw [a] (s) -- node[arr,below] {$\bM_2$} (t);
        \draw [a] (s) .. controls +(-1,6) and +(-6,2) .. node[arr,al] {$\bM_1$} (m);
        \node [cell] at (0,2.5) {$\rho$};
        \node [cell] at (-6.5,5) {$\phi$};
      \end{tikzpicture}
      \eqq
      \begin{tikzpicture}[xscale=.875]
        \path (0,10) -- (0,-10);
        \node [ob] (s) at (-10,0) {$\oX$};
        \node [ob] (m) at (0,7) {$\oS$};
        \node [ob] (t) at (10,0) {$\oS$};
        \draw [a] (s) -- node[arr,al] {$\bM_1$} (m);
        \draw [a] (m) -- node[arr,ar] {$\mT$} (t);
        \draw [a] (s) -- node[arr,above=-2,pos=.3] {$\bM_1$} (t);
        \draw [a] (s) .. controls +(6,-6) and +(-6,-6) .. node[arr,below] {$\bM_2$} (t);
        \node [cell] at (0,2.5) {$\rho$};
        \node [cell] at (0,-2.5) {$\phi$};
      \end{tikzpicture}
    }
    \qqq
    \mathclap{
      \begin{tikzpicture}
        \begin{scope}
          \clip [boundc] (-7,-7) rectangle (7,7);
          \fill [bg] (-7,-7) rectangle (7,7);
          \renewcommand{\mypatha}{(-2.6,7) -- (-2.6,2.5)}
          \renewcommand{\mypathb}{(-2.6,2.5) .. controls +(0,-3.5) and +(-1,1) .. (0,-2.5) --  (0,-7)}
          \node [tmhomh,red] at (-2.6,2.5) {};
          \fill [bgd] \mypatha -- \mypathb -- (-7,-7) -- (-7,7) -- cycle;
          \draw [omon,red] (2.6,7) -- (2.6,2.5) .. controls +(0,-3.5) and +(1,1) .. (0,-2.5);
          \draw [omods=red] \mypatha;
          \draw [omods=red,white] \mypathb;
          \node [tmhom] at (-2.6,2.5) {};
        \end{scope}
        \draw [bound] (-7,-7) rectangle (7,7);
      \end{tikzpicture}
      \eqq
      \begin{tikzpicture}
        \begin{scope}
          \clip [boundc] (-7,-7) rectangle (7,7);
          \fill [bg] (-7,-7) rectangle (7,7);
          \renewcommand{\mypatha}{(-2.6,7) .. controls +(0,-2.5) and +(-1,1) .. (0,2.5) -- (0,-2.5)}
          \renewcommand{\mypathb}{(0,-2.5) -- (0,-7)}
          \node [tmhomh,red] at (0,-2.5) {};
          \fill [bgd] \mypatha -- \mypathb -- (-7,-7) -- (-7,7) -- cycle;
          \draw [omon,red] (2.6,7) .. controls +(0,-2.5) and +(1,1) .. (0,2.5);
          \draw [omods=red] \mypatha;
          \draw [omods=red,white] \mypathb;
          \node [tmhom] at (0,-2.5) {};
        \end{scope}
        \draw [bound] (-7,-7) rectangle (7,7);
      \end{tikzpicture}
    }
  \]
\end{definition}

\pagebreak

\begin{convention}
  In string diagrams, we write the 1-cell $\bM$ as a string coated on
  the right by a film the same color as $\mT$, and we write the 2-cell
  $\rho$ as an unmarked node merging $\mT$ with the film. This is to
  suggest the module laws by topological deformation.


  
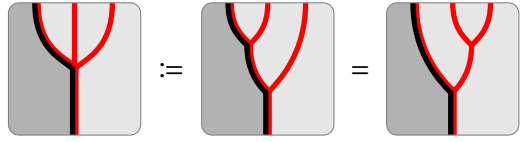
\[
    \begin{tikzpicture}
      \begin{scope}
        \clip [boundc] (-7,-7) rectangle (7,7);
        \fill [bg] (-7,-7) rectangle (7,7);
        \renewcommand{\mypath}{(-4,7) .. controls +(0,-5) and +(-1,1) .. (0,0) -- (0,-7)}
        \fill [bgd] \mypath -- (-7,-7) -- (-7,7) -- cycle;
        \draw [omon,red] (0,7) -- (0,0);
        \draw [omon,red] (4,7) .. controls +(0,-5) and +(1,1) .. (0,0);
        \draw [omods=red] \mypath;
      \end{scope}
      \draw [bound] (-7,-7) rectangle (7,7);
    \end{tikzpicture}
    \eqquad\coloneqq\eqquad
    \begin{tikzpicture}
      \begin{scope}
        \clip [boundc] (-7,-7) rectangle (7,7);
        \fill [bg] (-7,-7) rectangle (7,7);
        \renewcommand{\mypath}{(-4,7) .. controls +(0,-2.5) and +(-1,1) .. (-2,2.5) .. controls +(0,-2.5) and +(-1,1) .. (0,-2.5) -- (0,-7)}
        \fill [bgd] \mypath -- (-7,-7) -- (-7,7) -- cycle;
        \draw [omon,red] (0,7) .. controls +(0,-2.5) and +(1,1) .. (-2,2.5);
        \draw [omon,red] (4,7) .. controls +(0,-5.5) and +(1,1) .. (0,-2.5);
        \draw [omods=red] \mypath;
      \end{scope}
      \draw [bound] (-7,-7) rectangle (7,7);
    \end{tikzpicture}
    \eqq
    \begin{tikzpicture}
      \begin{scope}
        \clip [boundc] (-7,-7) rectangle (7,7);
        \fill [bg] (-7,-7) rectangle (7,7);
        \renewcommand{\mypath}{(-4,7) .. controls +(0,-5.5) and +(-1,1) .. (0,-2.5) -- (0,-7)}
        \fill [bgd] \mypath -- (-7,-7) -- (-7,7) -- cycle;
        \draw [omon,red] (0,7) .. controls +(0,-2.5) and +(-1,1) .. (2,2.5);
        \draw [omon,red] (4,7) .. controls +(0,-2.5) and +(1,1) .. (2,2.5);
        \draw [omon,red] (2,2.5) .. controls +(0,-2.5) and +(1,1) .. (0,-2.5);
        \draw [omon,red] (0,-7) -- (0,-2.5);
        \node [tmon,red] at (2,2.5) {};
        \draw [omods=red] \mypath;
      \end{scope}
      \draw [bound] (-7,-7) rectangle (7,7);
    \end{tikzpicture}
  \]
  \afterimage

  Likewise we write a $\mT$-module map as a node coated on the right
  by the same film, to suggest the $\mT$-module map law by topological
  deformation.
\end{convention}

\begin{definition}\label{def:em}
  An \emph{algebra object} (a.k.a.\ \emph{Eilenberg-Moore object}) of
  a monad $\mT$ on $\oS$ is the domain $\AlgTS$ of a universal
  $\mT$-module\footnote{A $\mT$-module may also be viewed as a
    \emph{lax cone} to $\mT$, viewing $\mT$ as a functor from the free
    2-category containing a monad. A universal $\mT$-module is then a
    \emph{lax limit} of this functor~\cite{street:limits}.}  (a.k.a.\
  \emph{monadic} 1-cell) $\mcarT$. Explicitly, this means $\mcarT$ is
  a $\mT$-module \vspace{-5pt}
  \[\preq
    \mathclap{\mcarT \colon \AlgTS \to \oS}
    \qqq
    \mathclap{
      \begin{tikzpicture}[yscale=.5]
        \begin{scope}
          \clip [boundc] (-6,-6) rectangle (6,6);
          \fill [bg] (-6,-6) rectangle (6,6);
          \fill [em=red] (-6,-6) rectangle (0,6);
          \draw [oem,red] (0,6) -- (0,-6);
        \end{scope}
        \draw [bound] (-6,-6) rectangle (6,6);
        \node [label] at (-3,0) {$\AlgTS$};
        \node [label] at (3,0) {$\oS$};
        \node [label,red,above] at (0,6) {$\mcarT$};
        \node [label,below] at (0,-6) {$\phantom{\mcarT}$};
      \end{tikzpicture}
    }
  \]
  \vspace{-12.5pt}
  \[\preq
    \mathclap{
      \begin{tikzpicture}
        \node [ob] (s) at (-7,0) {$\AlgTS$};
        \node [ob] (m) at (0,7) {$\oS$};
        \node [ob] (t) at (7,0) {$\oS$};
        \draw [a] (s) -- node[arr,al] {$\mcarT$} (m);
        \draw [a] (m) -- node[arr,ar] {$\mT$} (t);
        \draw [a] (s) -- node[arr,below] {$\mcarT$} (t);
        \node [cell] at (0,2.5) {$\mcarmodT$};
      \end{tikzpicture}
    }
    \qqq
    \mathclap{
      \begin{tikzpicture}
        \begin{scope}
          \clip [boundc] (-6,-6) rectangle (6,6);
          \fill [bg] (-6,-6) rectangle (6,6);
          \renewcommand{\mypath}{(-2.6,6) .. controls +(0,-3.5) and +(-1,1) .. (0,0) -- (0,-6)}
          \fill [em=red] \mypath -- (-6,-6) -- (-6,6) -- cycle;
          \draw [omon,red] (2.6,6) .. controls +(0,-3.5) and +(1,1) .. (0,0);
          \draw [oem,red] \mypath;
        \end{scope}
        \draw [bound] (-6,-6) rectangle (6,6);
      \end{tikzpicture}
    }
  \]
  such that
  \begin{itemize}
  \item
    for any $\mT$-module $\bM$
    \vspace{-10pt}
    \[\preq
      \hspace{-\leftmargin}
      \mathclap{\bM\colon \oX \to \oS}
      \qqq
      \mathclap{
        \begin{tikzpicture}[yscale=.5]
          \begin{scope}
            \clip [boundc] (-6,-6) rectangle (6,6);
            \fill [bg] (-6,-6) rectangle (6,6);
            \fill [bgd] (-6,-6) rectangle (0,6);
            \draw [omods=red] (0,6) -- (0,-6);
          \end{scope}
          \draw [bound] (-6,-6) rectangle (6,6);
          \node [label] at (-3,0) {$\oX$};
          \node [label,above] at (0,6) {$\bM$};
          \node [label,below] at (0,-6) {$\phantom{\bM}$};
        \end{tikzpicture}
      }
    \]
    \vspace{-10pt}
    \[\preq
      \hspace{-\leftmargin}
      \mathclap{
        \begin{tikzpicture}
          \node [ob] (s) at (-7,0) {$\oX$};
          \node [ob] (m) at (0,7) {$\oS$};
          \node [ob] (t) at (7,0) {$\oS$};
          \draw [a] (s) -- node[arr,al] {$\bM$} (m);
          \draw [a] (m) -- node[arr,ar] {$\mT$} (t);
          \draw [a] (s) -- node[arr,below] {$\bM$} (t);
          \node [cell] at (0,2.5) {$\rho$};
        \end{tikzpicture}
      }
      \qqq
      \mathclap{
        \begin{tikzpicture}
          \begin{scope}
            \clip [boundc] (-6,-6) rectangle (6,6);
            \fill [bg] (-6,-6) rectangle (6,6);
            \renewcommand{\mypath}{(-2.6,6) .. controls +(0,-3.5) and +(-1,1) .. (0,0) -- (0,-6)}
            \fill [bgd] \mypath -- (-6,-6) -- (-6,6) -- cycle;
            \draw [omon,red] (2.6,6) .. controls +(0,-3.5) and +(1,1) .. (0,0);
            \draw [omods=red] \mypath;
          \end{scope}
          \draw [bound] (-6,-6) rectangle (6,6);
        \end{tikzpicture}
      }
      \afterimage
    \]
    there is a unique 1-cell $\lift{\bM}$
    \vspace{-10pt}
    \[\preq
      \hspace{-\leftmargin}
      \mathclap{\lift{\bM} \colon \oX \to \AlgTS}
      \qqq
      \mathclap{
        \begin{tikzpicture}[yscale=.5]
          \begin{scope}
            \clip [boundc] (-6,-6) rectangle (6,6);
            \fill [bg] (-6,-6) rectangle (6,6);
            \fill [bgd] (-6,-6) rectangle (0,6);
            \fill [em=red] (0,-6) rectangle (6,6);
            \draw [o] (0,6) -- (0,-6);
          \end{scope}
          \draw [bound] (-6,-6) rectangle (6,6);
          \node [label,above] at (0,6) {$\lift{\bM}$};
          \node [label,below] at (0,-6) {$\phantom{\lift{\bM}}$};
        \end{tikzpicture}
      }
      \vspace{-12.5pt}
    \]
    factoring the $\mT$-module $\bM$ through $\mcarT$
    \[\preq
      \hspace{-\leftmargin}
      \mathclap{
        \begin{tikzpicture}
          \node [ob] (s) at (-7,0) {$\oX$};
          \node [ob] (m) at (0,-7) {$\AlgTS$};
          \node [ob] (t) at (7,0) {$\oS$};
          \draw [a] (s) -- node[arr,bl] {$\lift{\bM}$} (m);
          \draw [a] (m) -- node[arr,br] {$\mcarT$} (t);
          \draw [a] (s) -- node[arr,above] {$\bM$} (t);
          \node at (0,-2.5) {$=$};
        \end{tikzpicture}
      }
      \qqq
      \mathclap{
        \begin{tikzpicture}[yscale=.5]
          \begin{scope}
            \clip [boundc] (-6,-6) rectangle (6,6);
            \fill [bg] (-6,-6) rectangle (6,6);
            \fill [bgd] (-6,-6) rectangle (0,6);
            \draw [omods=red] (0,6) -- (0,-6);
          \end{scope}
          \draw [bound] (-6,-6) rectangle (6,6);
        \end{tikzpicture}
        \eqq
        \begin{tikzpicture}[yscale=.5]
          \begin{scope}
            \clip [boundc] (-8,-6) rectangle (8,6);
            \fill [bg] (-8,-6) rectangle (8,6);
            \fill [bgd] (-8,-6) rectangle (-3,6);
            \fill [em=red] (-3,-6) rectangle (3,6);
            \draw [o] (-3,6) -- (-3,-6);
            \draw [oem,red] (3,6) -- (3,-6);
          \end{scope}
          \draw [bound] (-8,-6) rectangle (8,6);
        \end{tikzpicture}
      }
    \]
    \vspace{-7.5pt}
    \[\preq
      \hspace{-\leftmargin}
      \mathclap{
        \begin{tikzpicture}
          \node [ob] (s) at (-7,0) {$\oX$};
          \node [ob] (m) at (0,7) {$\oS$};
          \node [ob] (t) at (7,0) {$\oS$};
          \draw [a] (s) -- node[arr,al] {$\bM$} (m);
          \draw [a] (m) -- node[arr,ar] {$\mT$} (t);
          \draw [a] (s) -- node[arr,below] {$\bM$} (t);
          \node [cell] at (0,2.5) {$\rho$};
        \end{tikzpicture}
        \eqq
        \begin{tikzpicture}[xscale=.75]
          \node [ob] (s) at (-7,0) {$\AlgTS$};
          \node [ob] (m) at (0,7) {$\oS$};
          \node [ob] (t) at (7,0) {$\oS$};
          \draw [a] (s) -- node[arr,al] {$\mcarT$} (m);
          \draw [a] (m) -- node[arr,ar] {$\mT$} (t);
          \draw [a] (s) -- node[arr,below] {$\mcarT$} (t);
          \node [cell] at (0,2.5) {$\mcarmodT$};
          \node [ob] (ss) at (-21,0) {$\oX$};
          \draw [a] (ss) -- node[arr,above] {$\lift{\bM}$} (s);
        \end{tikzpicture}
      }
      \qqq
      \mathclap{
        \begin{tikzpicture}
          \begin{scope}
            \clip [boundc] (-6,-6) rectangle (6,6);
            \fill [bg] (-6,-6) rectangle (6,6);
            \renewcommand{\mypath}{(-2.6,6) .. controls +(0,-3.5) and +(-1,1) .. (0,0) -- (0,-6)}
            \fill [bgd] \mypath -- (-6,-6) -- (-6,6) -- cycle;
            \draw [omon,red] (2.6,6) .. controls +(0,-3.5) and +(1,1) .. (0,0);
            \draw [omods=red] \mypath;
          \end{scope}
          \draw [bound] (-6,-6) rectangle (6,6);
        \end{tikzpicture}
        \eqq
        \begin{tikzpicture}
          \begin{scope}
            \clip [boundc] (-10,-6) rectangle (6,6);
            \fill [bg] (-10,-6) rectangle (6,6);
            \fill [bgd] (-10,-6) rectangle (-5,6);
            \renewcommand{\mypath}{(-1.6,6) .. controls +(0,-3.5) and +(-1,1) .. (1,0) -- (1,-6)}
            \fill [em=red] \mypath -- (-5,-6) -- (-5,6) -- cycle;
            \draw [omon,red] (3.6,6) .. controls +(0,-3.5) and +(1,1) .. (1,0);
            \draw [oem,red] \mypath;
            \draw [o] (-5,6) -- (-5,-6);
          \end{scope}
          \draw [bound] (-10,-6) rectangle (6,6);
        \end{tikzpicture}
      }
    \]
    
  \item
    for any $\mT$-module map $\phi \colon \bM_1 \Rightarrow \bM_2$
    \[\preq
      \hspace{-\leftmargin}
      \mathclap{
        \begin{tikzpicture}
          \node [ob] (s) at (-7,0) {$\oX$};
          \node [ob] (t) at (7,0) {$\oS$};
          \draw [a] (s) .. controls +(5,3) and +(-5,3) .. node[arr,above] {$\bM_1$} (t);
          \draw [a] (s) .. controls +(5,-3) and +(-5,-3) .. node[arr,below] {$\bM_2$} (t);
          \node [cell] at (0,0) {$\phi$};
        \end{tikzpicture}
      }
      \qqq
      \mathclap{
        \begin{tikzpicture}
          \begin{scope}
            \clip [boundc] (-6,-6) rectangle (6,6);
            \fill [bg] (-6,-6) rectangle (6,6);
            \node [tmhomh,red] at (0,0) {};
            \fill [bgd] (-6,-6) rectangle (0,6);
            \draw [omods=red] (0,6) -- (0,0);
            \draw [omods=red,white] (0,0) -- (0,-6);
            \node [tmhom] at (0,0) {};
          \end{scope}
          \draw [bound] (-6,-6) rectangle (6,6);
          \node [label,left=3.5] at (0,0) {$\phi$};
          \node [label,above] at (0,6) {$\bM_1$};
          \node [label,below] at (0,-6) {$\bM_2$};
        \end{tikzpicture}
      }
    \]
    there is a unique 2-cell $\lift{\phi} \colon \lift{\bM_1}
    \Rightarrow \lift{\bM_2}$
    \vspace{-5pt}
    \[\preq
      \hspace{-\leftmargin}
      \mathclap{
        \begin{tikzpicture}
          \node [ob] (s) at (-7,0) {$\oX$};
          \node [ob] (t) at (7,0) {$\AlgTS$};
          \draw [a] (s) .. controls +(5,3) and +(-5,3) .. node[arr,above] {$\lift{\bM_1}$} (t);
          \draw [a] (s) .. controls +(5,-3) and +(-5,-3) .. node[arr,below] {$\lift{\bM_2}$} (t);
          \node [cell] at (0,0) {$\lift{\phi}$};
        \end{tikzpicture}
      }
      \qqq
      \mathclap{
        \begin{tikzpicture}
          \begin{scope}
            \clip [boundc] (-6,-6) rectangle (6,6);
            \fill [bg] (-6,-6) rectangle (6,6);
            \fill [bgd] (-6,-6) rectangle (0,6);
            \fill [em=red] (0,-6) rectangle (6,6);
            \draw [o] (0,6) -- (0,0);
            \draw [o,white] (0,0) -- (0,-6);
            \node [tmhomi] at (0,0) {};
          \end{scope}
          \draw [bound] (-6,-6) rectangle (6,6);
          \node [label,above] at (0,6) {$\lift{\bM_1}$};
          \node [label,below] at (0,-6) {$\lift{\bM_2}$};
          \node [label,left=3.5] at (0,0) {$\lift{\phi}$};
        \end{tikzpicture}
      }
    \]
    factoring $\phi$ through $\mcarmodT$
    \[\preq
      \hspace{-\leftmargin}
      \mathclap{
        \begin{tikzpicture}
          \node [ob] (s) at (-7,0) {$\oX$};
          \node [ob] (t) at (7,0) {$\oS$};
          \draw [a] (s) .. controls +(5,3) and +(-5,3) .. node[arr,above] {$\bM_1$} (t);
          \draw [a] (s) .. controls +(5,-3) and +(-5,-3) .. node[arr,below] {$\bM_2$} (t);
          \node [cell] at (0,0) {$\phi$};
        \end{tikzpicture}
        \eqq
        \begin{tikzpicture}[xscale=.75]
          \node [ob] (s) at (-7,0) {$\oX$};
          \node [ob] (t) at (7,0) {$\AlgTS$};
          \draw [a] (s) .. controls +(5,3) and +(-5,3) .. node[arr,above] {$\lift{\bM_1}$} (t);
          \draw [a] (s) .. controls +(5,-3) and +(-5,-3) .. node[arr,below] {$\lift{\bM_2}$} (t);
          \node [cell] at (0,0) {$\lift{\phi}$};
          \node [ob] (tt) at (21,0) {$\oS$};
          \draw [a] (t) -- node[arr,above] {$\mcarT$} (tt);
        \end{tikzpicture}
      }
      \qqq
      \mathclap{
        \begin{tikzpicture}
          \begin{scope}
            \clip [boundc] (-6,-6) rectangle (6,6);
            \fill [bg] (-6,-6) rectangle (6,6);
            \node [tmhomh,red] at (0,0) {};
            \fill [bgd] (-6,-6) rectangle (0,6);
            \draw [omods=red] (0,6) -- (0,0);
            \draw [omods=red,white] (0,0) -- (0,-6);
            \node [tmhom] at (0,0) {};
          \end{scope}
          \draw [bound] (-6,-6) rectangle (6,6);
        \end{tikzpicture}
        \eqq
        \begin{tikzpicture}
          \begin{scope}
            \clip [boundc] (-8,-6) rectangle (8,6);
            \fill [bg] (-8,-6) rectangle (8,6);
            \fill [bgd] (-8,-6) rectangle (-3,6);
            \fill [em=red] (-3,-6) rectangle (3,6);
            \draw [o] (-3,6) -- (-3,0);
            \draw [o,white] (-3,0) -- (-3,-6);
            \node [tmhomi] at (-3,0) {};
            \draw [oem,red] (3,6) -- (3,-6);
          \end{scope}
          \draw [bound] (-8,-6) rectangle (8,6);
        \end{tikzpicture}
      }
    \]
  \end{itemize}
\end{definition}

\begin{convention}
  In string diagrams, we write the algebra object $\AlgTS$ as a
  $\mT$-colored region. This is to suggest $\AlgTS$ behaves as an
  interior to the channel of fluid $\mT$. Indeed, the following allows
  us to expand the 1-cells and 2-cells of $\mT$ in string diagrams as
  regions of $\AlgTS$.
\end{convention}

\begin{proposition}\label{prop:termresolve}
  Let $\mT$ be a monad on $\oS$ with algebra object $\AlgTS$.
  \begin{enumerate}[label=(\roman*)]
  \item\label{item:resolveadj} There exists a canonical \emph{left
      adjoint} $\mlcarT$ of $\mcarT$. That is, there is a 1-cell
    \[\preq
      \hspace{-\leftmargin}
      \mathclap{\mlcarT \colon \oS \to \AlgTS}
      \qqq
      \mathclap{

      }
      \afterimage
    \]
    and 2-cells (the counit and unit)
    \[\preq
      \hspace{-\leftmargin}
      \mathclap{
        %
      }
      \afterimage
    \]
    satisfying the adjunction laws (a.k.a.\ snake equations)
    \[\preq
      \hspace{-\leftmargin}
      \mathclap{
        %
      }
      \afterimage
    \]
    
  \item\label{item:resolvemon} This adjunction \emph{resolves the
      monad $\mT$}. That is,\vspace{-7.5pt}
    \[\preq
      \hspace{-\leftmargin}
      \mathclap{
        %
      }
      \afterimage
    \]
    
  \item\label{item:resolveterm} This is a \emph{terminal resolution of
      $\mT$}. That is, for any other adjunction resolving $\mT$\vspace{-5pt}
    \[\preq
      \hspace{-\leftmargin}
      \mathclap{
        \fR \colon \oX \to \oS 
        \qquad
        \fL \colon \oS \to \oX
      }
      \qqq
      \mathclap{
        %
      }
    \]
  \end{enumerate}
\end{proposition}

\begin{proof}
  First observe that $\mT$ is itself a $\mT$-module with structure
  2-cell $\rho$ given by the multiplication $\mu$. Thus the universal
  property of $\mcarT$ yields a unique 1-cell
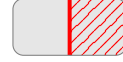
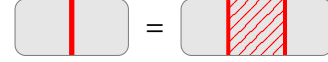
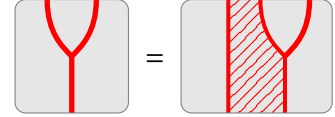
\vspace{-2.5pt}
  \[\preq
    \hspace{-\leftmargin}
    \mathclap{\mlcarT \colon \oS \to \AlgTS} 
    \qqq
    \mathclap{
      %
    }
    \afterimage
    \vspace{-2.5pt}
  \]
  
  Next, observe that $\mcarT \then \mT$ is a $\mT$-module (as is more
  generally the composition of any 1-cell --- here $\mcarT$ --- on the
  left of a $\mT$-module --- here $\mT$). Moreover, observe that
  $\mcarmodT \colon \mcarT \then \mT \Rightarrow \mcarT$ is a
  $\mT$-module map (as is more generally the structure 2-cell of any
  $\mT$-module --- this is the associativity law). Thus the universal
  property of $\mcarT$ yields a unique 2-cell
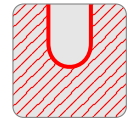
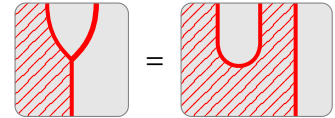
\vspace{-2.5pt}
  \[\preq
    \hspace{-\leftmargin}
    \mathclap{
      %
    }
    \vspace{-2.5pt}
  \]

  \noindent
  Composing $\mlcarT$ on the left yields the equation
  \[\preq
    \hspace{-\leftmargin}
    \mathclap{
      \mathllap{
        %
    }
  \]
  
  Thus we have the desired cells of \labelcref{item:resolveadj} and
  the desired equations of \labelcref{item:resolvemon}. The first
  snake equation of \labelcref{item:resolveadj} is given by the module
  unit law
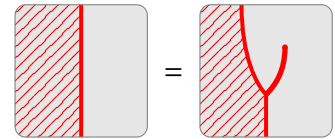
\vspace{-2.25pt}
  \[\preq
    \mathclap{
      %
    }
  \]

  \noindent
  On the other hand, composing $\mcarT$ on the right of the second
  snake equation yields the monad unit law
  
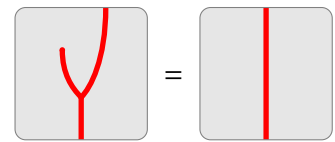
\vspace{-2.25pt}
  \[\preq
    \mathclap{
      %
    }
  \]
  and $\mcarT$ may be cancelled on the right of 2-cells (i.e.\
  $\mcarT$ is \emph{faithful}) because of its universal property that
  module maps factor through it uniquely.

  \pagebreak

  For \labelcref{item:resolveterm}, observe that given any adjunction
  resolving the monad $\mT$, the right adjoint $\fR$ is a $\mT$-module
  with the structure 2-cell defined as
  \[\preq
    \mathclap{
      %
    }
  \]
  Thus the universal property of $\mcarT$ yields a unique 1-cell
  \[\preq
    \hspace{-\leftmargin}
    \mathclap{\lift{\fR} \colon \oS \to \AlgTS}
    \qqq
    \mathclap{
      %
    }
  \]
  Since the adjunction resolves $\mT$ we have
  \[\preq
    \mathclap{
      %
    }
  \]
  and thus by uniqueness of factorings of modules through $\mcarT$ we
  have
  \[\preq
    \hspace{-\leftmargin}
    \mathclap{
      %
    }
  \]
  as desired.
\end{proof}

\pagebreak
\chapter{Morphisms}\label{sec:morphisms}

There are also notions of morphism between monads on different
objects, as well as higher morphisms between these morphisms.

\begin{definition}\label{def:monadone}
  A \emph{monad lax 1-cell}\footnote{The definitions of monad lax
    1-cell, monad 2-cell, and monad specialization make sense in an
    \emph{implicit 2-category}~\cite{fairbanks-shulman}, and using
    paths of 1-cells rather than individual 1-cells. This is
    essentially just to say that the definitions do not involve
    equations between 1-cells.}
  between monads $\mT_1$ on $\oS_1$ and
  $\mT_2$ on $\oS_2$ consists of a 1-cell
  \vspace{-5pt}
  \[\preq
    \mathclap{\fF \colon \oS_1 \to \oS_2}
    \qqq
    \mathclap
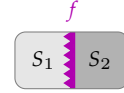
{
      \begin{tikzpicture}[yscale=.5]
        \begin{scope}
          \clip [boundc] (-6,-6) rectangle (6,6);
          \fill [bg] (-6,-6) rectangle (6,6);
          \fill [bgd] (0,-6) rectangle (6,6);
          \draw [olax,purp] (0,7) -- (0,-9);
        \end{scope}
        \draw [bound] (-6,-6) rectangle (6,6);
        \node [label] at (-3,0) {$\oS_1$};
        \node [label] at (3,0) {$\oS_2$};
        \node [label,purp,above] at (0,6) {$\fF$};
        \node [label,below] at (0,-6) {$\phantom{\fF}$};
      \end{tikzpicture}
    }
    \vspace{-5pt}
  \]
  and a 2-cell
  $\chi \colon \fF \then \mT_2 \Rightarrow \mT_1 \then \fF$
  \vspace{-5pt}
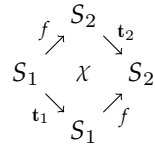
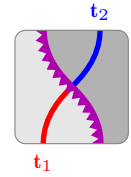
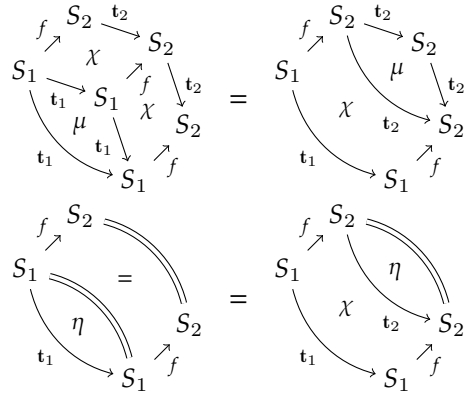
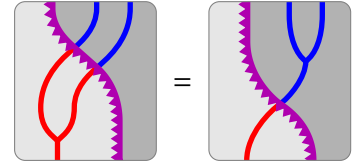
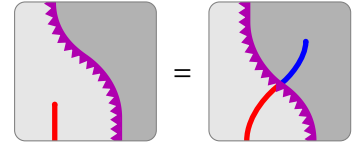
  \[\preq
    \mathclap{
      \begin{tikzpicture}[yscale=.875,xscale=.875]
        \node [ob] (s) at (-7,0) {$\oS_1$};
        \node [ob] (t) at (7,0) {$\oS_2$};
        \node [ob] (u) at (0,7) {$\oS_2$};
        \node [ob] (d) at (0,-7) {$\oS_1$};
        \draw [a] (s) -- node[arr,al] {$\fF$} (u);
        \draw [a] (u) -- node[arr,ar] {$\mT_2$} (t);
        \draw [a] (s) -- node[arr,bl] {$\mT_1$} (d);
        \draw [a] (d) -- node[arr,br] {$\fF$} (t);
        \node [cell] at (0,0) {$\chi$};
      \end{tikzpicture}
    }
    \qqq
    \mathclap{
      \begin{tikzpicture}
        \begin{scope}
          \clip [boundc] (-6,-6) rectangle (6,6);
          \fill [bg] (-6,-6) rectangle (6,6);
          \renewcommand{\mypath}{(-3,6) .. controls +(0,-3) and +(-1,1) .. (0,0) .. controls +(1,-1) and +(0,3) .. (3,-6)};
          \fill [bgd] \mypath -- (6,-6) -- (6,6) -- cycle;
          \draw [omon,blue] (3,6) .. controls +(0,-3) and +(1,1) .. (0,0);
          \draw [omon,red] (0,0) .. controls +(-1,-1) and +(0,3) .. (-3,-6);
          \draw [olax,purp] (-3,7) -- \mypath -- (3,-8);
        \end{scope}
        \draw [bound] (-6,-6) rectangle (6,6);
        \node [label,blue,above] at (3,6) {$\mT_2$};
        \node [label,red,below] at (-3,-6) {$\mT_1$};
      \end{tikzpicture}
    }
    \vspace{-7.5pt}
  \]
  satisfying
  
  \[\preq
    \mathclap{
      \begin{tikzpicture}[scale=.825]
        \node [ob] (s) at (-7,0) {$\oS_1$};
        \node [ob] (t) at (10.5,3.5) {$\oS_2$};
        \node [ob] (u) at (0,7) {$\oS_2$};
        \node [ob] (d) at (3.5,-3.5) {$\oS_1$};
        \node [ob] (tt) at (14,-7) {$\oS_2$};
        \node [ob] (dd) at (7,-14) {$\oS_1$};
        \draw [a] (s) -- node[arr,al] {$\fF$} (u);
        \draw [a] (u) -- node[arr,above] {$\mT_2$} (t);
        \draw [a] (s) -- node[arr,below,pos=.3] {$\mT_1$} (d);
        \draw [a] (d) -- node[arr,br] {$\fF$} (t);
        \draw [a] (t) -- node[arr,right] {$\mT_2$} (tt);
        \draw [a] (d) -- node[arr,left,pos=.7] {$\mT_1$} (dd);
        \draw [a] (dd) -- node[arr,br] {$\fF$} (tt);
        \draw [a] (s) .. controls +(2,-8) and +(-8,2) .. node[arr,bl] {$\mT_1$} (dd);
        \node [cell] at (1.75,1.75) {$\chi$};
        \node [cell] at (8.75,-5.25) {$\chi$};
        \node [cell] at (0,-7) {$\mu$};
      \end{tikzpicture}
      \eqq
      \begin{tikzpicture}[scale=.825]
        \node [ob] (s) at (-7,0) {$\oS_1$};
        \node [ob] (t) at (10.5,3.5) {$\oS_2$};
        \node [ob] (u) at (0,7) {$\oS_2$};
        \node [ob] (tt) at (14,-7) {$\oS_2$};
        \node [ob] (dd) at (7,-14) {$\oS_1$};
        \draw [a] (s) -- node[arr,al] {$\fF$} (u);
        \draw [a] (u) -- node[arr,above] {$\mT_2$} (t);
        \draw [a] (t) -- node[arr,right] {$\mT_2$} (tt);
        \draw [a] (dd) -- node[arr,br] {$\fF$} (tt);
        \draw [a] (u) .. controls +(2,-8) and +(-8,2) .. node[arr,bl,pos=.8] {$\mT_2$} (tt);
        \draw [a] (s) .. controls +(2,-8) and +(-8,2) .. node[arr,bl] {$\mT_1$} (dd);
        \node [cell] at (.5,-5) {$\chi$};
        \node [cell] at (7,0) {$\mu$};
      \end{tikzpicture}
    }
    \qqq
    \mathclap{
      \begin{tikzpicture}[xscale=.7894,yscale=.7]
        \begin{scope}
          \clip [boundc] (-10,-14) rectangle (9,10);
          \fill [bg] (-10,-14) rectangle (9,10);
          \renewcommand{\mypath}{(-5,10) -- (-5,9) .. controls +(0,-3) and +(-1,1) .. (-2,3) -- (1,0) .. controls +(2,-2) and +(0,3) .. (4,-8) -- (4,-14)};
          \fill [bgd] \mypath -- (9,-14) -- (9,10) -- cycle;
          \draw [omon,blue] (1,10) -- (1,9) .. controls +(0,-3) and +(1,1) .. (-2,3);
          \draw [omon,blue] (5.5,10) -- (5.5,9) .. controls +(0,-5) and +(1,1) .. (1,0);
          \draw [omon,red] (-2,3) .. controls +(-1,-1) and +(0,5) .. (-6.5,-6) .. controls +(0,-3) and +(-1,1) .. (-4.25,-11);
          \draw [omon,red] (1,0) .. controls +(-1,-1) and +(0,3) .. (-2,-6) .. controls +(0,-3) and +(1,1) .. (-4.25,-11);
          \draw [omon,red] (-4.25,-11) -- (-4.25,-14);
          \node [tmon,red] at (-4.25,-11) {};
          \draw [olax,purp] (-5,11) -- \mypath -- (5,-16);
        \end{scope}
        \draw [bound] (-10,-14) rectangle (9,10);
      \end{tikzpicture}
      \eqq
      \begin{tikzpicture}[yscale=.7]
        \begin{scope}
          \clip [boundc] (-7.5,-13) rectangle (7.5,11);
          \fill [bg] (-7.5,-14) rectangle (7.5,11);
          \renewcommand{\mypath}{(-3.5,11) -- (-3.5,5) .. controls +(0,-5) and +(-1,1) .. (0,-4.5) .. controls +(1,-1) and +(0,5) .. (3.5,-13) -- (3.5,-14)};
          \fill [bgd] \mypath -- (7.5,-14) -- (7.5,11) -- cycle;
          \draw [omon,blue] (1,11) -- (1,9) .. controls +(0,-3.5) and +(-1,1) .. (2.75,2);
          \draw [omon,blue] (4.5,11) --(4.5,9) .. controls +(0,-3.5) and +(1,1) .. (2.75,2);
          \draw [omon,blue] (2.75,2) .. controls +(0,-2.5) and +(1,1) .. (0,-4.5);
          \node [tmon,blue] at (2.75,2) {};
          \draw [omon,red] (0,-4.5) .. controls +(-1,-1) and +(0,4) .. (-3.5,-13) -- (-3.5,-14);
          \draw [olax,purp] (-3.5,12) -- \mypath -- (3.5,-16);
        \end{scope}
        \draw [bound] (-7.5,-13) rectangle (7.5,11);
      \end{tikzpicture}
    }
  \]
  \vspace{-10pt}
  \[\preq
    \mathclap{
      \begin{tikzpicture}[scale=.825]
        \node [ob] (s) at (-7,0) {$\oS_1$};
        \node [ob] (u) at (0,7) {$\oS_2$};
        \node [ob] at (0,7) {$\phantom{\oS_2}$};
        \node [ob] (tt) at (14,-7) {$\oS_2$};
        \node [ob] (dd) at (7,-14) {$\oS_1$};
        \draw [a] (s) -- node[arr,al] {$\fF$} (u);
        \draw [a] (dd) -- node[arr,br] {$\fF$} (tt);
        \draw [a] (s) .. controls +(2,-8) and +(-8,2) .. node[arr,bl] {$\mT_1$} (dd);
        \draw [eq] (u) .. controls +(8,-2) and +(-2,8) .. (tt);
        \draw [eq] (s) .. controls +(8,-2) and +(-2,8) .. (dd);
        \node [cell] at (5.75,-1.25) {$=$};
        \node [cell] at (-.25,-7.25) {$\eta$};
      \end{tikzpicture}
      \eqq
      \begin{tikzpicture}[scale=.825]
        \node [ob] (s) at (-7,0) {$\oS_1$};
        \node [ob] (u) at (0,7) {$\oS_2$};
        \node [ob] (tt) at (14,-7) {$\oS_2$};
        \node [ob] (dd) at (7,-14) {$\oS_1$};
        \draw [a] (s) -- node[arr,al] {$\fF$} (u);
        \draw [a] (dd) -- node[arr,br] {$\fF$} (tt);
        \draw [a] (u) .. controls +(2,-8) and +(-8,2) .. node[arr,bl,pos=.8] {$\mT_2$} (tt);
        \draw [eq] (u) .. controls +(8,-2) and +(-2,8) .. (tt);
        \draw [a] (s) .. controls +(2,-8) and +(-8,2) .. node[arr,bl] {$\mT_1$} (dd);
        \node [cell] at (.5,-5) {$\chi$};
        \node [cell] at (6.75,-.25) {$\eta$};
      \end{tikzpicture}
    }
    \qqq
    \mathclap{
      \begin{tikzpicture}[yscale=.7]
        \begin{scope}
          \clip [boundc] (-7.5,-13) rectangle (7.5,8);
          \fill [bg] (-7.5,-14) rectangle (7.5,11);
          \renewcommand{\mypath}{(-3.5,11) -- (-3.5,8) .. controls +(0,-5) and +(-1,1) .. (0,-.5) .. controls +(1,-1) and +(0,5) .. (3.5,-9.5) -- (3.5,-14)};
          \fill [bgd] \mypath -- (7.5,-14) -- (7.5,11) -- cycle;
          \draw [omon,red] (-3.25,-7.5) -- (-3.25,-14);
          \node [tmon,red] at (-3.25,-7.5) {};
          \draw [olax,purp] (-3.5,12) -- \mypath -- (3.5,-16);
        \end{scope}
        \draw [bound] (-7.5,-13) rectangle (7.5,8);
      \end{tikzpicture}
      \eqq
      \begin{tikzpicture}[yscale=.7]
        \begin{scope}
          \clip [boundc] (-7.5,-13) rectangle (7.5,8);
          \fill [bg] (-7.5,-14) rectangle (7.5,11);
          \renewcommand{\mypath}{(-3.5,11) -- (-3.5,5) .. controls +(0,-5) and +(-1,1) .. (0,-4.5) .. controls +(1,-1) and +(0,5) .. (3.5,-13) -- (3.5,-14)};
          \fill [bgd] \mypath -- (7.5,-14) -- (7.5,11) -- cycle;
          \draw [omon,blue] (2.75,2) .. controls +(0,-2.5) and +(1,1) .. (0,-4.5);
          \node [tmon,blue] at (2.75,2) {};
          \draw [omon,red] (0,-4.5) .. controls +(-1,-1) and +(0,4) .. (-3.5,-13) -- (-3.5,-14);
          \draw [olax,purp] (-3.5,12) -- \mypath -- (3.5,-16);
        \end{scope}
        \draw [bound] (-7.5,-13) rectangle (7.5,8);
      \end{tikzpicture}
    }
  \]
  \afterimage

  \noindent
  There are two kinds of morphisms between monad lax 1-cells. The
  first is from~\cite{street:monads}, and the second is from~\cite{wood:ii} and
  \cite{lack-street}.

  A \emph{monad 2-cell} between monad lax 1-cells
  $\fF_1, \fF_2 \colon \mT_1 \to \mT_2$ is a 2-cell
  $\gamma \colon \fF_1 \Rightarrow \fF_2$
  \vspace{-7.5pt}
  \[\preq
    \mathclap{
      \begin{tikzpicture}
        \node [ob] (s) at (-7,0) {$\oS_1$};
        \node [ob] (t) at (7,0) {$\oS_2$};
        \draw [a] (s) .. controls +(5,3) and +(-5,3) .. node[arr,above] {$\fF_1$} (t);
        \draw [a] (s) .. controls +(5,-3) and +(-5,-3) .. node[arr,below] {$\fF_2$} (t);
        \node [cell] at (0,0) {$\gamma$};
      \end{tikzpicture}
    }
    \qqq
    \mathclap{
      \begin{tikzpicture}
        \begin{scope}
          \clip [boundc] (-6,-6) rectangle (6,6);
          \fill [bg] (-6,-6) rectangle (6,6);
          \fill [bgd] (0,-6) rectangle (6,6);
          \draw [olax,purp] (0,7) -- (0,0);
          \draw [olax,dpurp] (0,0) -- (0,-8);
          \node [tmodi] at (0,0) {};
        \end{scope}
        \draw [bound] (-6,-6) rectangle (6,6);
        \node [label,purp,above] at (0,6) {$\fF_1$};
        \node [label,dpurp,below] at (0,-6) {$\fF_2$};
        \node [label,left=4] at (0,0) {$\gamma$};
      \end{tikzpicture}
    }
    \vspace{-10pt}
  \]
  satisfying
  
  \[\preq
    \mathclap{
      \begin{tikzpicture}[xscale=.765625,yscale=.875]
        \path (0,10) -- (0,-10);
        \node [ob] (s) at (-10,0) {$\oS_1$};
        \node [ob] (t) at (10,0) {$\oS_2$};
        \node [ob] (u) at (0,7) {$\oS_2$};
        \node [ob] (d) at (0,-7) {$\oS_1$};
        \draw [a] (s) -- node[arr,brcl,pos=.33] {$\fF_2$} (u);
        \draw [a] (u) -- node[arr,ar] {$\mT_2$} (t);
        \draw [a] (s) -- node[arr,bl] {$\mT_1$} (d);
        \draw [a] (d) -- node[arr,br] {$\fF_2$} (t);
        \draw [a] (s) .. controls +(-1,6) and +(-6,2) .. node[arr,al] {$\fF_1$} (u);
        \node [cell] at (.5,-.5) {$\chi$};
        \node [cell] at (-6.5,5) {$\gamma$};
      \end{tikzpicture}
      \eqq
      \begin{tikzpicture}[xscale=.765625,yscale=.875]
        \path (0,10) -- (0,-10);
        \node [ob] (s) at (-10,0) {$\oS_1$};
        \node [ob] (t) at (10,0) {$\oS_2$};
        \node [ob] (u) at (0,7) {$\oS_2$};
        \node [ob] (d) at (0,-7) {$\oS_1$};
        \draw [a] (s) -- node[arr,al] {$\fF_1$} (u);
        \draw [a] (u) -- node[arr,ar] {$\mT_2$} (t);
        \draw [a] (s) -- node[arr,bl] {$\mT_1$} (d);
        \draw [a] (d) -- node[arr,alcl,pos=.66] {$\fF_1$} (t);
        \draw [a] (d) .. controls +(6,-2) and +(1,-6) .. node[arr,br] {$\fF_2$} (t);
        \node [cell] at (-.5,.5) {$\chi$};
        \node [cell] at (6.75,-5.5) {$\gamma$};
      \end{tikzpicture}
    }
    \qqq
    \mathclap{
      \begin{tikzpicture}[xscale=.875,yscale=.825]
        \begin{scope}
          \clip [boundc] (-7.5,-6) rectangle (7.5,10);
          \fill [bg] (-7.5,-6) rectangle (7.5,10);
          \renewcommand{\mypatha}{(-3,10) -- (-3,6)};
          \renewcommand{\mypathb}{(-3,6) .. controls +(0,-3) and +(-1,1) .. (0,0) .. controls +(1,-1) and +(0,3) .. (3,-7)};
          \fill [bgd] \mypatha -- \mypathb -- (7.5,-6) -- (7.5,10) -- cycle;
          \draw [omon,blue] (3,10) -- (3,6) .. controls +(0,-3) and +(1,1) .. (0,0);
          \draw [omon,red] (0,0) .. controls +(-1,-1) and +(0,3) .. (-3,-6);
          \draw [olax,purp] (-3,11) -- \mypatha;
          \draw [olax,dpurp] \mypathb -- (3,-8);
          \node [tmodi] at (-3,6) {};
        \end{scope}
        \draw [bound] (-7.5,-6) rectangle (7.5,10);
      \end{tikzpicture}
      \eqq
      \begin{tikzpicture}[xscale=.875,yscale=.825]
        \begin{scope}
          \clip [boundc] (-7.5,-10) rectangle (7.5,6);
          \fill [bg] (-7.5,-10) rectangle (7.5,6);
          \renewcommand{\mypatha}{(-3,7.5) .. controls +(0,-3) and +(-1,1) .. (0,0) .. controls +(1,-1) and +(0,3) .. (3,-6.5)};
          \renewcommand{\mypathb}{(3,-6) -- (3,-10)};
          \fill [bgd] \mypatha -- \mypathb -- (7.5,-10) -- (7.5,6) -- cycle;
          \draw [omon,blue] (3,6) .. controls +(0,-3) and +(1,1) .. (0,0);
          \draw [omon,red] (0,0) .. controls +(-1,-1) and +(0,3) .. (-3,-6) -- (-3,-10);
          \draw [olax,purp] (-3,8.5) -- \mypatha;
          \draw [olax,dpurp] \mypathb -- (3,-12);
          \node [tmodi] at (3,-6) {};
        \end{scope}
        \draw [bound] (-7.5,-10) rectangle (7.5,6);
      \end{tikzpicture}
    }
  \]
  \afterimage

  \noindent
  The other kind of morphism between monad lax 1-cells was not
  explicitly given its own name in~\cite{wood:ii}
  or~\cite{lack-street}, so here we call it by its application
  in~\cite{garner:ionads,fairbanks-carlson-spivak}.
  
  A \emph{monad specialization} between
  $\fF_1, \fF_2 \colon \mT_1 \to \mT_2$ is a 2-cell
  $\sigma \colon \fF_1 \Rightarrow \mT_1 \then \fF_2$\vspace{-2.5pt}
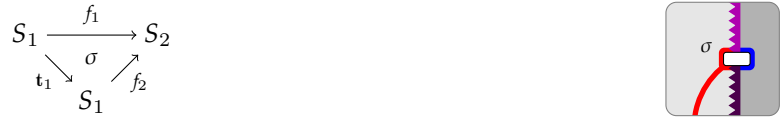
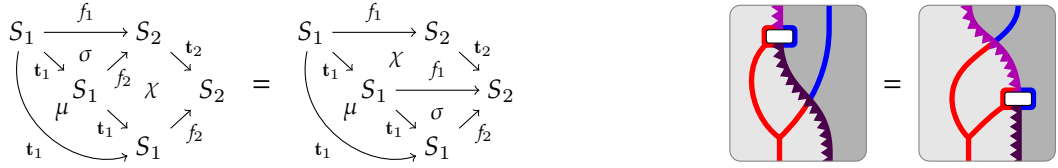
  \[\preq
    \mathclap{
      \begin{tikzpicture}
        \node [ob] (s) at (-7,0) {$\oS_1$};
        \node [ob] (m) at (0,-7) {$\oS_1$};
        \node [ob] (t) at (7,0) {$\oS_2$};
        \draw [a] (s) -- node[arr,above] {$\fF_1$} (t);
        \draw [a] (s) -- node[arr,bl] {$\mT_1$} (m);
        \draw [a] (m) -- node[arr,br] {$\fF_2$} (t);
        \node [cell] at (0,-2.5) {$\sigma$};
      \end{tikzpicture}
    }
    \qqq
    \mathclap{
      \begin{tikzpicture}
        \begin{scope}
          \clip [boundc] (-6,-6) rectangle (6,6);
          \fill [bg] (-6,-6) rectangle (6,6);
          \begin{scope}
            \clip (1.5,6) -- (1.5,-6) -- (6,-6) -- (6,6) -- cycle;
            \fill [bgd] (-6,-6) rectangle (6,6);
            \node [tspech,blue] at (1.5,0) {};
          \end{scope}
          \begin{scope}
            \clip (1.5,6) -- (1.5,-6) -- (-6,-6) -- (-6,6) -- cycle;
            \node [tspech,red] at (1.5,0) {};
          \end{scope}
          \draw [omon,red] (1.5,0) .. controls +(-2.5,-1) and +(0,3) .. (-3,-7);
          \draw [olax,purp] (1.5,7) -- (1.5,0);
          \draw [olax,dpurp] (1.5,0) -- (1.5,-8);
          \node [tspec] at (1.5,0) {};
        \end{scope}
        \draw [bound] (-6,-6) rectangle (6,6);
        \node [label,above=4.5,left=5] at (1.5,0) {$\sigma$};
      \end{tikzpicture}
    }\vspace{-5pt}
  \]

  \noindent
  satisfying\vspace{-2.5pt}
  \[\preq
    \mathclap{
      \begin{tikzpicture}[xscale=.765625,scale=.875]
        \node [ob] (ss) at (-20,7) {$\oS_1$};
        \node [ob] (s) at (-10,0) {$\oS_1$};
        \node [ob] (t) at (10,0) {$\oS_2$};
        \node [ob] (u) at (0,7) {$\oS_2$};
        \node [ob] (d) at (0,-7) {$\oS_1$};
        \draw [a] (ss) -- node[arr,blcl] {$\mT_1$} (s);
        \draw [a] (s) -- node[arr,brcl,pos=.15] {$\fF_2$} (u);
        \draw [a] (u) -- node[arr,ar] {$\mT_2$} (t);
        \draw [a] (s) -- node[arr,blcl] {$\mT_1$} (d);
        \draw [a] (d) -- node[arr,br] {$\fF_2$} (t);
        \draw [a] (ss) -- node[arr,above] {$\fF_1$} (u);
        \draw [a] (ss) .. controls +(-3,-11) and +(-11,-3) .. node[arr,bl] {$\mT_1$} (d);
        \node [cell] at (.5,0) {$\chi$};
        \node [cell] at (-10,4.25) {$\sigma$};
        \node [cell] at (-13.75,-2.625) {$\mu$};
      \end{tikzpicture}
      \eqq
      \begin{tikzpicture}[xscale=.765625,scale=.875]
        \node [ob] (ss) at (-20,7) {$\oS_1$};
        \node [ob] (s) at (-10,0) {$\oS_1$};
        \node [ob] (t) at (10,0) {$\oS_2$};
        \node [ob] (u) at (0,7) {$\oS_2$};
        \node [ob] (d) at (0,-7) {$\oS_1$};
        \draw [a] (ss) -- node[arr,blcl] {$\mT_1$} (s);
        \draw [a] (s) -- node[arr,above] {$\fF_1$} (t);
        \draw [a] (u) -- node[arr,ar] {$\mT_2$} (t);
        \draw [a] (s) -- node[arr,blcl] {$\mT_1$} (d);
        \draw [a] (d) -- node[arr,br] {$\fF_2$} (t);
        \draw [a] (ss) -- node[arr,above] {$\fF_1$} (u);
        \draw [a] (ss) .. controls +(-3,-11) and +(-11,-3) .. node[arr,bl] {$\mT_1$} (d);
        \node [cell] at (0,-2.75) {$\sigma$};
        \node [cell] at (-6.5,3.75) {$\chi$};
        \node [cell] at (-13.75,-2.625) {$\mu$};
      \end{tikzpicture}
    }
    \qqq
    \mathclap{
      \begin{tikzpicture}[xscale=.875,yscale=.825]
        \begin{scope}
          \clip [boundc] (-9,-10) rectangle (7.5,10);
          \fill [bg] (-9,-10) rectangle (7.5,10);
          \coordinate (mypoint) at (-3,6);
          \renewcommand{\mypath}{(3,10) -- (3,7) .. controls +(0,-8) and +(1,1) .. (-3,-7)}
          \renewcommand{\mypatha}{(-3,10) -- (mypoint)};
          \renewcommand{\mypathb}{(mypoint) .. controls +(0,-8)  and +(0,8) .. (3,-9) -- (3,-12)};
          \begin{scope}
            \clip \mypatha -- \mypathb -- (7.5,-10) -- (7.5,10) -- cycle;
            \fill [bgd] (-9,-10) rectangle (7.5,10);
            \node [tspech,blue] at (mypoint) {};
            \draw [omon,blue] \mypath;
          \end{scope}
          \begin{scope}
            \clip \mypatha -- \mypathb -- (-9,-10) -- (-9,10) -- cycle;
            \node [tspech,red] at (mypoint) {};
            \draw [omon,red] \mypath;
            \draw [omon,red] (-3,-7) -- (-3,-10);
            \draw [omon,red] (mypoint) .. controls +(-5,-2) and +(-4,3) .. (-3,-7);
            \node [tmon,red] at (-3,-7) {};
          \end{scope}
          \draw [olax,purp] (-3,11) -- \mypatha;
          \draw [olax,dpurp] \mypathb -- (3,-12);
          \node [tspec] at (mypoint) {};
        \end{scope}
        \draw [bound] (-9,-10) rectangle (7.5,10);
      \end{tikzpicture}
      \eqq
      \begin{tikzpicture}[xscale=.875,yscale=.825]
        \begin{scope}
          \clip [boundc] (-9,-10) rectangle (7.5,10);
          \fill [bg] (-9,-10) rectangle (7.5,10);
          \coordinate (mypoint) at (3,-2);
          \renewcommand{\mypath}{(3,10) .. controls +(0,-5) and +(0,6) .. (-5,-2) .. controls +(0,-2) and +(-1,1) .. (-3,-7)}
          \renewcommand{\mypatha}{(-3,10) .. controls +(0,-5) and +(0,9) .. (3,-2.5)};
          \renewcommand{\mypathb}{(mypoint) -- (3,-12)};
          \begin{scope}
            \clip \mypatha -- \mypathb -- (7.5,-10) -- (7.5,10) -- cycle;
            \fill [bgd] (-9,-10) rectangle (7.5,10);
            \node [tspech,blue] at (mypoint) {};
            \draw [omon,blue] \mypath;
          \end{scope}
          \begin{scope}
            \clip \mypatha -- \mypathb -- (-9,-10) -- (-9,10) -- cycle;
            \node [tspech,red] at (mypoint) {};
            \draw [omon,red] \mypath;
            \draw [omon,red] (-3,-7) -- (-3,-10)(0,4);
            \draw [omon,red] (mypoint) .. controls +(-2.5,-1) and +(1,1) .. (-3,-7);
            \node [tmon,red] at (-3,-7) {};
          \end{scope}
          \draw [olax,purp] (-3,11) -- \mypatha;
          \draw [olax,dpurp] \mypathb -- (3,-12);
          \node [tspec] at (mypoint) {};
        \end{scope}
        \draw [bound] (-9,-10) rectangle (7.5,10);
      \end{tikzpicture}
    }
  \]
  \vspace{-15pt}
\end{definition}

\begin{convention}\label{rem:sliding}
  In string diagrams, a monad lax 1-cell $\fF$ is a sliding membrane,
  transforming between $\oS_1$ and $\mT_1$ on its left and $\oS_2$ and
  $\mT_2$ on its right, capable of crossing the monad fluid through
  the NW/SE direction. We write $\fF$ with a texture indicating the
  direction monads progress downward through it. A monad 2-cell
  $\gamma$ is a node between lax 1-cells that likewise can freely
  slide across monads. A monad specialization $\sigma$ is similar to a
  monad 2-cell, except it is itself stuck in monad fluid and cannot
  come out.
\end{convention}

Given a 2-category $\tX$, we write the 2-category of monads, monad lax
1-cells, and monad 2-cells as $\Mnd(\tX)$, and we write the 2-category
of monads, monad lax 1-cells, and monad specializations as $\EM(\tX)$
following~\cite{lack-street}.\footnote{This 2-category is called
  $\EM(\tX)$ because it is the free completion of the 2-category $\tX$
  under algebra objects, a.k.a.\ Eilenberg-Moore
  objects. See~\cite{lack-street} for details.} That these form
2-categories is straightforward; we also cover this later in the more
general \cref{prop:doublemonads}.

\begin{remark}\label{rem:tospec}
  There is an identity-on-1-cells functor $\Mnd(\tX) \to \EM(\tX)$
  given by \vspace{-2.5pt}
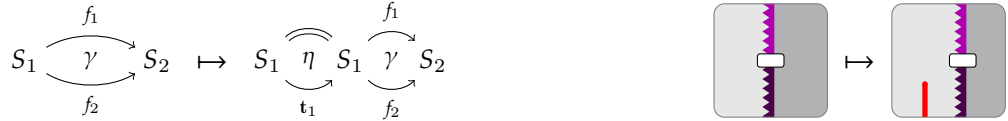
  \[\preq
    \mathclap{
      \begin{tikzpicture}
        \node [ob] (s) at (-7,0) {$\oS_1$};
        \node [ob] (t) at (7,0) {$\oS_2$};
        \draw [a] (s) .. controls +(5,3) and +(-5,3) .. node[arr,above] {$\fF_1$} (t);
        \draw [a] (s) .. controls +(5,-3) and +(-5,-3) .. node[arr,below] {$\fF_2$} (t);
        \node [cell] at (0,0) {$\gamma$};
      \end{tikzpicture}
      \eqquad\mapsto\eqquad
      \begin{tikzpicture}[xscale=.625,yscale=1.125]
        \node [ob] (s) at (-7,0) {$\oS_1$};
        \node [ob] (t) at (7,0) {$\oS_2$};
        \node [ob] (ss) at (-21,0) {$\oS_1$};
        \draw [eq] (ss) .. controls +(5,3) and +(-5,3) .. (s);
        \draw [a] (ss) .. controls +(5,-3) and +(-5,-3) .. node[arr,below] {$\mT_1$} (s);
        \draw [a] (s) .. controls +(5,3) and +(-5,3) .. node[arr,above] {$\fF_1$} (t);
        \draw [a] (s) .. controls +(5,-3) and +(-5,-3) .. node[arr,below] {$\fF_2$} (t);
        \node [cell] at (-14,0) {$\eta$};
        \node [cell] at (0,0) {$\gamma$};
      \end{tikzpicture}
    }
    \qqq
    \mathclap{
      \begin{tikzpicture}
        \begin{scope}
          \clip [boundc] (-6,-6) rectangle (6,6);
          \fill [bg] (-6,-6) rectangle (6,6);
          \fill [bgd] (0,-6) rectangle (6,6);
          \draw [olax,purp] (0,7) -- (0,0);
          \draw [olax,dpurp] (0,0) -- (0,-8);
          \node [tmodi] at (0,0) {};
        \end{scope}
        \draw [bound] (-6,-6) rectangle (6,6);
      \end{tikzpicture}
      \eqquad\mapsto\eqquad
      \begin{tikzpicture}
        \begin{scope}
          \clip [boundc] (-6,-6) rectangle (6,6);
          \fill [bg] (-6,-6) rectangle (6,6);
          \fill [bgd] (1.5,-6) rectangle (6,6);
          \draw [olax,purp] (1.5,7) -- (1.5,0);
          \draw [olax,dpurp] (1.5,0) -- (1.5,-8);
          \node [tmodi] at (1.5,0) {};
          \draw [omon,red] (-2.5,-2.5) -- (-2.5,-7);
          \node [tmon,red] at (-2.5,-2.5) {};
        \end{scope}
        \draw [bound] (-6,-6) rectangle (6,6);
      \end{tikzpicture}
    }
    \vspace{-2.5pt}
  \]
  
  \noindent
  In fact, an arbitrary 2-cell $\gamma \colon \fF_1 \Rightarrow \fF_2$
  is a monad 2-cell if and only if its composite with the unit $\eta$,
  as shown above, is a monad specialization. Thus monad 2-cells are
  equivalently monad specializations equipped with the data of a
  factoring through the unit.
\end{remark}

\begin{remark}\label{rem:transfors}
  The cells of $\Mnd(\tX)$ are instances of the 2-category-theoretic
  notions of functor, lax transformation, and
  modification. Specifically, let $\tDelta$ denote the free 2-category
  containing a monad on an object. A monad in $\tX$ is equivalently a
  functor of 2-categories from $\tDelta$ to $\tX$, a monad lax 1-cell
  is equivalently a lax transformation between the corresponding
  functors, and a monad 2-cell is equivalently a modification between
  the corresponding lax transformations.

  In string diagrams, functors, transformations (lax, colax, or
  neither), and modifications between 2-categories $\tcat{C}$ and
  $\tcat{D}$ are represented as regions, strings, and nodes that may
  be superimposed on string diagrams in $\tcat{C}$ (with strings
  crossing in an appropriate direction) to yield string diagrams in
  $\tcat{D}$. In particular, monad lax 1-cells and monad 2-cells are
  represented as strings and nodes which may slide across the monad
  diagrams, as described in \cref{rem:sliding}. See
  \cite{morehouse:two} or \cite[Appendix A]{fairbanks-shulman} for
  some explanation of how this graphical language arises.

  A monad in a 2-category may also be viewed as a lax
  functor\footnote{A lax functor of 2-categories is the same as a
    functor between their underlying virtual 2-categories. The notions
    of lax transformation and modification make sense not only for lax
    functors between 2-categories, but more generally for functors
    from a virtual 2-category to a 2-category (or implicit
    2-category).} from the terminal 2-category $\tcat{1}$. The cells
  of $\Mnd(\tX)$ are then equally well viewed as such lax functors
  from $\tcat{1}$, lax transformations between them, and modifications
  between them. The 2-cells of $\EM(\tX)$ are given by
  \emph{modulations} between such lax transformations~\cite[Section
  4]{cockett-koslowski-seely-wood}.
\end{remark}

\begin{remark}\label{rem:laxmodule}
  The notion of monad lax 1-cell subsumes earlier notions.
  \begin{itemize}
  \item
    A monad map
    $\mT_1 \to \mT_2$ is equivalently a monad lax 1-cell
    $\mT_2 \to \mT_1$ that is carried by an identity 1-cell
    $\id_{\oS}$.
  \item A $\mT$-module $\bM \colon \oX \to \oS$ is equivalently a monad lax
    1-cell $\id_{\oX} \to \mT$ from an identity monad.
  \item
    Simply a 1-cell $\oS \to \oY$ is equivalently a monad lax
    1-cell $\mT \to \id_{\oY}$ to an identity monad.
  \end{itemize}

  The latter two observations are relevant for defining 2-categorical
  adjoints to the functor of 2-categories
  $\id_{\dash} \colon \tX \to \Mnd(\tX)$ as studied
  in~\cite{street:monads}. Namely, to give an algebra object $\AlgTS$
  is precisely to define the right adjoint to $\id_{\dash}$ at the
  object $\mT$. On the other hand the left adjoint to $\id_{\dash}$
  always exists, sending a monad to its underlying object.
\end{remark}

\begin{proposition}\label{prop:formaltfae}
  Let $\mT_1$ and $\mT_2$ be monads on $\oS_1$ and $\oS_2$, with
  algebra objects $\Alg{\mT_1}{\oS_1}$ and $\Alg{\mT_2}{\oS_2}$. Let

  $\fF \colon \oS_1 \to \oS_2$ be an arbitrary 1-cell.
  
  \begin{enumerate}[label=(\roman*)]
  \item\label{item:monadone}
    Giving $\fF$ the structure of a monad lax 1-cell $\mT_1 \to
    \mT_2$
    
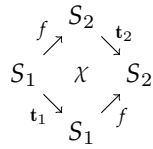
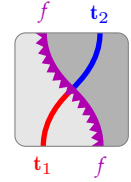
\vspace{-7.5pt}
    \[\preq
      \hspace{-\leftmargin}
      \mathclap{
        %
      }
      \afterimage
      \vspace{-7.5pt}
    \]

    \noindent
    is equivalent to giving a 1-cell between algebra objects
    \vspace{-16.25pt}
    \[\preq
      \hspace{-\leftmargin}
      \mathclap{\emlift{\fF} \colon \Alg{\mT_1}{\oS_1} \to \Alg{\mT_2}{\oS_2}}
      \qqq
      \mathclap{
        %
      }
    \]
    \afterimage
    \vspace{-8.75pt}
  \end{enumerate}

  \noindent
  Now let $\fF_1, \fF_2 \colon S_1 \to S_2$ be monad lax 1-cells
  $\mT_1 \to \mT_2$. Let $\gamma \colon \fF_1 \Rightarrow \fF_2$
  be an arbitrary 2-cell.
  
  \begin{enumerate}[label=(\roman*)]
    \setcounter{enumi}{1}
  \item\label{item:monadtwo} We have that $\gamma$ is a monad 2-cell
    $\fF_1 \Rightarrow \fF_2$ if and only if there is a (unique) 2-cell
    ${\emlift{\gamma} \colon \emlift{\fF_1} \Rightarrow
      \emlift{\fF_2}}$ \vspace{-15pt}
    \[\preq
      \hspace{-\leftmargin}
      \mathclap{
        %
      }
    \]
    is equivalent to an arbitrary 2-cell
    $\emlift{\sigma} \colon \emlift{\fF_1} \Rightarrow \emlift{\fF_2}$.
    \[\preq
      \hspace{-\leftmargin}
      \mathclap{
        %
      }
      \afterimage
      \vspace{-10pt}
    \]
  \end{enumerate}
\end{proposition}

More generally, analogous results hold replacing
$\mcar{\mT_1} \colon \Alg{\mT_1}{\oS_1} \to \oS_1$ with any right
adjoint 1-cell inducing $\mT_1$, without changing the argument. We
will also generalize this further in
\cref{prop:cformaltfae,prop:pfformaltfae}.

\begin{proof}
  First observe that for fixed $\fF \colon \oS_1 \to \oS_2$, giving
  monad lax 1-cell structure on $\fF$ is equivalent to giving
  $\mT_2$-module structure on $\mcar{\mT_1} \then \fF$
  \vspace{-7.5pt}
  
  \[\preq
    \mathclap{
      %
    }
    \afterimage
    \vspace{-5pt}
  \]

  \noindent
  since the module laws
  \vspace{-7.5pt}

  \[\preq
    \mathclap{
      %
    }
    \afterimage
  \]
  
  \noindent
  are mate\footnote{In a 2-category, 2-cells are called \emph{mate}
    when they are related by composing with the unit and counit of an
    adjunction to replace 1-cells in the domain with their adjoints in
    the codomain and/or vice versa. Taking mates is an invertible
    process (providing natural bijections), and equations between
    2-cells can be deduced from equations involving their mates.} to
  the monad lax 1-cell laws.  We thus obtain
  \labelcref{item:monadone}, since $\mT_2$-module structures on
  $\mcar{\mT_1} \then \fF$ are equivalent to factorings of
  $\mcar{\mT_1} \then \fF$ through $\mcar{\mT_2}$ as desired, by the
  universal property of $\Alg{\mT_2}{\oS_2}$.
  
  Likewise, a monad specialization
  $\sigma \colon \fF_1 \Rightarrow \fF_2$ is equivalent to a
  $\mT_2$-module map
  $\phi \colon \mcar{\mT_1} \then \fF_1 \Rightarrow \mcar{\mT_1} \then
  \fF_2$
  \[\preq
    \mathclap{
      \begin{tikzpicture}[yscale=.75]
        \node [ob] (s) at (-7,0) {$\Alg{\mT_1}{\oS_1}$};
        \node [ob] (u) at (0,7) {$\oS_1$};
        \node [ob] (d) at (0,-7) {$\oS_1$};
        \node [ob] (t) at (7,0) {$\oS_2$};
        \draw [a,long] (s) -- node[arr,al] {$\mcar{\mT_1}$} (u);
        \draw [a,long] (u) -- node[arr,ar] {$\fF_1$} (t);
        \draw [a,long] (s) -- node[arr,bl] {$\mcar{\mT_1}$} (d);
        \draw [a,long] (d) -- node[arr,br] {$\fF_2$} (t);
        \node [cell] at (0,0) {$\phi$};
      \end{tikzpicture}
    }
    \qqq
    \mathclap{
      \begin{tikzpicture}
        \begin{scope}
          \clip [boundc] (-6,-6) rectangle (6,6);
          \begin{scope}[shift={(1.5,0)}]
            \fill [bg] (-10,-6) rectangle (6,6);
            \renewcommand{\mypatha}{(-4,-6) .. controls +(0,3) and +(-1,-1) .. (0,0) .. controls +(-1,1) and +(0,-3) .. (-4,6)}
            \renewcommand{\mypathb}{(0,8) .. controls +(0,-2) and +(0,3) .. (0,0)}
            \renewcommand{\mypathc}{(0,0) .. controls +(0,-3) and +(0,2) .. (0,-8)}
            \fill [em=red] \mypatha -- (-10,6) -- (-10,-6) -- cycle;
            \begin{scope}
              \clip \mypatha -- \mypathb -- \mypathc -- cycle;
              \node [tspech,red] at (-.25,0) {};
            \end{scope}
            \begin{scope}
              \clip \mypathb -- \mypathc -- (6,-6) -- (6,6) -- cycle;
              \fill [bgd] (-10,-6) rectangle (6,6);
              \node [tspech,blue] at (-.25,0) {};
            \end{scope}
            \draw [oem,red] \mypatha;
            \draw [olax,purp] \mypathb;
            \draw [olax,dpurp] \mypathc;
            \node [tspec] at (-.25,0) {};
          \end{scope}
        \end{scope}
        \draw [bound] (-6,-6) rectangle (6,6);
        \node [label,right=9.5] at (0,0) {$\phi$};
      \end{tikzpicture}
    }
  \]
  \afterimage
  since the module map law
  \[\preq
    \mathclap{
      \begin{tikzpicture}[xscale=1.5,yscale=1.5]
        \node [ob] (s) at (-7,0) {$\Alg{\mT_1}{\oS_1}$};
        \node [ob] (u) at (0,7) {$\oS_2$};
        \node [ob] (su) at (-3.5,3.5) {$\oS_1$};
        \node [ob] (suu) at (-6.5,6.5) {$\oS_1$};
        \node [ob] (sd) at (0,0) {$\oS_1$};
        \node [ob] (t) at (7,0) {$\oS_2$};
        \draw [a] (s) -- node[arr,left] {$\mcar{\mT_1}$} (suu);
        \draw [a] (suu) -- node[arr,above] {$\fF_1$} (u);
        \draw [a,long] (s) -- node[arr,brcl,pos=0] {$\mcar{\mT_1}$} (su);
        \draw [a,long] (su) -- node[arr,brcl,pos=0] {$\fF_2$} (u);
        \draw [a] (u) -- node[arr,ar] {$\mT_2$} (t);
        \draw [a] (s) -- node[arr,below] {$\mcar{\mT_1}$} (sd);
        \draw [a] (sd) -- node[arr,below] {$\fF_2$} (t);
        \node [cell] at (-4.9,4.9) {$\phi$};
        \node [cell] at (0,2.5) {$\rho$};
        \node [ob] (d) at (0,-7) {$\phantom{\oS_1}$};
      \end{tikzpicture}
      \eq\;
      \begin{tikzpicture}[xscale=1.5,yscale=1.5]
        \node [ob] (s) at (-7,0) {$\Alg{\mT_1}{\oS_1}$};
        \node [ob] (u) at (0,7) {$\oS_2$};
        \node [ob] (su) at (-3.5,3.5) {$\oS_1$};
        \node [ob] (d) at (0,-5) {$\oS_1$};
        \node [ob] (sd) at (0,0) {$\oS_1$};
        \node [ob] (t) at (7,0) {$\oS_2$};
        \draw [a,long] (s) -- node[arr,al,pos=0] {$\mcar{\mT_1}$} (su);
        \draw [a,long] (su) -- node[arr,al,pos=0] {$\fF_1$} (u);
        \draw [a] (s) -- node[arr,bl] {$\mcar{\mT_1}$} (d);
        \draw [a] (d) -- node[arr,br] {$\fF_2$} (t);
        \draw [a] (u) -- node[arr,ar] {$\mT_2$} (t);
        \draw [a] (s) -- node[arr,bcl] {$\mcar{\mT_1}$} (sd);
        \draw [a] (sd) -- node[arr,bcl] {$\fF_1$} (t);
        \node [cell] at (0,-2.3) {$\phi$};
        \node [cell] at (0,3) {$\rho$};
      \end{tikzpicture}
    }
    \qqq
    \mathclap{
      \begin{tikzpicture}[xscale=.875,yscale=1.125]
        \begin{scope}
          \clip [boundc] (-9,-7) rectangle (8,7);
          \fill [bg] (-9,-7) rectangle (8,7);
          \coordinate (mypointa) at (-1.6,2.5);
          \coordinate (mypointb) at (0,-2.5);
          \coordinate (mypointal) at (-1.85,2.5);
          \renewcommand{\mypatha}{(-4,-7) .. controls +(0,3) and +(-1,-1) .. (mypointb) .. controls +(-5,1.5) and +(-4,-2) .. (mypointa) .. controls +(-2,1) and +(0,-2.5) .. (-5.6,7)}
          \renewcommand{\mypathb}{(-1.6,8) -- (mypointa)}
          \renewcommand{\mypathc}{(mypointa) .. controls +(0,-2) and +(0,3) .. (mypointb) -- (0,-8)}
          \fill [em=red] \mypatha -- (-9,7) -- (-9,-7) -- cycle;
          \begin{scope}
            \clip \mypathb -- \mypathc -- (8,-7) -- (8,7) -- cycle;
            \fill [bgd] (-9,-7) rectangle (8,7);
            \node [tspech,blue] at (mypointal) {};
          \end{scope}
          \begin{scope}
            \clip \mypatha -- \mypathb -- \mypathc -- cycle;
            \node [tspech,red] at (mypointal) {};
          \end{scope}
          \draw [oem,red] \mypatha;
          \draw [omon,blue] (4.5,7) -- (4.5,3) .. controls +(0,-3) and +(2,1) .. (mypointb);
          \draw [olax,purp] \mypathb;
          \draw [olax,dpurp] \mypathc;
          \node [tspec] at (mypointal) {};
        \end{scope}
        \draw [bound] (-9,-7) rectangle (8,7);
      \end{tikzpicture}
      \eqq
      \begin{tikzpicture}[xscale=.875,yscale=1.125]
        \begin{scope}
          \clip [boundc] (-9,-7) rectangle (8,7);
          \fill [bg] (-9,-7) rectangle (8,7);
          \coordinate (mypointa) at (0,2.5);
          \coordinate (mypointb) at (0,-2.5);
          \coordinate (mypointbl) at (-.25,-2.5);
          \renewcommand{\mypatha}{(-4,-7) .. controls +(0,3) and +(-1,-1) .. (mypointb) .. controls +(-4,2) and +(-4.5,-2) .. (mypointa) .. controls +(-3.25,1) and +(0,-2) .. (-5.6,7)}
          \renewcommand{\mypathb}{(-1.6,8) .. controls +(0,-3) and +(0,3) .. (mypointa) -- (mypointb)}
          \renewcommand{\mypathc}{(mypointb) -- (0,-8)}
          \fill [em=red] \mypatha -- (-9,7) -- (-9,-7) -- cycle;
          \begin{scope}
            \clip \mypathb -- \mypathc -- (8,-7) -- (8,7) -- cycle;
            \fill [bgd] (-9,-7) rectangle (8,7);
            \node [tspech,blue] at (mypointbl) {};
          \end{scope}
          \begin{scope}
            \clip \mypatha -- \mypathb -- \mypathc -- cycle;
            \node [tspech,red] at (mypointbl) {};
          \end{scope}
          \draw [oem,red] \mypatha;
          \draw [omon,blue] (4.5,7) .. controls +(0,-2) and +(3,1) .. (mypointa);
          \draw [olax,purp] \mypathb;
          \draw [olax,dpurp] \mypathc;
          \node [tspec] at (mypointbl) {};
        \end{scope}
        \draw [bound] (-9,-7) rectangle (8,7);
      \end{tikzpicture}
    }
    \vspace{-10pt}
  \]
  \afterimage

  \noindent
  is mate to the monad specialization law.  We thus obtain
  \labelcref{item:monadspec}, since $\mT_2$-module maps are equivalent
  to 2-cells between 1-cells into $\Alg{\mT_2}{\oS_2}$ again by its
  universal property.

  Finally, observe that the monad specialization $\sigma$ arises from a monad
  2-cell $\gamma$ as in \cref{rem:tospec}
  \[\preq
    \mathclap{
      \begin{tikzpicture}
        \node [ob] (s) at (-7,0) {$\oS_1$};
        \node [ob] (m) at (0,-7) {$\oS_1$};
        \node [ob] (t) at (7,0) {$\oS_2$};
        \draw [a] (s) -- node[arr,above] {$\fF_1$} (t);
        \draw [a] (s) -- node[arr,bl] {$\mT_1$} (m);
        \draw [a] (m) -- node[arr,br] {$\fF_2$} (t);
        \node [cell] at (0,-2.5) {$\sigma$};
      \end{tikzpicture}
      \eqq
      \begin{tikzpicture}[xscale=.625,yscale=1.125]
        \node [ob] (s) at (-7,0) {$\oS_1$};
        \node [ob] (t) at (7,0) {$\oS_2$};
        \node [ob] (ss) at (-21,0) {$\oS_1$};
        \draw [eq] (ss) .. controls +(5,3) and +(-5,3) .. (s);
        \draw [a] (ss) .. controls +(5,-3) and +(-5,-3) .. node[arr,below] {$\mT_1$} (s);
        \draw [a] (s) .. controls +(5,3) and +(-5,3) .. node[arr,above] {$\fF_1$} (t);
        \draw [a] (s) .. controls +(5,-3) and +(-5,-3) .. node[arr,below] {$\fF_2$} (t);
        \node [cell] at (-14,0) {$\eta$};
        \node [cell] at (0,0) {$\gamma$};
      \end{tikzpicture}
    }
    \qqq
    \mathclap{
      \begin{tikzpicture}
        \begin{scope}
          \clip [boundc] (-6,-6) rectangle (6,6);
          \fill [bg] (-6,-6) rectangle (6,6);
          \begin{scope}
            \clip (1.5,6) -- (1.5,-6) -- (6,-6) -- (6,6) -- cycle;
            \fill [bgd] (-6,-6) rectangle (6,6);
            \node [tspech,blue] at (1.5,0) {};
          \end{scope}
          \begin{scope}
            \clip (1.5,6) -- (1.5,-6) -- (-6,-6) -- (-6,6) -- cycle;
            \node [tspech,red] at (1.5,0) {};
          \end{scope}
          \draw [omon,red] (1.5,0) .. controls +(-2.5,-1) and +(0,3) .. (-3,-7);
          \draw [olax,purp] (1.5,7) -- (1.5,0);
          \draw [olax,dpurp] (1.5,0) -- (1.5,-8);
          \node [tspec] at (1.5,0) {};
        \end{scope}
        \draw [bound] (-6,-6) rectangle (6,6);
        \node [label,above=4.5,left=5] at (1.5,0) {$\sigma$};
      \end{tikzpicture}
      \eqq
      \begin{tikzpicture}
        \begin{scope}
          \clip [boundc] (-6,-6) rectangle (6,6);
          \fill [bg] (-6,-6) rectangle (6,6);
          \fill [bgd] (1.5,-6) rectangle (6,6);
          \draw [olax,purp] (1.5,7) -- (1.5,0);
          \draw [olax,dpurp] (1.5,0) -- (1.5,-8);
          \node [tmodi] at (1.5,0) {};
          \draw [omon,red] (-2.5,-2.5) -- (-2.5,-7);
          \node [tmon,red] at (-2.5,-2.5) {};
        \end{scope}
        \draw [bound] (-6,-6) rectangle (6,6);
        \node [label,above=3,left=3] at (1.5,0) {$\gamma$};
      \end{tikzpicture}
    }
  \]
  \afterimage
  if and only if the corresponding $\mT_2$-module map factors as
  \[\preq
    \mathclap{
      \begin{tikzpicture}[xscale=1,yscale=.75]
        \node [ob] (s) at (-7,0) {$\Alg{\mT_1}{\oS_1}$};
        \node [ob] (u) at (0,7) {$\oS_1$};
        \node [ob] (d) at (0,-7) {$\oS_1$};
        \node [ob] (t) at (7,0) {$\oS_2$};
        \draw [a,long] (s) -- node[arr,al] {$\mcar{\mT_1}$} (u);
        \draw [a,long] (u) -- node[arr,ar] {$\fF_1$} (t);
        \draw [a,long] (s) -- node[arr,bl] {$\mcar{\mT_1}$} (d);
        \draw [a,long] (d) -- node[arr,br] {$\fF_2$} (t);
        \node [cell] at (0,0) {$\phi$};
      \end{tikzpicture}
      \eqq
      \begin{tikzpicture}[xscale=.625,yscale=1.125]
        \node [ob] (s) at (-7,0) {$\oS_1$};
        \node [ob] (t) at (7,0) {$\oS_2$};
        \draw [a] (s) .. controls +(5,3) and +(-5,3) .. node[arr,above] {$\fF_1$} (t);
        \draw [a] (s) .. controls +(5,-3) and +(-5,-3) .. node[arr,below] {$\fF_2$} (t);
        \node [cell] at (0,0) {$\gamma$};
        \node [ob] (ss) at (-21,0) {$\Alg{\mT_1}{\oS_1}$};
        \draw [a] (ss) -- node[arr,above] {$\mcar{\mT_1}$} (s);
      \end{tikzpicture}
    }
    \qqq
    \mathclap{
      \begin{tikzpicture}
        \begin{scope}
          \clip [boundc] (-6,-6) rectangle (6,6);
          \begin{scope}[shift={(1.5,0)}]
            \fill [bg] (-10,-6) rectangle (6,6);
            \renewcommand{\mypatha}{(-4,-6) .. controls +(0,3) and +(-1,-1) .. (0,0) .. controls +(-1,1) and +(0,-3) .. (-4,6)}
            \renewcommand{\mypathb}{(0,8) .. controls +(0,-2) and +(0,3) .. (0,0)}
            \renewcommand{\mypathc}{(0,0) .. controls +(0,-3) and +(0,2) .. (0,-8)}
            \fill [em=red] \mypatha -- (-10,6) -- (-10,-6) -- cycle;
            \begin{scope}
              \clip \mypatha -- \mypathb -- \mypathc -- cycle;
              \node [tspech,red] at (-.25,0) {};
            \end{scope}
            \begin{scope}
              \clip \mypathb -- \mypathc -- (6,-6) -- (6,6) -- cycle;
              \fill [bgd] (-10,-6) rectangle (6,6);
              \node [tspech,blue] at (-.25,0) {};
            \end{scope}
            \draw [oem,red] \mypatha;
            \draw [olax,purp] \mypathb;
            \draw [olax,dpurp] \mypathc;
            \node [tspec] at (-.25,0) {};
          \end{scope}
        \end{scope}
        \draw [bound] (-6,-6) rectangle (6,6);
      \end{tikzpicture}
      \eqq
      \begin{tikzpicture}
        \begin{scope}
          \clip [boundc] (-6,-6) rectangle (6,6);
          \begin{scope}[xscale=1.0625,shift={(-.375,0)}]
            \fill [bg] (-6,-6) rectangle (6,6);
            \fill [bgd] (2,-6) rectangle (6,6);
            \fill [em=red] (-6,-6) rectangle (-2,6);
            \draw [oem,red] (-2,6) -- (-2,-6);
            \draw [olax,purp] (2,8) -- (2,0);
            \draw [olax,dpurp] (2,0) -- (2,-8);
            \node [tmhomi] at (2,0) {};
          \end{scope}
        \end{scope}
        \draw [bound] (-6,-6) rectangle (6,6);
      \end{tikzpicture}
    }
  \]
  \afterimage
  which yields \labelcref{item:monadtwo}.
\end{proof}


\begin{remark}\label{rem:emarrow}
  The above \cref{prop:formaltfae}~\labelcref{item:monadone} and
  \labelcref{item:monadtwo} tell us that the 2-category consisting of
  monads in $\tX$ with (chosen) algebra objects, monad lax 1-cells,
  and monad 2-cells embeds as a full sub-2-category of the 2-category
  $[\arrow, \tX]$ of functors from $\arrow$ to $\tX$, strict
  transformations, and modifications. (Here $\arrow$ denotes the free
  2-category containing two objects and a 1-cell between them, so that
  functors from $\arrow$ to a 2-category $\tX$ are 1-cells in $\tX$.)

  A functor from an arbitrary 2-category $\tC$ to $[\arrow, \tX]$ is
  (by currying) the same as a functor from $\arrow$ to $[\tC,
  \tX]$. Assuming the relevant algebra objects exist, a diagram
  ${\tC \to \Mnd(\tX)}$ is then equivalent to a pair of diagrams
  $\tC \to \tX$ and a strict transformation between them whose
  components are 1-cells of the form $\mcarT \colon \AlgTS \to \oS$.
\end{remark}

\begin{remark}
  By the universal property of algebra objects, a 1-cell
  $\Alg{\mT_1}{\oS_1} \to \Alg{\mT_2}{\oS_2}$ is determined by
  assigning each $\mT_1$-module $\oX \to \oS_1$ a $\mT_2$-module
  $\oX \to \oS_2$, functorial in module maps and natural in composing
  with cells into $\oX$. (This is a natural transformation between the
  $\Cat$-valued representable functors defining $\Alg{\mT_1}{\oS_1}$
  and $\Alg{\mT_2}{\oS_2}$, which corresponds to a 1-cell
  $\Alg{\mT_1}{\oS_1} \to \Alg{\mT_2}{\oS_2}$ via Yoneda
  embedding. Likewise, a 2-cell between such 1-cells is determined by
  a modification between the natural transformations.)

  According to~\cref{prop:formaltfae}~\labelcref{item:monadone}, any
  monad lax 1-cell $\fF \colon \mT_1 \to \mT_2$ induces a 1-cell
  $\emlift{\fF} \colon \Alg{\mT_1}{\oS_1} \to \Alg{\mT_2}{\oS_2}$. Here the
  corresponding assignment on modules is given by composition of monad
  1-cells, viewing the $\mT_1$-module itself as a monad lax 1-cell
  from identity.

  \[\preq
    \hspace{-\leftmargin}
    \mathclap{
      \begin{tikzpicture}
        \node [ob] (s) at (-7,0) {$\oX$};
        \node [ob] (m) at (0,7) {$\oS_1$};
        \node [ob] (t) at (7,0) {$\oS_1$};
        \draw [a] (s) -- node[arr,al] {$\bM$} (m);
        \draw [a] (m) -- node[arr,ar] {$\mT_1$} (t);
        \draw [a] (s) -- node[arr,below] {$\bM$} (t);
        \node [cell] at (0,2.5) {$\rho$};
      \end{tikzpicture}
      \eqquad\mapsto\eqquad
      \begin{tikzpicture}
        \node [ob] (s) at (-7,0) {$\oX$};
        \node [ob] (m) at (0,7) {$\oS_1$};
        \node [ob] (t) at (7,0) {$\oS_1$};
        \node [ob] (mt) at (10.5,7) {$\oS_2$};
        \node [ob] (tt) at (17.5,0) {$\oS_2$};
        \draw [a] (s) -- node[arr,al] {$\bM$} (m);
        \draw [a] (m) -- node[arr,arcl] {$\mT_1$} (t);
        \draw [a] (s) -- node[arr,below] {$\bM$} (t);
        \draw [a] (m) -- node[arr,above] {$\fF$} (mt);
        \draw [a] (mt) -- node[arr,ar] {$\mT_2$} (tt);
        \draw [a] (t) -- node[arr,below] {$\fF$} (tt);
        \node [cell] at (0,2.5) {$\rho$};
        \node [cell] at (8.75,3.5) {$\chi$};
      \end{tikzpicture}
    }
    \qqq
    \mathclap{
      \begin{tikzpicture}
        \begin{scope}
          \clip [boundc] (-6,-6) rectangle (6,6);
          \fill [bg] (-6,-6) rectangle (6,6);
          \renewcommand{\mypath}{(-2.6,6) .. controls +(0,-3.5) and +(-1,1) .. (0,0) -- (0,-6)}
          \fill [bgm] \mypath -- (-6,-6) -- (-6,6) -- cycle;
          \draw [omon,red] (2.6,6) .. controls +(0,-3.5) and +(1,1) .. (0,0);
          \draw [omods=red] \mypath;
        \end{scope}
        \draw [bound] (-6,-6) rectangle (6,6);
        \node [label,red,above] at (2.6,6) {$\mT_1$};
        \node [label,above] at (-2.6,6) {$\bM$};
        \node [label,red,below] at (2.6,-6) {$\phantom{\mT_1}$};
        \node [label,below] at (-2.6,-6) {$\phantom{\bM}$};
      \end{tikzpicture}
      \eqquad\mapsto\eqquad
      \begin{tikzpicture}
        \begin{scope}
          \clip [boundc] (-6,-6) rectangle (6,6);
          \fill [bg] (-6,-6) rectangle (6,6);
          \renewcommand{\mypatha}{(-3.1,6) -- (-3.1,5) .. controls +(0,-3.5) and +(-1,1) .. (-.5,-2) -- (-.5,-6)}
          \renewcommand{\mypathb}{(.125,9) -- (.125,5.875) .. controls +(0,-6) and +(0,6) .. (2.875,-4.125) -- (2.875,-9)}
          \renewcommand{\mypathc}{(3.5,6) .. controls +(0,-4.5) and +(1,1) .. (-.5,-2)}
          \fill [bgm] \mypatha -- (-6,-6) -- (-6,6) -- cycle;
          \begin{scope}
            \clip \mypathb -- (-6,-6) -- (-6,6) -- cycle;
            \draw [omon,red] \mypathc;
          \end{scope}
          \begin{scope}
            \clip \mypathb -- (6,-6) -- (6,6) -- cycle;
            \fill [bgd] (-6,-6) rectangle (6,6);
            \draw [omon,blue] \mypathc;
          \end{scope}
          \draw [omods=red] \mypatha;
          \draw [olax,purp] \mypathb;
        \end{scope}
        \draw [bound] (-6,-6) rectangle (6,6);
        \node [label,blue,above] at (3.5,6) {$\mT_2$};
        \node [label,purp,above] at (.125,6) {$\fF$};
        \node [label,blue,below] at (3.5,-6) {$\phantom{\mT_2}$};
        \node [label,purp,below] at (.125,-6) {$\phantom{\fF}$};
      \end{tikzpicture}
    }
  \]
  \afterimage

  
  \noindent
  (Likewise, according
  to~\cref{prop:formaltfae}~\labelcref{item:monadone}, any
  specialization $\sigma \colon \fF_1 \Rightarrow \fF_2$ induces a
  2-cell $\emlift{\sigma} \colon \emlift{\fF_1} \Rightarrow \emlift{\fF_2}$, and here the
  corresponding modification between module assignments is given by
  horizontal composition with $\sigma$ in $\EM(\tX)$.)

  In particular, $\mT_1$ is a $\mT_1$-module, and so
  $\mT_1 \then f$ is a $\mT_2$-module.
  \[\preq
    \hspace{-\leftmargin}
    \mathclap{
      \begin{tikzpicture}
        \node [ob] (s) at (-7,0) {$\oS_1$};
        \node [ob] (m) at (0,7) {$\oS_1$};
        \node [ob] (t) at (7,0) {$\oS_1$};
        \node [ob] (mt) at (10.5,7) {$\oS_2$};
        \node [ob] (tt) at (17.5,0) {$\oS_2$};
        \draw [a] (s) -- node[arr,al] {$\mT_1$} (m);
        \draw [a] (m) -- node[arr,arcl] {$\mT_1$} (t);
        \draw [a] (s) -- node[arr,below] {$\mT_1$} (t);
        \draw [a] (m) -- node[arr,above] {$\fF$} (mt);
        \draw [a] (mt) -- node[arr,ar] {$\mT_2$} (tt);
        \draw [a] (t) -- node[arr,below] {$\fF$} (tt);
        \node [cell] at (0,2.5) {$\mu$};
        \node [cell] at (8.75,3.5) {$\chi$};
      \end{tikzpicture}
    }
    \qqq
    \mathclap{
      \begin{tikzpicture}
        \begin{scope}
          \clip [boundc] (-6,-6) rectangle (6,6);
          \fill [bg] (-6,-6) rectangle (6,6);
          \renewcommand{\mypatha}{(-3.1,6) -- (-3.1,5) .. controls +(0,-3.5) and +(-1,1) .. (-.5,-2) -- (-.5,-6)}
          \renewcommand{\mypathb}{(.125,9) -- (.125,5.875) .. controls +(0,-6) and +(0,6) .. (2.875,-4.125) -- (2.875,-9)}
          \renewcommand{\mypathc}{(3.5,6) .. controls +(0,-4.5) and +(1,1) .. (-.5,-2)}
          \begin{scope}
            \clip \mypathb -- (-6,-6) -- (-6,6) -- cycle;
            \draw [omon,red] \mypathc;
          \end{scope}
          \begin{scope}
            \clip \mypathb -- (6,-6) -- (6,6) -- cycle;
            \fill [bgd] (-6,-6) rectangle (6,6);
            \draw [omon,blue] \mypathc;
          \end{scope}
          \draw [omon,red] \mypatha;
          \draw [olax,purp] \mypathb;
        \end{scope}
        \draw [bound] (-6,-6) rectangle (6,6);
      \end{tikzpicture}
    }
  \]
  \afterimage

  \noindent
  Incidentally, the first monad lax 1-cell law of \cref{def:monadone}
  says precisely that the structure 2-cell
  $\chi \colon f \then \mT_2 \Rightarrow \mT_1 \then f$ constitutes a
  $\mT_2$-module map.
\end{remark}

\begin{convention}\label{con:expand}
  In string diagrams, the identity from
  \cref{prop:formaltfae}~\labelcref{item:monadone}
  may be written in various forms as crossings of strings
  \[
    \,
    \begin{tikzpicture}[yscale=.875]
      \node [ob] (s) at (-7,0) {$\Alg{\mT_1}{\oS_1}$};
      \node [ob] (u) at (0,7) {$\oS_1$};
      \node [ob] (t) at (7,0) {$\oS_2$};
      \node [ob] (d) at (0,-7) {$\Alg{\mT_2}{\oS_2}$};
      \draw [a,long] (s) -- node[arr,al] {$\mcar{\mT_1}$} (u);
      \draw [a,long] (u) -- node[arr,ar] {$\fF$} (t);
      \draw [a,long] (s) -- node[arr,bl] {$\emlift{\fF}$} (d);
      \draw [a,long] (d) -- node[arr,br] {$\mcar{\mT_2}$} (t);
      \node at (0,0) {$=$};
    \end{tikzpicture}
    \eqspace{9pt}
    \begin{tikzpicture}[yscale=.875]
      \node [ob] (s) at (-7,0) {$\Alg{\mT_1}{\oS_1}$};
      \node [ob] (u) at (0,7) {$\oS_1$};
      \node [ob] (t) at (7,0) {$\oS_2$};
      \node [ob] (d) at (0,-7) {$\oS_1$};
      \draw [a,long] (s) -- node[arr,al] {$\mcar{\mT_1}$} (u);
      \draw [a,long] (u) -- node[arr,ar] {$\fF$} (t);
      \draw [a,long] (s) -- node[arr,bl] {$\mcar{\mT_1}$} (d);
      \draw [a,long] (d) -- node[arr,br] {$\fF$} (t);
      \node at (0,0) {$=$};
    \end{tikzpicture}
    \eqspace{9pt}
    \begin{tikzpicture}[yscale=-.875]
      \node [ob] (s) at (-7,0) {$\Alg{\mT_1}{\oS_1}$};
      \node [ob] (u) at (0,7) {$\oS_1$};
      \node [ob] (t) at (7,0) {$\oS_2$};
      \node [ob] (d) at (0,-7) {$\Alg{\mT_2}{\oS_2}$};
      \draw [a,long] (s) -- node[arr,bl] {$\mcar{\mT_1}$} (u);
      \draw [a,long] (u) -- node[arr,br] {$\fF$} (t);
      \draw [a,long] (s) -- node[arr,al] {$\emlift{\fF}$} (d);
      \draw [a,long] (d) -- node[arr,ar] {$\mcar{\mT_2}$} (t);
      \node at (0,0) {$=$};
    \end{tikzpicture}
    \eqspace{9pt}
    \begin{tikzpicture}[yscale=.875]
      \node [ob] (s) at (-7,0) {$\Alg{\mT_1}{\oS_1}$};
      \node [ob] (u) at (0,7) {$\Alg{\mT_2}{\oS_2}$};
      \node [ob] (t) at (7,0) {$\oS_2$};
      \node [ob] (d) at (0,-7) {$\Alg{\mT_2}{\oS_2}$};
      \draw [a,long] (s) -- node[arr,al] {$\emlift{\fF}$} (u);
      \draw [a,long] (u) -- node[arr,ar] {$\mcar{\mT_2}$} (t);
      \draw [a,long] (s) -- node[arr,bl] {$\emlift{\fF}$} (d);
      \draw [a,long] (d) -- node[arr,br] {$\mcar{\mT_2}$} (t);
      \node at (0,0) {$=$};
    \end{tikzpicture}
  \]
  \[
    \begin{tikzpicture}
      \begin{scope}
        \clip [boundc] (-8,-6) rectangle (8,6);
        \fill [bg] (-8,-6) rectangle (8,6);
        \renewcommand{\mypatha}{(3,7) -- (3,6) .. controls +(0,-3) and +(1,1) .. (0,0) .. controls +(-1,-1) and +(0,3) .. (-3,-6) -- (-3,-7)};
        \renewcommand{\mypathb}{(-3,6) .. controls +(0,-3) and +(-1,1) .. (0,0) .. controls +(1,-1) and +(0,3) .. (3,-6)};
        \fill [em=red] \mypathb -- (-8,-6) -- (-8,6) -- cycle;
        \draw [oem,red] \mypathb;
        \begin{scope}
          \clip \mypatha -- (8,-6) -- (8,6) -- cycle;
          \fill [bgd] (-8,-6) rectangle (8,6);
          \fill [em=blue] \mypathb -- (-8,-6) -- (-8,6) -- cycle;
          \draw [oem,blue] \mypathb;
        \end{scope}
        \draw [olax,purp] \mypatha;
      \end{scope}
      \draw [bound] (-8,-6) rectangle (8,6);
      \node [label,red,above] at (-3,6) {$\mcar{\mT_1}$};
      \node [label,blue,below] at (3,-6) {$\mcar{\mT_2}$};
      \node [label,purp,below] at (-3,-6) {$\emlift{\fF}$};
      \node [label,purp,above] at (3,6) {$\fF$};
    \end{tikzpicture}
    \eqspace{15pt}
    \begin{tikzpicture}
      \begin{scope}
        \clip [boundc] (-8,-6) rectangle (8,6);
        \fill [bg] (-8,-6) rectangle (8,6);
        \renewcommand{\mypatha}{(3,7) -- (3,-7)};
        \renewcommand{\mypathb}{(-3,6) -- (-3,-6)};
        \fill [em=red] \mypathb -- (-8,-6) -- (-8,6) -- cycle;
        \draw [oem,red] \mypathb;
        \begin{scope}
          \clip \mypatha -- (8,-6) -- (8,6) -- cycle;
          \fill [bgd] (-8,-6) rectangle (8,6);
          \fill [em=blue] \mypathb -- (-8,-6) -- (-8,6) -- cycle;
          \draw [oem,blue] \mypathb;
        \end{scope}
        \draw [olax,purp] \mypatha;
      \end{scope}
      \draw [bound] (-8,-6) rectangle (8,6);
    \end{tikzpicture}
    \eqspace{15pt}
    \begin{tikzpicture}
      \begin{scope}
        \clip [boundc] (-8,-6) rectangle (8,6);
        \fill [bg] (-8,-6) rectangle (8,6);
        \renewcommand{\mypatha}{(-3,7) -- (-3,6) .. controls +(0,-3) and +(-1,1) .. (0,0) .. controls +(1,-1) and +(0,3) .. (3,-6) -- (3,-7)};
        \renewcommand{\mypathb}{(3,6) .. controls +(0,-3) and +(1,1) .. (0,0) .. controls +(-1,-1) and +(0,3) .. (-3,-6)};
        \fill [em=red] \mypathb -- (-8,-6) -- (-8,6) -- cycle;
        \draw [oem,red] \mypathb;
        \begin{scope}
          \clip \mypatha -- (8,-6) -- (8,6) -- cycle;
          \fill [bgd] (-8,-6) rectangle (8,6);
          \fill [em=blue] \mypathb -- (-8,-6) -- (-8,6) -- cycle;
          \draw [oem,blue] \mypathb;
        \end{scope}
        \draw [olax,purp] \mypatha;
      \end{scope}
      \draw [bound] (-8,-6) rectangle (8,6);
    \end{tikzpicture}
    \eqspace{15pt}
    \begin{tikzpicture}
      \begin{scope}
        \clip [boundc] (-8,-6) rectangle (8,6);
        \fill [bg] (-8,-6) rectangle (8,6);
        \renewcommand{\mypatha}{(-3,7) -- (-3,-7)};
        \renewcommand{\mypathb}{(3,6) -- (3,-6)};
        \fill [em=red] \mypathb -- (-8,-6) -- (-8,6) -- cycle;
        \draw [oem,red] \mypathb;
        \begin{scope}
          \clip \mypatha -- (8,-6) -- (8,6) -- cycle;
          \fill [bgd] (-8,-6) rectangle (8,6);
          \fill [em=blue] \mypathb -- (-8,-6) -- (-8,6) -- cycle;
          \draw [oem,blue] \mypathb;
        \end{scope}
        \draw [olax,purp] \mypatha;
      \end{scope}
      \draw [bound] (-8,-6) rectangle (8,6);
    \end{tikzpicture}
  \]
  \afterimage

  \noindent
  The first of these has a mate
  \[\preq
    \mathclap{
      \begin{tikzpicture}[yscale=.875]
        \node [ob] (s) at (-7,0) {$\oS_1$};
        \node [ob] (u) at (0,7) {$\oS_2$};
        \node [ob] (t) at (7,0) {$\Alg{\mT_2}{\oS_2}$};
        \node [ob] (d) at (0,-7) {$\Alg{\mT_1}{\oS_1}$};
        \draw [a,long] (s) -- node[arr,al] {$\fF$} (u);
        \draw [a,long] (u) -- node[arr,ar] {$\mlcar{\mT_2}$} (t);
        \draw [a,long] (s) -- node[arr,bl] {$\mlcar{\mT_1}$} (d);
        \draw [a,long] (d) -- node[arr,br] {$\emlift{\fF}$} (t);
        \node at (0,0) {$(=)^*$};
      \end{tikzpicture}
    }
    \qqq
    \mathclap{
      \begin{tikzpicture}
        \begin{scope}
          \clip [boundc] (-8,-6) rectangle (8,6);
          \fill [bg] (-8,-6) rectangle (8,6);
          \renewcommand{\mypatha}{(-3,7) -- (-3,6) .. controls +(0,-3) and +(-1,1) .. (0,0) .. controls +(1,-1) and +(0,3) .. (3,-6) -- (3,-7)};
          \renewcommand{\mypathb}{(3,6) .. controls +(0,-3) and +(1,1) .. (0,0) .. controls +(-1,-1) and +(0,3) .. (-3,-6)};
          \fill [em=red] \mypathb -- (8,-6) -- (8,6) -- cycle;
          \draw [oem,red] \mypathb;
          \begin{scope}
            \clip \mypatha -- (8,-6) -- (8,6) -- cycle;
            \fill [bgd] (-8,-6) rectangle (8,6);
            \fill [em=blue] \mypathb -- (8,-6) -- (8,6) -- cycle;
            \draw [oem,blue] \mypathb;
          \end{scope}
          \draw [olax,purp] \mypatha;
        \end{scope}
        \draw [bound] (-8,-6) rectangle (8,6);
        \node [label,blue,above] at (3,6) {$\mlcar{\mT_2}$};
        \node [label,red,below] at (-3,-6) {$\mlcar{\mT_1}$};
      \end{tikzpicture}
    }
  \]
  
  \noindent
  Now mate to the definition of $\emlift{\fF}$ is
  \[\preq
    \mathclap{
      \begin{tikzpicture}[yscale=.875,xscale=.875]
        \node [ob] (s) at (-7,0) {$\oS_1$};
        \node [ob] (t) at (7,0) {$\oS_2$};
        \node [ob] (u) at (0,7) {$\oS_2$};
        \node [ob] (d) at (0,-7) {$\oS_1$};
        \draw [a] (s) -- node[arr,al] {$\fF$} (u);
        \draw [a] (u) -- node[arr,ar] {$\mT_2$} (t);
        \draw [a] (s) -- node[arr,bl] {$\mT_1$} (d);
        \draw [a] (d) -- node[arr,br] {$\fF$} (t);
        \node [cell] at (0,0) {$\chi$};
      \end{tikzpicture}
      \eqq
      \begin{tikzpicture}[yscale=.875]
        \node [ob] (s) at (-7,0) {$\oS_1$};
        \node [ob] (u) at (0,7) {$\oS_2$};
        \node [ob] (t) at (7,0) {$\Alg{\mT_2}{\oS_2}$};
        \node [ob] (d) at (0,-7) {$\Alg{\mT_1}{\oS_1}$};
        \node [ob] (tt) at (14,-7) {$\oS_2$};
        \node [ob] (dt) at (7,-14) {$\oS_1$};
        \draw [a,long] (s) -- node[arr,al] {$\fF$} (u);
        \draw [a,long] (u) -- node[arr,ar] {$\mlcar{\mT_2}$} (t);
        \draw [a,long] (s) -- node[arr,bl] {$\mlcar{\mT_1}$} (d);
        \draw [a,long] (d) -- node[arr,brcl] {$\emlift{\fF}$} (t);
        \draw [a,long] (t) -- node[arr,ar] {$\mcar{\mT_2}$} (tt);
        \draw [a,long] (d) -- node[arr,bl] {$\mcar{\mT_1}$} (dt);
        \draw [a,long] (dt) -- node[arr,br] {$\fF$} (tt);
        \node at (0,0) {$(=)^*$};
        \node at (7,-7) {$=$};
      \end{tikzpicture}
    }
    \qqq
    \mathclap{
      \begin{tikzpicture}[scale=1.3]
        \begin{scope}
          \clip [boundc] (-6,-6) rectangle (6,6);
          \fill [bg] (-6,-6) rectangle (6,6);
          \renewcommand{\mypath}{(-3,6) .. controls +(0,-3) and +(-1,1) .. (0,0) .. controls +(1,-1) and +(0,3) .. (3,-6)};
          \fill [bgd] \mypath -- (6,-6) -- (6,6) -- cycle;
          \draw [omon,blue] (3,6) .. controls +(0,-3) and +(1,1) .. (0,0);
          \draw [omon,red] (0,0) .. controls +(-1,-1) and +(0,3) .. (-3,-6);
          \draw [olax,purp] (-3,7) -- \mypath -- (3,-8);
        \end{scope}
        \draw [bound] (-6,-6) rectangle (6,6);
        \node [label,blue,above] at (3,6) {$\mT_2$};
        \node [label,red,below] at (-3,-6) {$\mT_1$};
      \end{tikzpicture}
      \eqq
      \begin{tikzpicture}[scale=1.3]
        \begin{scope}
          \clip [boundc] (-7,-6) rectangle (7,6);
          \fill [bg] (-7,-6) rectangle (7,6);
          \renewcommand{\mypatha}{(-3,7) -- (-3,6) .. controls +(0,-3) and +(-1,1) .. (0,0) .. controls +(1,-1) and +(0,3) .. (3,-6) -- (3,-7)};
          \renewcommand{\mypathb}{(4.5,6) -- (4.5,5) .. controls +(0,-5) and +(0,5) .. (-1.5,-6)};
          \renewcommand{\mypathc}{(-4.5,-6) -- (-4.5,-5) .. controls +(0,5) and +(0,-5) .. (1.5,6)};
          \fill [em=red] \mypathb -- \mypathc -- cycle;
          \draw [oem,red] \mypathb;
          \draw [oem,red] \mypathc;
          \begin{scope}
            \clip \mypatha -- (8,-6) -- (8,6) -- cycle;
            \fill [bgd] (-8,-6) rectangle (8,6);
            \fill [em=blue] \mypathb -- \mypathc -- cycle;
            \draw [oem,blue] \mypathb;
            \draw [oem,blue] \mypathc;
          \end{scope}
          \draw [olax,purp] \mypatha;
        \end{scope}
        \draw [bound] (-7,-6) rectangle (7,6);
      \end{tikzpicture}
    }
  \]
  
  \noindent
  Likewise for specializations we have
  \[\preq
    \mathclap{
      \begin{tikzpicture}
        \node [ob] (s) at (-7,0) {$\oS_1$};
        \node [ob] (m) at (0,-7) {$\oS_1$};
        \node [ob] (t) at (7,0) {$\oS_2$};
        \draw [a] (s) -- node[arr,above] {$\fF_1$} (t);
        \draw [a] (s) -- node[arr,bl] {$\mT_1$} (m);
        \draw [a] (m) -- node[arr,br] {$\fF_2$} (t);
        \node [cell] at (0,-2.5) {$\sigma$};
      \end{tikzpicture}
      \eqq
      \begin{tikzpicture}
        \node [ob] (ss) at (-21,7) {$\oS_1$};
        \node [ob] (s) at (-14,0) {$\Alg{\mT_1}{\oS_1}$};
        \node [ob] (t) at (0,0) {$\Alg{\mT_2}{\oS_2}$};
        \node [ob] (d) at (0,-7) {$\oS_1$};
        \node [ob] (u) at (0,7) {$\oS_1$};
        \node [ob] (tt) at (10.5,0) {$\oS_2$};
        \draw [eq] (ss) .. controls +(7,4) and +(-7,4) .. (u);
        \draw [a] (ss) -- node[arr,bl] {$\mlcar{\mT_1}$} (s);
        \draw [a] (s) .. controls +(4,6) and +(-5,0) .. node[arr,alcl] {$\mcar{\mT_1}$} (u);
        \draw [a] (s) .. controls +(4,-6) and +(-5,0) .. node[arr,bl] {$\mcar{\mT_1}$} (d);
        \draw [a] (u) -- node[arr,ar] {$\fF_1$} (tt);
        \draw [a] (d) -- node[arr,br] {$\fF_2$} (tt);
        \draw [a] (t) -- node[arr,acl,pos=.4] {$\mcar{\mT_2}$} (tt);
        \draw [a] (s) .. controls +(5,3) and +(-5,3) .. node[arr,acl,pos=.8] {$\emlift{\fF_1}$} (t);
        \draw [a] (s) .. controls +(5,-3) and +(-5,-3) .. node[arr,bcl,pos=.8] {$\emlift{\fF_2}$} (t);
        \node [cell] at (-7,0) {$\emlift{\sigma}$};
        \node [cell] at (-14.25,6) {$\eta$};
        \node [cell] at (0,3.5) {$=$};
        \node [cell] at (0,-4) {$=$};
      \end{tikzpicture}
    }
    \qqq
    \mathclap{
      \begin{tikzpicture}[scale=1.3]
        \begin{scope}
          \clip [boundc] (-6,-6) rectangle (6,6);
          \fill [bg] (-6,-6) rectangle (6,6);
          \node [tspech,red] at (0,0) {};
          \draw [omon,red] (-.5,0) .. controls +(-1,-1) and +(0,3.5) .. (-3.5,-6);
          \begin{scope}
            \clip (0,7) -- (0,-7) -- (6,-6) -- (6,6) -- cycle;
            \fill [bgd] (-6,-6) rectangle (6,6);
            \node [tspech,blue] at (0,0) {};
          \end{scope}
          \draw [olax,purp] (0,7) -- (0,0);
          \draw [olax,dpurp] (0,0) -- (0,-8);
        \end{scope}
        \draw [bound] (-6,-6) rectangle (6,6);
        \node [tspec] at (0,0) {};
        \node [label,purp,above] at (0,6) {$\fF_1$};
        \node [label,dpurp,below] at (0,-6) {$\fF_2$};
        \node [label,above=4.5,left=5] at (0,0) {$\sigma$};
      \end{tikzpicture}
      \eqq
      \begin{tikzpicture}[scale=1.3]
        \begin{scope}
          \clip [boundc] (-8,-6) rectangle (6,6);
          \fill [bg] (-8,-6) rectangle (6,6);
          \renewcommand{\mypatha}{};
          \renewcommand{\mypathb}{(-5.5,-6) -- (-5.5,-1) .. controls +(0,8) and +(0,3.5) .. (2.75,0) .. controls +(0,-3.5) and +(0,3) .. (-2.5,-6)};
          \fill [em=red] \mypathb -- cycle;
          \draw [oem,red] \mypathb;
          \begin{scope}
            \clip (0,7) -- (0,-7) -- (8,-6) -- (8,6) -- cycle;
            \fill [bgd] (-8,-6) rectangle (8,6);
            \fill [em=blue] \mypathb -- cycle;
            \draw [oem,blue] \mypathb;
          \end{scope}
          \draw [olax,purp] (0,7) -- (0,0);
          \draw [olax,dpurp] (0,0) -- (0,-8);
          \node [tspec] at (0,0) {};
        \end{scope}
        \draw [bound] (-8,-6) rectangle (6,6);
      \end{tikzpicture}
    }
  \]
  
  In summary, the results accumulated thus far allow for the expansion
  of monads $\mT$ admitting algebra objects in string diagrams as
  regions of $\Alg{\mT}{\oS}$ bordered by $\mlcar{\mT}$ and
  $\mcar{\mT}$. This applies to not only the monads themselves, but
  also their crossings through monad lax 1-cells, as well as the monad
  films drawn on modules, module maps, and specializations.
\end{convention}

\pagebreak

\chapter{Double categories}\label{sec:doublemonads}

\begin{definition}
  A \emph{monad colax 1-cell} is defined by flipping all of the
  diagrams in \cref{def:monadone} horizontally.\vspace{-7.5pt}
  \[\preq
    \mathclap{\fG \colon \oS_1 \to \oS_2}
    \qqq
    \mathclap{
      \begin{tikzpicture}[xscale=-1,yscale=.5]
        \begin{scope}
          \clip [boundc] (-6,-6) rectangle (6,6);
          \fill [bg] (-6,-6) rectangle (6,6);
          \fill [bgd] (-6,-6) rectangle (0,6);
          \draw [oclax,orange] (0,7) -- (0,-9);
        \end{scope}
        \draw [bound] (-6,-6) rectangle (6,6);
        \node [label] at (-3,0) {$\oS_2$};
        \node [label] at (3,0) {$\oS_1$};
        \node [label,orange,above] at (0,6) {$\fG$};
        \node [label,below] at (0,-6) {$\phantom{\fG}$};
      \end{tikzpicture}
    }
  \]
  \vspace{-15pt}
  \[\preq
    \mathclap{
      \begin{tikzpicture}[yscale=.875,xscale=-.875]
        \node [ob] (s) at (-7,0) {$\oS_2$};
        \node [ob] (t) at (7,0) {$\oS_1$};
        \node [ob] (u) at (0,7) {$\oS_1$};
        \node [ob] (d) at (0,-7) {$\oS_2$};
        \draw [ab] (s) -- node[arr,ar] {$\fG$} (u);
        \draw [ab] (u) -- node[arr,al] {$\mT_1$} (t);
        \draw [ab] (s) -- node[arr,br] {$\mT_2$} (d);
        \draw [ab] (d) -- node[arr,bl] {$\fG$} (t);
        \node [cell] at (0,0) {$\xi$};
      \end{tikzpicture}
    }
    \qqq
    \mathclap{
      \begin{tikzpicture}[xscale=-1]
        \begin{scope}
          \clip [boundc] (-6,-6) rectangle (6,6);
          \fill [bg] (-6,-6) rectangle (6,6);
          \renewcommand{\mypath}{(-3,6) .. controls +(0,-3) and +(-1,1) .. (0,0) .. controls +(1,-1) and +(0,3) .. (3,-6)};
          \fill [bgd] \mypath -- (-6,-6) -- (-6,6) -- cycle;
          \draw [omon,red] (3,6) .. controls +(0,-3) and +(1,1) .. (0,0);
          \draw [omon,blue] (0,0) .. controls +(-1,-1) and +(0,3) .. (-3,-6);
          \draw [oclax,orange] (-3,7) -- \mypath -- (3,-8);
        \end{scope}
        \draw [bound] (-6,-6) rectangle (6,6);
        \node [label,red,above] at (3,6) {$\mT_1$};
        \node [label,blue,below] at (-3,-6) {$\mT_2$};
      \end{tikzpicture}
    }
  \]
  \vspace{-5pt}
  \[\preq
    \mathclap{
      \begin{tikzpicture}[scale=.825,xscale=-1]
        \node [ob] (s) at (-7,0) {$\oS_2$};
        \node [ob] (t) at (10.5,3.5) {$\oS_1$};
        \node [ob] (u) at (0,7) {$\oS_1$};
        \node [ob] (d) at (3.5,-3.5) {$\oS_2$};
        \node [ob] (tt) at (14,-7) {$\oS_1$};
        \node [ob] (dd) at (7,-14) {$\oS_2$};
        \draw [ab] (s) -- node[arr,ar] {$\fG$} (u);
        \draw [ab] (u) -- node[arr,above] {$\mT_1$} (t);
        \draw [ab] (s) -- node[arr,below,pos=.3] {$\mT_2$} (d);
        \draw [ab] (d) -- node[arr,bl] {$\fG$} (t);
        \draw [ab] (t) -- node[arr,left] {$\mT_1$} (tt);
        \draw [ab] (d) -- node[arr,right,pos=.7] {$\mT_2$} (dd);
        \draw [ab] (dd) -- node[arr,bl] {$\fG$} (tt);
        \draw [ab] (s) .. controls +(2,-8) and +(-8,2) .. node[arr,br] {$\mT_2$} (dd);
        \node [cell] at (1.75,1.75) {$\xi$};
        \node [cell] at (8.75,-5.25) {$\xi$};
        \node [cell] at (0,-7) {$\mu$};
      \end{tikzpicture}
      \eqq
      \begin{tikzpicture}[scale=.825,xscale=-1]
        \node [ob] (s) at (-7,0) {$\oS_2$};
        \node [ob] (t) at (10.5,3.5) {$\oS_1$};
        \node [ob] (u) at (0,7) {$\oS_1$};
        \node [ob] (tt) at (14,-7) {$\oS_1$};
        \node [ob] (dd) at (7,-14) {$\oS_2$};
        \draw [ab] (s) -- node[arr,ar] {$\fG$} (u);
        \draw [ab] (u) -- node[arr,above] {$\mT_1$} (t);
        \draw [ab] (t) -- node[arr,left] {$\mT_1$} (tt);
        \draw [ab] (dd) -- node[arr,bl] {$\fG$} (tt);
        \draw [ab] (u) .. controls +(2,-8) and +(-8,2) .. node[arr,br,pos=.8] {$\mT_1$} (tt);
        \draw [ab] (s) .. controls +(2,-8) and +(-8,2) .. node[arr,br] {$\mT_2$} (dd);
        \node [cell] at (.5,-5) {$\xi$};
        \node [cell] at (7,0) {$\mu$};
      \end{tikzpicture}
    }
    \qqq
    \mathclap{
      \begin{tikzpicture}[xscale=-.7894,yscale=.7]
        \begin{scope}
          \clip [boundc] (-10,-14) rectangle (9,10);
          \fill [bg] (-10,-14) rectangle (9,10);
          \renewcommand{\mypath}{(-5,10) -- (-5,9) .. controls +(0,-3) and +(-1,1) .. (-2,3) -- (1,0) .. controls +(2,-2) and +(0,3) .. (4,-8) -- (4,-14)};
          \fill [bgd] \mypath -- (-10,-14) -- (-10,10) -- cycle;
          \draw [omon,red] (1,10) -- (1,9) .. controls +(0,-3) and +(1,1) .. (-2,3);
          \draw [omon,red] (5.5,10) -- (5.5,9) .. controls +(0,-5) and +(1,1) .. (1,0);
          \draw [omon,blue] (-2,3) .. controls +(-1,-1) and +(0,5) .. (-6.5,-6) .. controls +(0,-3) and +(-1,1) .. (-4.25,-11);
          \draw [omon,blue] (1,0) .. controls +(-1,-1) and +(0,3) .. (-2,-6) .. controls +(0,-3) and +(1,1) .. (-4.25,-11);
          \draw [omon,blue] (-4.25,-11) -- (-4.25,-14);
          \node [tmon,blue] at (-4.25,-11) {};
          \draw [oclax,orange] (-5,11) -- \mypath -- (5,-16);
        \end{scope}
        \draw [bound] (-10,-14) rectangle (9,10);
      \end{tikzpicture}
      \eqq
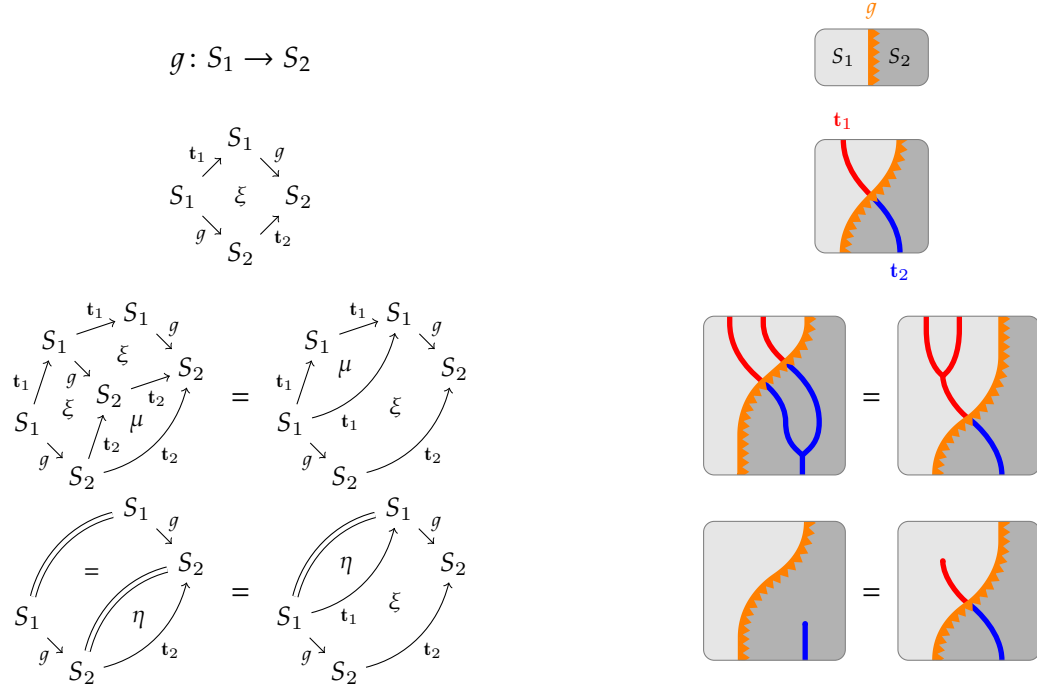
      \begin{tikzpicture}[xscale=-1,yscale=.7]
        \begin{scope}
          \clip [boundc] (-7.5,-13) rectangle (7.5,11);
          \fill [bg] (-7.5,-14) rectangle (7.5,11);
          \renewcommand{\mypath}{(-3.5,11) -- (-3.5,5) .. controls +(0,-5) and +(-1,1) .. (0,-4.5) .. controls +(1,-1) and +(0,5) .. (3.5,-13) -- (3.5,-14)};
          \fill [bgd] \mypath -- (-7.5,-14) -- (-7.5,11) -- cycle;
          \draw [omon,red] (1,11) -- (1,9) .. controls +(0,-3.5) and +(-1,1) .. (2.75,2);
          \draw [omon,red] (4.5,11) --(4.5,9) .. controls +(0,-3.5) and +(1,1) .. (2.75,2);
          \draw [omon,red] (2.75,2) .. controls +(0,-2.5) and +(1,1) .. (0,-4.5);
          \node [tmon,red] at (2.75,2) {};
          \draw [omon,blue] (0,-4.5) .. controls +(-1,-1) and +(0,4) .. (-3.5,-13) -- (-3.5,-14);
          \draw [oclax,orange] (-3.5,12) -- \mypath -- (3.5,-16);
        \end{scope}
        \draw [bound] (-7.5,-13) rectangle (7.5,11);
      \end{tikzpicture}
    }
  \]
  \vspace{-10pt}
  \[\preq
    \mathclap{
      \begin{tikzpicture}[scale=.825,xscale=-1]
        \node [ob] (s) at (-7,0) {$\oS_2$};
        \node [ob] (u) at (0,7) {$\oS_1$};
        \node [ob] at (0,7) {$\phantom{\oS_1}$};
        \node [ob] (tt) at (14,-7) {$\oS_1$};
        \node [ob] (dd) at (7,-14) {$\oS_2$};
        \draw [ab] (s) -- node[arr,ar] {$\fG$} (u);
        \draw [ab] (dd) -- node[arr,bl] {$\fG$} (tt);
        \draw [ab] (s) .. controls +(2,-8) and +(-8,2) .. node[arr,br] {$\mT_2$} (dd);
        \draw [eq] (u) .. controls +(8,-2) and +(-2,8) .. (tt);
        \draw [eq] (s) .. controls +(8,-2) and +(-2,8) .. (dd);
        \node [cell] at (5.75,-1.25) {$=$};
        \node [cell] at (-.25,-7.25) {$\eta$};
      \end{tikzpicture}
      \eqq
      \begin{tikzpicture}[scale=.825,xscale=-1]
        \node [ob] (s) at (-7,0) {$\oS_2$};
        \node [ob] (u) at (0,7) {$\oS_1$};
        \node [ob] (tt) at (14,-7) {$\oS_1$};
        \node [ob] (dd) at (7,-14) {$\oS_2$};
        \draw [ab] (s) -- node[arr,ar] {$\fG$} (u);
        \draw [ab] (dd) -- node[arr,bl] {$\fG$} (tt);
        \draw [ab] (u) .. controls +(2,-8) and +(-8,2) .. node[arr,br,pos=.8] {$\mT_1$} (tt);
        \draw [eq] (u) .. controls +(8,-2) and +(-2,8) .. (tt);
        \draw [ab] (s) .. controls +(2,-8) and +(-8,2) .. node[arr,br] {$\mT_2$} (dd);
        \node [cell] at (.5,-5) {$\xi$};
        \node [cell] at (6.75,-.25) {$\eta$};
      \end{tikzpicture}
    }
    \qqq
    \mathclap{
      \begin{tikzpicture}[xscale=-1,yscale=.7]
        \begin{scope}
          \clip [boundc] (-7.5,-13) rectangle (7.5,8);
          \fill [bg] (-7.5,-14) rectangle (7.5,11);
          \renewcommand{\mypath}{(-3.5,11) -- (-3.5,8) .. controls +(0,-5) and +(-1,1) .. (0,-.5) .. controls +(1,-1) and +(0,5) .. (3.5,-9.5) -- (3.5,-14)};
          \fill [bgd] \mypath -- (-7.5,-14) -- (-7.5,11) -- cycle;
          \draw [omon,blue] (-3.25,-7.5) -- (-3.25,-14);
          \node [tmon,blue] at (-3.25,-7.5) {};
          \draw [oclax,orange] (-3.5,12) -- \mypath -- (3.5,-16);
        \end{scope}
        \draw [bound] (-7.5,-13) rectangle (7.5,8);
      \end{tikzpicture}
      \eqq
      \begin{tikzpicture}[xscale=-1,yscale=.7]
        \begin{scope}
          \clip [boundc] (-7.5,-13) rectangle (7.5,8);
          \fill [bg] (-7.5,-14) rectangle (7.5,11);
          \renewcommand{\mypath}{(-3.5,11) -- (-3.5,5) .. controls +(0,-5) and +(-1,1) .. (0,-4.5) .. controls +(1,-1) and +(0,5) .. (3.5,-13) -- (3.5,-14)};
          \fill [bgd] \mypath -- (-7.5,-14) -- (-7.5,11) -- cycle;
          \draw [omon,red] (2.75,2) .. controls +(0,-2.5) and +(1,1) .. (0,-4.5);
          \node [tmon,red] at (2.75,2) {};
          \draw [omon,blue] (0,-4.5) .. controls +(-1,-1) and +(0,4) .. (-3.5,-13) -- (-3.5,-14);
          \draw [oclax,orange] (-3.5,12) -- \mypath -- (3.5,-16);
        \end{scope}
        \draw [bound] (-7.5,-13) rectangle (7.5,8);
      \end{tikzpicture}
    }
    \afterimage
    \vspace{2.5pt}
  \]
  
  That is, a monad colax 1-cell between monads $\mT_1$
  and $\mT_2$ is a monad lax 1-cell in $\tX\op$ between the
  monads therein corresponding to $\mT_2$ and $\mT_1$.
\end{definition}

Likewise there are notions of monad 2-cell and monad specialization
between monad colax 1-cells, obtained by flipping the ones for monad
lax 1-cells. But there are also more general notions of monad 2-cell
and monad specialization that involve both lax and colax 1-cells at
once.

\begin{definition}\label{def:doublemon}
  A \emph{monad 2-cell}\footnotemark\ from lax $\fF_1 \colon \mT_\dW \to \mT_\dN$
  and colax $\fG_1 \colon \mT_\dN \to \mT_\dE$ to colax
  $\fG_2 \colon \mT_\dW \to \mT_\dS$ and lax
  $\fF_2 \colon \mT_\dS \to \mT_\dE$ is a 2-cell \vspace{-25pt}
  \[\preq
    \mathclap{

    }
  \]
  \afterimage
\end{definition}
\vspace{-10pt}

\footnotetext{These definitions of monad 2-cell and monad
  specialization make sense in an implicit 2-category.}

Note that this recovers \cref{def:monadone} when the colax 1-cells are
identities.

\begin{proposition}\label{prop:doublemonads}
  Let $\tX$ be a 2-category.
  \begin{enumerate}[label=(\roman*)]
  \item\label{item:basicdouble} There is a double category
    $\MMnd(\tX)$ in which the 0-cells are monads, the 1-cells are lax
    and colax 1-cells, and the 2-cells\footnote{As with 2-categories,
      we use \emph{2-cell} to mean a two-dimensional cell in a double
      category. Some authors call this a \emph{double cell} or simply
      a \emph{cell}.} are monad 2-cells.
  \item\label{item:specdouble}
    There is a double category $\EEM(\tX)$ in which the 0-cells are
    monads, the 1-cells are lax and colax 1-cells, and the 2-cells are
    monad specializations.
  \end{enumerate}
\end{proposition}
\begin{proof}
  For \labelcref{item:basicdouble}, all compositions are given by
  composing underlying cells in the 2-category $\tX$. The
  structure 2-cells of composite monad lax 1-cells are
  given by
  
  \[\preq
    \mathclap{

    }
    \afterimage
  \]

  \noindent
  and the structure 2-cells of identities are identities. Colax
  1-cells are dual. It is straightforwardly checked that compositions
  of (co)lax 1-cells are again (co)lax 1-cells, and
  likewise for monad 2-cells. The various associativity and unit laws
  in $\MMnd(\tX)$ follow from those in $\tX$.

  \pagebreak
  
  For \labelcref{item:specdouble}, the 1-cell structure is identical
  to that of \labelcref{item:basicdouble}. Composite
  specializations are given in the lax direction by
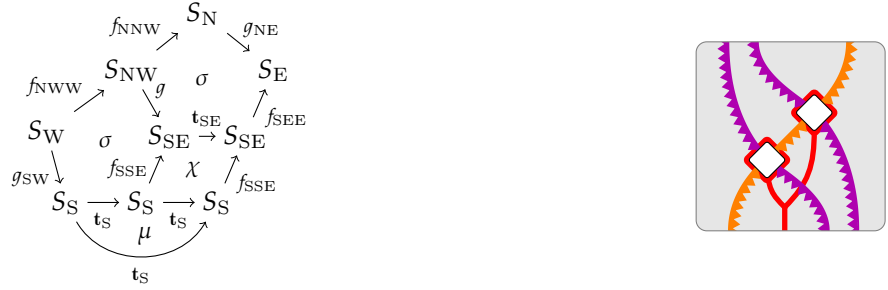
\vspace{-5pt}
  
  \[\preq
    \mathclap{
      %
    }
  \]

  \noindent
  and identity specializations are obtained from identity 2-cells by
  composing with monad units as in \cref{rem:tospec}. The following
  string diagram manipulation witnesses that composite specializations
  are specializations (uses of associativity omitted):
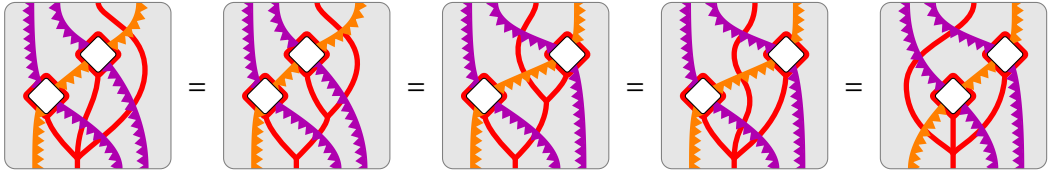
\vspace{-6pt}
  
  \[
    %
    \afterimage
    \vspace{6.25pt}
  \]

  \noindent
  The colax direction is dual. Associativity and unit laws for
  specializations also follow straightforwardly from those for monads
  and the (co)lax 1-cell laws, and the following string diagram
  manipulation witnesses the interchange law (uses of associativity
  omitted):
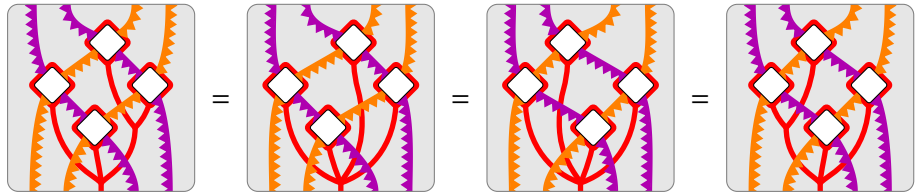
\vspace{-6pt}
  
  \[
    %
  \]
\end{proof}

\begin{remark}
  There is an identity-on-1-cells functor $\MMnd(\tX) \to \EEM(\tX)$
  given by 
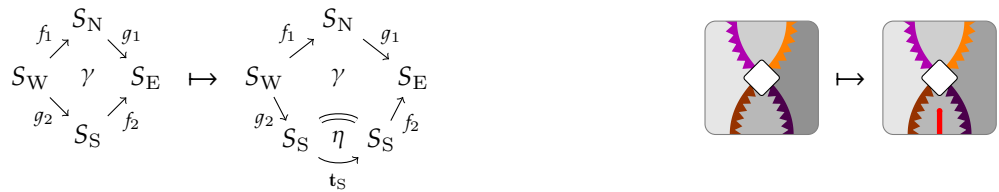
\vspace{-27pt}
  \[\preq
    \mathclap{
      %
    }
  \]
  \vspace{-22.5pt}
  
  \noindent
  In fact, an arbitrary 2-cell
  $\gamma \colon \fF_1 \then \fG_1 \Rightarrow \fG_2 \then \fF_2$ is a
  monad 2-cell if and only if its composite with the unit $\eta$, as
  shown above, is a monad specialization. Thus monad 2-cells are
  equivalently monad specializations equipped with the data of a
  factoring through the unit.
\end{remark}

\begin{remark}
  If the carrier of a monad lax 1-cell has a left adjoint 1-cell, then
  this left adjoint has the structure of a monad colax 1-cell.
  \[\preq \mathclap{
      %
    }
    \vspace{-5pt}
  \]
  \afterimage
  
  \noindent
  Dually, if the carrier of a monad colax 1-cell has a right
  adjoint 1-cell, then the right adjoint has the structure of a monad lax
  1-cell.
  \[\preq
    \mathclap{
      %
    }
    \vspace{-5pt}
  \]
  \afterimage
  
  Now consider lax $\fF_1$ and $\fF_2$ and colax $\fG_1$ and $\fG_2$
  as in \cref{def:doublemon}. If $\fG_1$ and $\fG_2$ have right
  adjoints $\fR_1$ and $\fR_2$, then a monad 2-cell or specialization
  from $\fF_1$ and $\fG_1$ to $\fG_2$ and $\fF_2$ is equivalent to a
  monad 2-cell or specialization from $\fR_2 \then \fF_1$ to
  $\fF_2 \then \fR_1$, which involves only monad lax 1-cells. (Dually,
  if $\fF_1$ and $\fF_2$ have left adjoints $\fL_1$ and $\fL_2$, then
  a monad 2-cell or specialization from $\fF_1$ and $\fG_1$ to $\fG_2$
  and $\fF_2$ is equivalent to a monad 2-cell or specialization from
  $\fG_1 \then \fL_2$ to $\fL_1 \then \fG_2$, which involves only
  colax 1-cells.) Thus the more general concepts of monad 2-cell and
  specialization from \cref{def:doublemon} reduce to the concepts from
  \cref{def:monadone} in the special case that appropriate adjoints
  exist.
\end{remark}

\begin{remark}\label{rem:generaldouble}
  Just as monad lax 1-cells and their monad 2-cells are instances of
  lax transformations and their modifications as discussed in
  \cref{rem:transfors}, monad colax 1-cells and their monad 2-cells
  are instances of colax transformations and their modifications. The
  more general monad 2-cells involving both lax and colax 1-cells are
  instances of more general modifications involving both lax and colax
  transformations. Writing $\tDelta$ for the free 2-category
  containing a monad on an object, the double category $\MMnd(\tX)$ is
  precisely the double category $\HHom(\tDelta, \tX)$ of functors, lax
  and colax transformations, and modifications defined
  in~\cite[Proposition A.8]{fairbanks-shulman}.
  
  The construction of $\HHom(\tX, \tY)$ actually makes sense in the
  generality that $\tX$ and $\tY$ are weak 2-categories (a.k.a.\
  bicategories), in which case we obtain a \emph{doubly weak double
    category}~\cite[Definition 3.3]{fairbanks-shulman} rather than a
  strict double category.\footnote{In fact, the construction of
    $\HHom(\tX, \tY)$ in~\cite[Proposition A.8]{fairbanks-shulman} is
    stated in the generality that $\tX$ and $\tY$ are implicit
    2-categories. An implicit 2-category is more general than a weak
    2-category, but simultaneously is an instance of a strict
    2-category. Through this framework, various constructions that
    make sense in strict 2-categories make sense without alteration in
    weak 2-categories.} Likewise, the above constructions of
  $\MMnd(\tX)$ and $\EEM(\tX)$ actually make sense in the same
  generality. In particular, for $\tX$ a weak 2-category, $\EEM(\tX)$
  is a novel example of a doubly weak double category.
  
  Some general facts about $\HHom(\tX, \tY)$ are as follows: the
  conjoint pairs\footnote{\emph{Companions} and \emph{conjoints} are
    double-categorical analogues of
    adjoints~\cite{grandis-pare,shulman:frbi}.} are the mate pair lax
  and colax transformations, and the companion pairs are (up to
  isomorphism) the componentwise inverse lax and colax transformations
  (a.k.a.\ \emph{pseudonatural transformations})~\cite[Remark
  A.9]{fairbanks-shulman}. In particular this yields descriptions of
  companions and conjoints in $\MMnd(\tX) = \HHom(\tDelta, \tX)$.
\end{remark}



\begin{remark}\label{rem:clarke-dimeglio}
  The double category $\MMnd(\tX)$ for a 2-category $\tX$ can also be
  seen as a special case of the double category $\MMnd(\DD)$ for a
  double category $\DD$ from~\cite[Definition
  2.4]{fiore-gambino-kock}, in which objects are loose monads in
  $\DD$, loose 1-cells are lax 1-cells in the underlying loose (weak)
  2-category $\Loose(\DD)$, tight 1-cells are \emph{monad tight
    1-cells}, and 2-cells satisfy conditions mate\footnote{By this we
    mean conditions obtained from those of \cref{def:doublemon} by
    replacing the top right and bottom left 1-cells with tight
    1-cells, as one does when ``bending'' loose 1-cells into their
    companion tight 1-cells according to the mate correspondence for
    companions in double categories.}  to those
  from~\cref{def:doublemon} for monad 2-cells.  \vspace{-15pt}
  \[\preq
    \mathclap{
      \begin{tikzpicture}[scale=.875]
        \path (0,18) -- (0,-18);
        \node [ob] (s) at (-7,7) {$\oS_\dNW$};
        \node [ob] (t) at (7,-7) {$\oS_\dSE$};
        \node [ob] (u) at (7,7) {$\oS_\dNE$};
        \node [ob] (d) at (-7,-7) {$\oS_\dSW$};
        \draw [a] (s) -- node[arr,above] {$\fF_1$} (u);
        \draw [a] (u) -- node[arr,right] {$\fG_1$} (t);
        \draw [a] (s) -- node[arr,left] {$\fG_2$} (d);
        \draw [a] (d) -- node[arr,below] {$\fF_2$} (t);
        \node [cell] at (0,0) {$\gamma$};
      \end{tikzpicture}
    }
    \qqq
    \mathclap{
      \begin{tikzpicture}[scale=1.5]
        \begin{scope}
          \clip [boundc] (-6,-6) rectangle (6,6);
          \fill [bg] (-6,-6) rectangle (6,6);
          \renewcommand{\mypatha}{(-2,7) -- (-2,6) .. controls +(0,-3) and +(-1,1) .. (0,0)};
          \renewcommand{\mypathb}{(0,0) .. controls +(1,-1) and +(0,3) .. (2,-6) -- (2,-7)};
          \renewcommand{\mypathc}{(7,0) -- (0,0)};
          \renewcommand{\mypathd}{(0,0) -- (-7,0)};
          \begin{scope}
            \clip \mypatha -- \mypathb -- (-6,-6) -- (-6,6) -- cycle;
            \begin{scope}
              \clip \mypathc -- \mypathd -- (-6,6) -- (6,6) -- cycle;
              \fill [wback] (-6,-6) rectangle (6,6);
            \end{scope}
            \begin{scope}
              \clip \mypathc -- \mypathd -- (-6,6) -- (6,6) -- cycle;
              \fill [sback] (-6,-6) rectangle (6,6);
            \end{scope}
          \end{scope}
          \begin{scope}
            \clip \mypatha -- \mypathb -- (6,-6) -- (6,6) -- cycle;
            \begin{scope}
              \clip \mypathc -- \mypathd -- (-6,6) -- (6,6) -- cycle;
              \fill [nback] (-6,-6) rectangle (6,6);
            \end{scope}
            \begin{scope}
              \clip \mypathc -- \mypathd -- (-6,-6) -- (6,-6) -- cycle;
              \fill [eback] (-6,-6) rectangle (6,6);
            \end{scope}
          \end{scope}
          \draw [olax,purp] \mypatha;
          \draw [olax,dpurp] \mypathb;
          \draw [oclax,orange] \mypathc;
          \draw [oclax,dorange] \mypathd;
          \node [tdmodis] at (0,0) {};
        \end{scope}
        \draw [bound] (-6,-6) rectangle (6,6);
        \node [label] at (2.5,3.25) {$\oS_{\dNE}$};
        \node [label] at (3.85,-2) {$\oS_{\dSE}$};
        \node [label] at (-3.85,2) {$\oS_{\dNW}$};
        \node [label] at (-2.5,-3.25) {$\oS_{\dSW}$};
        \node [label,purp,above] at (-2,6) {$\fF_1$};
        \node [label,orange,right] at (6,0) {$\fG_1$};
        \node [label,dorange,left] at (-6,0) {$\fG_2$};
        \node [label,dpurp,below] at (2,-6) {$\fF_2$};
        \node [label] at (0,0) {$\gamma$};
      \end{tikzpicture}
    }
  \]
  \vspace{-20pt}
  
  \noindent
  Any 2-category $\tX$ induces a double category of squares
  $\SSq(\tX)$ (a.k.a.\ \emph{quintets}), and $\MMnd(\tX)$ is precisely
  $\MMnd(\SSq(\tX))$, by direct comparison of definitions. There is
  also a double category $\EEM(\DD)$, originally mentioned (but not
  explicitly defined) in~\cite[Chapter 7]{clarke}, which has the same
  objects and 1-cells as $\MMnd(\DD)$, and 2-cells satisfying
  conditions mate to those from~\cref{def:doublemon} for
  specializations. Hence $\EEM(\tX)$ is
  $\EEM(\SSq(\tX))$.
  \vspace{-15pt}
  \[\preq
    \mathclap{
      \begin{tikzpicture}[scale=.875,xscale=2]
        \path (0,18) -- (0,-18);
        \node [ob] (s) at (-7,7) {$\oS_\dNW$};
        \node [ob] (t) at (7,-7) {$\oS_\dSE$};
        \node [ob] (u) at (7,7) {$\oS_\dNE$};
        \node [ob] (d) at (-7,-7) {$\oS_\dSW$};
        \node [ob] (m) at (0,-7) {$\oS_\dSW$};
        \draw [a] (s) -- node[arr,above] {$\fF_1$} (u);
        \draw [a] (u) -- node[arr,right] {$\fG_1$} (t);
        \draw [a] (s) -- node[arr,left] {$\fG_2$} (d);
        \draw [a] (d) -- node[arr,below] {$\mT_{\dSW}$} (m);
        \draw [a] (m) -- node[arr,below] {$\fF_2$} (t);
        \node [cell] at (0,0) {$\sigma$};
      \end{tikzpicture}
    }
    \qqq
    \mathclap{
      \begin{tikzpicture}[scale=1.5]
        \begin{scope}
          \clip [boundc] (-6,-6) rectangle (6,6);
          \fill [bg] (-6,-6) rectangle (6,6);
          \renewcommand{\mypatha}{(-2,7) -- (-2,6) .. controls +(0,-3) and +(-1,1) .. (0,0)};
          \renewcommand{\mypathb}{(0,0) .. controls +(1,-1) and +(0,3) .. (2,-6) -- (2,-7)};
          \renewcommand{\mypathc}{(7,0) -- (0,0)};
          \renewcommand{\mypathd}{(0,0) -- (-7,0)};
          \begin{scope}
            \clip \mypatha -- \mypathb -- (-6,-6) -- (-6,6) -- cycle;
            \begin{scope}
              \clip \mypathc -- \mypathd -- (-6,6) -- (6,6) -- cycle;
              \fill [wback] (-6,-6) rectangle (6,6);
              \node [tdspecsh,wwire] at (0,0) {};
            \end{scope}
            \begin{scope}
              \clip \mypathc -- \mypathd -- (-6,-6) -- (6,-6) -- cycle;
              \fill [sback] (-6,-6) rectangle (6,6);
              \node [tdspecsh,swire] at (0,0) {};
            \end{scope}
          \end{scope}
          \begin{scope}
            \clip \mypatha -- \mypathb -- (6,-6) -- (6,6) -- cycle;
            \begin{scope}
              \clip \mypathc -- \mypathd -- (-6,6) -- (6,6) -- cycle;
              \fill [nback] (-6,-6) rectangle (6,6);
              \node [tdspecsh,nwire] at (0,0) {};
            \end{scope}
            \begin{scope}
              \clip \mypathc -- \mypathd -- (-6,-6) -- (6,-6) -- cycle;
              \fill [eback] (-6,-6) rectangle (6,6);
              \node [tdspecsh,ewire] at (0,0) {};
            \end{scope}
          \end{scope}
          \draw [omon,swire] (0,.25) .. controls +(-1,-1) and +(0,3) .. (-2.5,-6);
          \draw [olax,purp] \mypatha;
          \draw [olax,dpurp] \mypathb;
          \draw [oclax,orange] \mypathc;
          \draw [oclax,dorange] \mypathd;
          \node [tdspecs] at (0,0) {};
        \end{scope}
        \draw [bound] (-6,-6) rectangle (6,6);
        \node [label] at (0,0) {$\sigma$};
        \node [label,red,below] at (-2.5,-6) {$\mT_\dSW$};
        \node [label,red,above] at (-2.5,6) {$\phantom{\mT_\dSW}$};
      \end{tikzpicture}
    }
  \]
  \vspace{-15pt}

  Of note are the sub double categories $\MMndret(\DD)$ and
  $\EEMret(\DD)$ of $\MMnd(\tX)$ and $\EEM(\DD)$ with the same
  objects, tight 1-cells, and 2-cells, but in which loose 1-cells are
  \emph{monad retromorphisms}~\cite[Definition
  5.3]{clarke}.\footnote{Though it seems a definition of $\EEM(\DD)$
    does not appear in the literature, $\EEMret(\DD)$ is described
    explicitly in~\cite{dimeglio:talk}.}
  A monad retromorphism is by definition a monad lax 1-cell in
  $\Loose(\DD)$ that is carried by a companion in $\DD$.\footnote{An
    alternative tweak is to define a monad retromorphism directly in
    terms of a tight 1-cell (rather than a companion loose 1-cell) so
    that the structure 2-cell $\chi$ is given by a
    \emph{retrocell}~\cite[Definition 2.1]{pare:retro}. This way we
    get a strict double category $\MMndret(\DD)$ even for $\DD$ that
    is weak in the loose direction.}  In the case that all tight
  1-cells in $\DD$ have companions, a monad tight 1-cell, on the other
  hand, corresponds to a monad colax 1-cell in $\Loose(\DD)$ that is
  carried by a companion in $\DD$.

  All 1-cells are companions in the double category $\SSq(\tX)$, so
  $\MMnd(\tX)$ and $\EEM(\tX)$ may equally well be viewed as
  \[\MMndret(\SSq(\tX)) \simeq \MMnd(\SSq(\tX))\qquad \text{and}
    \qquad \EEMret(\SSq(\tX)) \simeq \EEM(\SSq(\tX)).\] Conversely,
  given a double category $\DD$ in which all tight 1-cells have
  companions, $\MMndret(\DD)$ and $\EEMret(\DD)$ are recovered from
  $\MMnd(\Loose(\DD))$ and $\EEM(\Loose(\DD))$ as the locally full sub
  double categories consisting of 1-cells carried by companions in
  $\DD$.
\end{remark}

This yields the following examples from the work of Clarke and Di
Meglio~\cite{clarke,clarke-dimeglio,dimeglio:talk}.

\begin{example}\label{ex:internal}
  Let $\cS$ be a category with pullbacks. Denote the
  2-category\footnote{Here one can either use a strict 2-category
    equivalent to $\SSpan(\cS)$ as in \cite[Section 2.3]{lack-street}
    (such as the category of left adjoint functors between categories
    $\Set^X$ in the case $\cS = \Set$), use the equipment structure of
    $\SSpan(\cS)$ to define composition of just maps strictly, or work
    with weak 2-categories and doubly weak double categories.} of
  spans in $\cS$ by $\Span(\cS)$.
  \begin{enumerate}[label=(\roman*)]
  \item\label{item:internalflat} The double category of $\cS$-internal
    categories, retrofunctors and functors, and compatible squares
    from~\cite[Chapter 5]{clarke} is equivalent to the locally full
    sub double category of $\MMnd(\Span(\cS))$ consisting of 1-cells
    carried by maps of $\cS$.

  \item\label{item:internal} The double category with the same objects
    and 1-cells but whose 2-cells are \emph{transformations} as
    defined in~\cite{dimeglio:talk} is equivalent to the locally full
    sub double category of $\EEM(\Span(\cS))$ consisting of 1-cells
    carried by maps of $\cS$.
  \end{enumerate}

  In particular, the full sub double categories of
  $\MMnd(\Span(\Set))$ and $\EEM(\Span(\Set))$ consisting of 1-cells
  carried by functions are the two double categories of categories,
  functors and retrofunctors from~\cite{clarke}
  and~\cite{clarke-dimeglio}.\footnote{There are also enriched (as
    opposed to internal) versions, involving $\MMat(\cS)$ instead of
    $\SSpan(\cS)$~\cite{dimeglio:talk}.}
\end{example}

\pagebreak
\chapter{Bimodules}

\begin{definition}
  Let $\mT$ be a monad on $\oS$. A \emph{$\mT$-opmodule} (a.k.a.\
  \emph{left $\mT$-module}) is defined by flipping all of the diagrams
  in \cref{def:module} horizontally.
  \vspace{-5pt}
  
  \[\preq
    \mathclap{\bM\colon \oS \to \oX}
    \qqq
    \mathclap{
      \begin{tikzpicture}[yscale=.5,xscale=-1]
        \begin{scope}
          \clip [boundc] (-6,-6) rectangle (6,6);
          \fill [bg] (-6,-6) rectangle (6,6);
          \fill [bgd] (-6,-6) rectangle (0,6);
          \draw [olmods=red] (0,6) -- (0,-6);
        \end{scope}
        \draw [bound] (-6,-6) rectangle (6,6);
        \node [label] at (-3.125,0) {$\oX$};
        \node [label] at (3,0) {$\oS$};
        \node [label,above] at (0,6) {$\bM$};
        \node [label,below] at (0,-6) {$\phantom{\bM}$};
      \end{tikzpicture}
    }
  \]
  \vspace{-12pt}
  \[\preq
    \mathclap{
      \begin{tikzpicture}
        \node [ob] (s) at (-7,0) {$\oS$};
        \node [ob] (m) at (0,7) {$\oS$};
        \node [ob] (t) at (7,0) {$\oX$};
        \draw [a] (s) -- node[arr,al] {$\mT$} (m);
        \draw [a] (m) -- node[arr,ar] {$\bM$} (t);
        \draw [a] (s) -- node[arr,below] {$\bM$} (t);
        \node [cell] at (0,2.5) {$\lambda$};
      \end{tikzpicture}
    }
    \qqq
    \mathclap{
      \begin{tikzpicture}[xscale=-1]
        \begin{scope}
          \clip [boundc] (-6,-6) rectangle (6,6);
          \fill [bg] (-6,-6) rectangle (6,6);
          \renewcommand{\mypath}{(-2.6,6) .. controls +(0,-3.5) and +(-1,1) .. (0,0) -- (0,-6)}
          \fill [bgd] \mypath -- (-6,-6) -- (-6,6) -- cycle;
          \draw [omon,red] (2.6,6) .. controls +(0,-3.5) and +(1,1) .. (0,0);
          \draw [olmods=red] \mypath;
        \end{scope}
        \draw [bound] (-6,-6) rectangle (6,6);
      \end{tikzpicture}
    }
  \]
  \[\preq
    \mathclap{
      \begin{tikzpicture}[xscale=.75,yscale=1.125]
        \node [ob] (s) at (-14,0) {$\oS$};
        \node [ob] (mm) at (-9,7) {$\oS$};
        \node [ob] (m) at (2,7) {$\oS$};
        \node [ob] (t) at (7,0) {$\oX$};
        \draw [a] (s) -- node[arr,al] {$\mT$} (mm);
        \draw [a] (mm) -- node[arr,above] {$\mT$} (m);
        \draw [a] (s) -- node[arr,br] {$\mT$} (m);
        \draw [a] (m) -- node[arr,ar] {$\bM$} (t);
        \draw [a] (s) -- node[arr,below] {$\bM$} (t);
        \node [cell] at (.5,2.5) {$\lambda$};
        \node [cell] at (-7,5) {$\mu$};
      \end{tikzpicture}
      \eqq
      \begin{tikzpicture}[xscale=.75,yscale=1.125]
        \node [ob] (s) at (-7,0) {$\oS$};
        \node [ob] (m) at (-2,7) {$\oS$};
        \node [ob] (mm) at (9,7) {$\oS$};
        \node [ob] (t) at (14,0) {$\oX$};
        \draw [a] (m) -- node[arr,above] {$\mT$} (mm);
        \draw [a] (mm) -- node[arr,ar] {$\bM$} (t);
        \draw [a] (s) -- node[arr,al] {$\mT$} (m);
        \draw [a] (m) -- node[arr,bl] {$\bM$} (t);
        \draw [a] (s) -- node[arr,below] {$\bM$} (t);
        \node [cell] at (-.5,2.5) {$\lambda$};
        \node [cell] at (7,5) {$\lambda$};
      \end{tikzpicture}
    }
    \qqq
    \mathclap{
      \begin{tikzpicture}[xscale=-1]
        \begin{scope}
          \clip [boundc] (-7,-7) rectangle (7,7);
          \fill [bg] (-7,-7) rectangle (7,7);
          \renewcommand{\mypath}{(-4,7) .. controls +(0,-5.5) and +(-1,1) .. (0,-2.5) -- (0,-7)}
          \fill [bgd] \mypath -- (-7,-7) -- (-7,7) -- cycle;
          \draw [omon,red] (0,7) .. controls +(0,-2.5) and +(-1,1) .. (2,2.5);
          \draw [omon,red] (4,7) .. controls +(0,-2.5) and +(1,1) .. (2,2.5);
          \draw [omon,red] (2,2.5) .. controls +(0,-2.5) and +(1,1) .. (0,-2.5);
          \node [tmon,red] at (2,2.5) {};
          \draw [olmods=red] \mypath;
        \end{scope}
        \draw [bound] (-7,-7) rectangle (7,7);
      \end{tikzpicture}
      \eqq
      \begin{tikzpicture}[xscale=-1]
        \begin{scope}
          \clip [boundc] (-7,-7) rectangle (7,7);
          \fill [bg] (-7,-7) rectangle (7,7);
          \renewcommand{\mypath}{(-4,7) .. controls +(0,-2.5) and +(-1,1) .. (-2,2.5) .. controls +(0,-2.5) and +(-1,1) .. (0,-2.5) -- (0,-7)}
          \fill [bgd] \mypath -- (-7,-7) -- (-7,7) -- cycle;
          \draw [omon,red] (0,7) .. controls +(0,-2.5) and +(1,1) .. (-2,2.5);
          \draw [omon,red] (4,7) .. controls +(0,-5.5) and +(1,1) .. (0,-2.5);
          \draw [olmods=red] \mypath;
        \end{scope}
        \draw [bound] (-7,-7) rectangle (7,7);
      \end{tikzpicture}
    }
  \]
  \vspace{-5pt}
  \[\preq
    \mathclap{
      \begin{tikzpicture}[xscale=.75,yscale=1.125]
        \path (0,11) -- (0,-4);
        \node [ob] (s) at (-14,0) {$\oS$};
        \node [ob] (m) at (2,7) {$\oS$};
        \node [ob] (t) at (7,0) {$\oX$};
        \draw [a] (s) -- node[arr,br] {$\mT$} (m);
        \draw [a] (m) -- node[arr,ar] {$\bM$} (t);
        \draw [a] (s) -- node[arr,below] {$\bM$} (t);
        \draw [eq] (m) .. controls +(-8,1.5) and +(1.5,8) .. (s);
        \node [cell] at (.5,2.5) {$\lambda$};
        \node [cell] at (-7,5) {$\eta$};
      \end{tikzpicture}
      \eqq
      \begin{tikzpicture}[xscale=.75,yscale=1.125]
        \path (0,11) -- (0,-4);
        \node [ob] (s) at (-14,0) {$\oS$};
        \node [ob] (t) at (7,0) {$\oX$};
        \draw [a] (s) .. controls +(4,8.5) and +(-4,8.5) .. node[arr,above] {$\bM$} (t);
        \draw [a] (s) -- node[arr,below] {$\bM$} (t);
        \node [cell] at (-3.5,3) {$=$};
      \end{tikzpicture}
    }
    \qqq
    \mathclap{
      \begin{tikzpicture}[xscale=-1]
        \begin{scope}
          \clip [boundc] (-7,-7) rectangle (7,7);
          \fill [bg] (-7,-7) rectangle (7,7);
          \renewcommand{\mypath}{(-2.6,7) .. controls +(0,-5.5) and +(-1,1) .. (0,-2.5) -- (0,-7)}
          \fill [bgd] \mypath -- (-7,-7) -- (-7,7) -- cycle;
          \draw [omon,red] (2,2.5) .. controls +(0,-2.5) and +(1,1) .. (0,-2.5);
          \node [tmon,red] at (2,2.5) {};
          \draw [olmods=red] \mypath;
        \end{scope}
        \draw [bound] (-7,-7) rectangle (7,7);
      \end{tikzpicture}
      \eqq
      \begin{tikzpicture}[xscale=-1]
        \begin{scope}
          \clip [boundc] (-7,-7) rectangle (7,7);
          \fill [bg] (-7,-7) rectangle (7,7);
          \fill [bgd] (-7,-7) rectangle (0,7);
          \draw [olmods=red] (0,7) -- (0,-7);
        \end{scope}
        \draw [bound] (-7,-7) rectangle (7,7);
      \end{tikzpicture}
    }
  \]
  \afterimage
  
  That is, a $\mT$-opmodule is a module in $\tX\op$ with respect to
  the monad therein corresponding to $\mT$.  A \emph{$\mT$-opmodule
    map} is defined similarly (equivalent to a module map of
  corresponding modules).
\end{definition}

\begin{definition}\label{def:bimodule}
  Let $\mT_\dL$ and $\mT_\dR$ be monads on $\oS_\dL$ and $\oS_\dR$. A
  \emph{$(\mT_\dL, \mT_\dR)$-bimodule}\footnote{The definitions of
    bimodule and bimodule map make sense in a virtual 2-category. A
    bimodule between monads can be defined more concisely as a functor
    from the terminal \emph{virtual biactegory}: a virtual 2-category
    with two objects $\oS_\dL$ and $\oS_\dR$ and no 1-cells
    $\oS_\dR \to \oS_\dL$.} consists of a 1-cell
  \[\preq
    \mathclap{\bM\colon \oS_\dL \to \oS_\dR}
    \qqq
    \mathclap{

    }
  \]
  that is both a $\mT_\dL$-opmodule and a $\mT_\dR$-module
  \[\preq
    \mathclap{
      %
    }
  \]
  satisfying the compatibility law
  \[\preq
    \mathclap{
      %
    }
  \]
  
  A \emph{bimodule map} (a.k.a. \emph{modulation}) between $(\mT_\dL,\mT_\dR)$-bimodules
  $\bM_1$ and $\bM_2$ is a 2-cell $\phi \colon \bM_1 \Rightarrow
  \bM_2$ that is both a module map and an opmodule map.
  \[\preq
    \mathclap{
      %
    }
  \]
\end{definition}
  
\begin{convention}
  In string diagrams, we write the 1-cell $\bM$ carrying a
  $(\mT_\dL, \mT_\dR)$-bimodule as a string coated by a film the same color
  as $\mT_\dL$ on the left and a film the same color as $\mT_\dR$ on the
  right. This is to suggest the bimodule laws by topological
  deformation.


  \[
    %
  \]
  \afterimage

  We write a bimodule map as a similarly coated node, to suggest the
  bimodule map laws by topological deformation.
\end{convention}

\begin{remark}
  Both the opmodule structure 2-cell
  and the module structure 2-cell
  of a bimodule are recovered from the unique induced 2-cell
  \[\preq
    \mathclap{
      %
    }
  \]
  by composing with a unit on either side.
  \[\preq
    \mathclap{
      %
    }
  \]
  \afterimage

\end{remark}

We next state an analogue of \cref{prop:formaltfae} for bimodules.

\begin{definition}
  An \emph{opalgebra object} (a.k.a.\ \emph{Kleisli object}) is the
  codomain $\OpAlgTS$ of a universal $\mT$-opmodule $\mkcarT$.
  
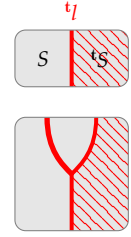
\vspace{-5pt}
  
  \[\preq
    \mathclap{\mkcarT \colon \oS \to \OpAlgTS}
    \qqq
    \mathclap{
      %
    }
  \]
  \afterimage
  
  That is, an opalgebra object for the monad $\mT$ is a module in
  $\tX\op$ over the monad therein corresponding to $\mT$.
\end{definition}

\begin{convention}
  To distinguish between opalgebra objects and algebra objects, we
  texture the former in string diagrams with $\diagdown$ lines (the
  direction that a monad springs from an opmodule) and the latter with
  $\diagup$ lines (the direction that a monad springs from a module).
\end{convention}

\begin{proposition}\label{prop:bimtfae}
  Let $\mT_\dL$ be a monad on $\oS_\dL$ with opalgebra object
  $\OpAlg{\mT_\dL}{\oS_\dL}$, and let $\mT_\dR$ be a monad on $\oS_\dR$,
  with algebra object $\Alg{\mT_\dR}{\oS_\dR}$. Let
  $\bM\colon \oS_\dL \to \oS_\dR$ be an arbitrary 1-cell.
  
  \begin{enumerate}[label=(\roman*)]
  \item\label{item:bimodone}
    Giving $\bM$ the structure of a $(\mT_\dL, \mT_\dR)$-bimodule
    \vspace{-5pt}
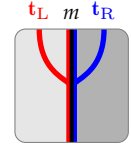
    \[\preq
      \hspace{-\leftmargin}
      \mathclap{
        \begin{tikzpicture}[xscale=.65,yscale=.75]
          \node [ob] (s) at (-15,0) {$\oS_\dL$};
          \node [ob] (mm) at (-7,9) {$\oS_\dL$};
          \node [ob] (m) at (7,9) {$\oS_\dR$};
          \node [ob] (t) at (15,0) {$\oS_\dR$};
          \draw [a] (s) -- node[arr,al] {$\mT_\dL$} (mm);
          \draw [a] (mm) -- node[arr,above] {$\bM$} (m);
          \draw [a] (m) -- node[arr,ar] {$\mT_\dR$} (t);
          \draw [a] (s) -- node[arr,below] {$\bM$} (t);
          \node [cell] at (0,4.5) {$\alpha$};
        \end{tikzpicture}
      }
      \qqq
      \mathclap{
        \begin{tikzpicture}[scale=.85714]
          \begin{scope}
            \clip [boundc] (-7,-7) rectangle (7,7);
            \fill [bg] (-7,-7) rectangle (7,7);
            \draw [omon,red] (-4,7) .. controls +(0,-5) and +(-1,1) .. (0,0);
            \draw [obmodh,red] (0,7) -- (0,-7);
            \begin{scope}
              \clip (0,-7) rectangle (7,7);
              \fill [bgd] (-7,-7) rectangle (7,7);
              \draw [obmodh,blue] (0,7) -- (0,-7);
              \draw [omon,blue] (4,7) .. controls +(0,-5) and +(1,1) .. (0,0);
            \end{scope}
            \draw [obmod] (0,7) -- (0,-7);
          \end{scope}
          \draw [bound] (-7,-7) rectangle (7,7);
          \node [label,above] at (0,7) {$\bM$};
          \node [label,red,above] at (-4,7) {$\mT_\dL$};
          \node [label,blue,above] at (4,7) {$\mT_\dR$};
          \node [label,below] at (0,-7) {$\phantom{\bM}$};
          \node [label,below] at (-4,-7) {$\phantom{\mT_\dL}$};
          \node [label,below] at (4,-7) {$\phantom{\mT_\dR}$};
        \end{tikzpicture}
      }
      \vspace{-5pt}
    \]
    is equivalent to giving an arbitrary 1-cell
    \vspace{-10pt}
    
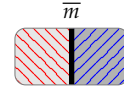
    \[\preq
      \hspace{-\leftmargin}
      \mathclap{\blift{\bM} \colon \OpAlg{\mT_\dL}{\oS_\dL} \to \Alg{\mT_\dR}{\oS_\dR}}
      \qqq
      \mathclap{
        \begin{tikzpicture}[yscale=.5]
          \begin{scope}
            \clip [boundc] (-6,-6) rectangle (6,6);
            \fill [bg] (-6,-6) rectangle (6,6);
            \fill [bgd] (0,-6) rectangle (6,6);
            \fill [kl=red] (-6,-6) rectangle (0,6);
            \fill [em=blue] (0,-6) rectangle (6,6);
            \draw [o] (0,7) -- (0,-7);
          \end{scope}
          \draw [bound] (-6,-6) rectangle (6,6);
          \node [label,above] at (0,6) {$\blift{\bM}$};
          \node [label,below] at (0,-6) {$\phantom{\blift{\bM}}$};
        \end{tikzpicture}
      }
    \]
    \afterimage
    \vspace{-15pt}
  \end{enumerate}

  \noindent
  Now let
  %
  $\bM_1\colon \oS_\dL \to \oS_\dR$ and
  $\bM_2\colon \oS_\dL \to \oS_\dR$ be $(\mT_\dL, \mT_\dR)$-bimodules.
  \begin{enumerate}[label=(\roman*)]
    \setcounter{enumi}{1}
  \item\label{item:bimodtwo} A bimodule map
    $\phi \colon \bM_1 \Rightarrow \bM_2$
    \vspace{-5pt}
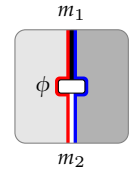
    \[\preq
      \hspace{-\leftmargin}
      \mathclap{
        \begin{tikzpicture}
          \node [ob] (s) at (-7,0) {$\oS_\dL$};
          \node [ob] (t) at (7,0) {$\oS_\dR$};
          \draw [a] (s) .. controls +(5,3) and +(-5,3) .. node[arr,above] {$\bM_1$} (t);
          \draw [a] (s) .. controls +(5,-3) and +(-5,-3) .. node[arr,below] {$\bM_2$} (t);
          \node [cell] at (0,0) {$\phi$};
        \end{tikzpicture}
      }
      \qqq
      \mathclap{
        \begin{tikzpicture}
          \begin{scope}
            \clip [boundc] (-6,-6) rectangle (6,6);
            \fill [bg] (-6,-6) rectangle (6,6);
            \draw [obmodh, red] (0,6) -- (0,-6);
            \node [tmhomh, red] at (0,0) {};
            \begin{scope}
              \clip (0,-6) rectangle (6,6);
              \fill [bgd] (-6,-6) rectangle (6,6);
              \draw [obmodh, blue] (0,6) -- (0,-6);
              \node [tmhomh, blue] at (0,0) {};
            \end{scope}
            \draw [obmod] (0,6) -- (0,0);
            \draw [obmod,white] (0,0) -- (0,-6);
            \node [tmhom] at (0,0) {};
          \end{scope}
          \draw [bound] (-6,-6) rectangle (6,6);
          \node [label,above] at (0,6) {$\bM_1$};
          \node [label,below] at (0,-6) {$\bM_2$};
          \node [label,left=4] at (0,0) {$\phi$};
        \end{tikzpicture}
      }
      \vspace{-5pt}
    \]
    is equivalent to an arbitrary 2-cell
    $\blift{\phi} \colon \blift{\bM_1} \Rightarrow \blift{\bM_2}$.
    \vspace{-5pt}
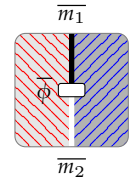
    \[\preq
      \hspace{-\leftmargin}
      \mathclap{
        \begin{tikzpicture}
          \node [ob] (s) at (-7,0) {$\OpAlg{\mT_\dL}{S_\dL}$};
          \node [ob] (t) at (7,0) {$\Alg{\mT_\dR}{S_\dR}$};
          \draw [a] (s) .. controls +(5,3) and +(-5,3) .. node[arr,above] {$\blift{\bM_1}$} (t);
          \draw [a] (s) .. controls +(5,-3) and +(-5,-3) .. node[arr,below] {$\blift{\bM_2}$} (t);
          \node [cell] at (0,0) {$\blift{\phi}$};
        \end{tikzpicture}
      }
      \qqq
      \mathclap{
        \begin{tikzpicture}
          \begin{scope}
            \clip [boundc] (-6,-6) rectangle (6,6);
            \fill [bg] (-6,-6) rectangle (6,6);
            \fill [bgd] (0,-6) rectangle (6,6);
            \fill [kl=red] (-6,-6) rectangle (0,6);
            \fill [em=blue] (0,-6) rectangle (6,6);
            \draw [o] (0,7) -- (0,0);
            \draw [o,white] (0,0) -- (0,-7);
            \node [tmhom] at (0,0) {};
          \end{scope}
          \draw [bound] (-6,-6) rectangle (6,6);
          \node [label,above] at (0,6) {$\blift{\bM_1}$};
          \node [label,below] at (0,-6) {$\blift{\bM_2}$};
          \node [label,left=3.5] at (0,0) {$\blift{\phi}$};
        \end{tikzpicture}
      }
    \]
  \end{enumerate}
\end{proposition}
\begin{proof}
  Observe that for any $(\mT_\dL, \mT_\dR)$-bimodule $\bM$, the
  structure 2-cell $\lambda \colon \mT_\dL \then \bM \Rightarrow \bM$
  constitutes a $\mT_\dR$-module map (this is the bimodule
  compatibility law). Applying the universal properties of
  $\Alg{\mT_\dR}{\oS_\dR}$ and $\OpAlg{\mT_\dL}{\oS_\dL}$ in turn, we
  may lift the $\mT_\dL$-opmodule structure on
  $\bM \colon \oS_\dL \to \oS_\dR$ to $\mT_\dL$-opmodule structure on
  $\lift{\bM} \colon \oS_\dL \to \Alg{\mT_\dR}{\oS_\dR}$, which lifts
  to a 1-cell $\OpAlg{\mT_\dL}{\oS_\dL} \to \Alg{\mT_\dR}{\oS_\dR}$ as
  desired.  Bimodule maps $\phi$ similarly lift in two steps, and
  these translation processes via universal property are by definition
  natural bijections.\footnote{Symmetrically, the 
    structure 2-cell $\rho \colon \bM \then \mT_\dR \Rightarrow \bM$
    constitutes a $\mT_\dL$-opmodule map, and the steps may equally
    well be performed in the other order, first lifting the
    $\mT_\dR$-module structure on $\bM$ to $\mT_\dR$-module structure
    on the corresponding 1-cell
    $\OpAlg{\mT_\dL}{\oS_\dL} \to \oS_\dR$. Both methods yield the
    same result, since the inverses of the natural bijections, given
    by composing $\mkcarmod{\mT_\dL}$ and $\mcarmod{\mT_\dR}$ on
    either side of the lifted cells, commute by associativity of
    composition.}
\end{proof}

\begin{definition}
  There is also a more general notion of map from multiple compatible
  bimodules to one. Let
  \[\preq
    \mathclap{

    }
  \]
  be a finite path of 1-cells where each $\oS_i$ is equipped with a monad
  $\mT_i$ and each $\bM_i$ is a $(\mT_{i-1},\mT_i)$-bimodule, and let
  \[\preq
    \mathclap{\bM\colon \oS_0 \to \oS_n}
    \qqq
    \mathclap{
      %
    }
  \]
\end{definition}

\begin{remark}
  In the case $n = 0$, meaning the input is a length zero path (i.e.\
  a single object equipped with a monad), the bimodule map laws
  specialize to
  \vspace{-5pt}
  \[
    \preq
    \mathclap{
      %
    }
  \]
\end{remark}

Monads, bimodules, and bimodule maps in $\tX$ comprise a \emph{virtual
  2-category}, which is like a 2-category but where 1-cells cannot
compose and each 2-cell has a path of compatible 1-cells as input but
only one 1-cell as output.


\begin{remark}\label{rem:generalbimod}
  There is a still more general variant of the definition of bimodule
  map that incorporates monad colax (or alternatively lax\footnote{We
    focus on colax 1-cells instead of lax 1-cells because this gives a
    virtual double category instead of a covirtual double category.})
  1-cells
  \[\preq
    \mathclap{

    }
  \]
  
  \noindent
  required to satisfy laws analogous to \cref{def:bimodule}
  \[\preq
    \mathclap{
      %
    }
  \]
  \afterimage

  A special case of note is with length $0$ input (consisting of just
  a monad $\mT_0$), and with output being a monad $\mT$ itself
  regarded as a $(\mT,\mT)$-bimodule. This is precisely a
  specialization between the monad colax 1-cells $\fF_\dR$ and
  $\fF_\dL$ from $\mT_0$ to $\mT$.  \vspace{-10pt}
  
  \[\preq
    \mathclap{
      %
    }
  \]
  
  Similar to this definition, and more commonly studied, are loose
  monads and bimodules in a double category (or virtual double
  category) $\DD$, which incorporates monad tight 1-cells instead of
  monad colax 1-cells~\cite[Definition 2.8]{cruttwell-shulman} (see
  also~\cite{wood:ii,leinster:enrichment} for earlier references).
  \[\preq
    \mathclap{
      %
    }
  \]
  
  Monads and bimodules in $\DD$ form a \emph{virtual double category}
  $\MMod(\DD)$, which is like a double category but where loose
  1-cells --- here bimodules --- cannot compose, and each 2-cell has a
  path of compatible source loose 1-cells but only one target loose
  1-cell. The above 2-categorical definition involving colax 1-cells
  can be seen as a special case of the double-categorical definition.
  Specifically, $\MMod(\tX) \coloneqq \MMod(\SSq(\tX))$ (similar to
  \cref{rem:clarke-dimeglio}).

  The virtual double category $\MMod(\DD)$ has \emph{unit loose
    1-cells}~\cite[Definition 5.1]{cruttwell-shulman}, given by monads
  regarded as bimodules over themselves, which implies it has an
  underlying tight 2-category~\cite[Proposition
  6.1]{cruttwell-shulman}.\footnote{For any virtual double category
    $\DD$, the virtual double category $\MMod(\DD)$ is characterized
    as the \emph{cofree virtual double category with units} on
    $\DD$~\cite[Proposition 5.14]{cruttwell-shulman}.}  The 2-category
  of monads, colax 1-cells, and specializations in the 2-category
  $\tX$ --- i.e.\ the dual construction of $\EM(\tX)$ (which is
  denoted $\Kl(\tX)$ in~\cite{lack-street}) --- is the same as
  underlying tight 2-category\footnote{Here we are taking the
    ``underlying tight 2-category'' to have 2-cells that go in the
    opposite direction of loose 1-cells, as in~\cite[Proposition
    6.1]{cruttwell-shulman}. This way $\tX$ is the underlying tight
    2-category of $\SSq(\tX)$, where $\SSq(\tX)$ is defined based on
    the picture that loose 1-cells point right, tight 1-cells point
    down, and 2-cells point down-left. But the other convention is
    often preferable, for instance to recover the $\Cat(\cS)$ as the
    underlying tight 2-category of $\MMod(\SSpan(\cS))$ with the
    arrows in categories being defined to point in the same direction
    as their carrying spans.} of $\MMod(\tX)$. This also connects with
  \cref{ex:internal}, as $\MMod(\SSpan(\cS))$ is the double category
  of categories, profunctors, and functors internal to
  $\cS$~\cite[Example 2.10]{cruttwell-shulman}.
\end{remark}

\begin{remark}
  Bimodules and bimodule maps are instances of the
  2-category-theoretic notions of module and modulation between lax
  functors of
  2-categories~\cite{cockett-koslowski-seely-wood,pare:modules}. Specificially,
  viewing monads as lax functors from the terminal 2-category
  $\tcat{1}$ as mentioned in \cref{rem:transfors}, a bimodule between
  monads is equivalently a module between such lax functors, and a
  (multi-input) bimodule map is equivalently a (multi-)modulation
  between such modules.

  Modules and modulations can also be represented in string
  diagrams. As mentioned in \cref{rem:transfors}, functors,
  transformations, and modifications between 2-categories $\tcat{C}$
  and $\tcat{D}$ are represented as regions, strings, and nodes that
  may be superimposed on string diagrams in $\tcat{C}$ to yield string
  diagrams in $\tcat{D}$. (Lax functors and their transformations and
  modifications are represented in the same way as ordinary functors,
  except we limit ourselves to those string diagrams in the domain
  2-category $\tcat{C}$ featuring any number of input strings but only
  one output string.) Similarly to transformations and modifications,
  modules and modulations are strings and nodes that may be
  superimposed on string diagrams in $\tcat{C}$ to yield string
  diagrams in $\tcat{D}$, only they must be superimposed to coincide
  exclusively with strings and nodes, not regions. Thus, whereas a
  transformation (e.g.\ monad 1-cell) appears as a membrane sliding
  over two instances of a shape (e.g.\ monad) and translating between
  them, a module (e.g.\ bimodule between monads) appears as a membrane
  between the two that is itself stuck on a string of that underlying
  shape.
\end{remark}

\pagebreak
\chapter{Distributive laws}

\begin{definition}\label{def:distributive}
  Let $\mT_1$ and $\mT_2$ be monads on $\oS$. A \emph{distributive
    law}\footnote{The definition of distributive law makes sense in an
    implicit 2-category.} of $\mT_1$ over $\mT_2$ consists of a 2-cell
  $\chi \colon \mT_2 \then \mT_1 \Rightarrow \mT_1 \then \mT_2$
  \[\preq
    \mathclap
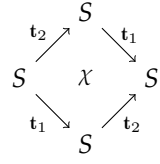
{
      \begin{tikzpicture}
        \node [ob] (s) at (-7,0)
        {$\oS$}; \node [ob] (d) at (0,-7)
        {$\oS$}; \node [ob] (u) at (0,7)
        {$\oS$}; \node [ob] (t) at (7,0)
        {$\oS$}; \draw [a] (s) -- node[arr,al]
        {$\mT_2$} (u); \draw [a] (u) -- node[arr,arcl]
        {$\mT_1$} (t); \draw [a] (s) -- node[arr,bl]
        {$\mT_1$} (d); \draw [a] (d) -- node[arr,br]
        {$\mT_2$} (t); \node [cell] at (0,0) {$\chi$};
      \end{tikzpicture}
    }
    \qqq
    \mathclap
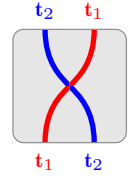
{
      \begin{tikzpicture}
        \begin{scope}
          \clip [boundc] (-6,-6) rectangle (6,6);
          \fill [bg] (-6,-6) rectangle (6,6);
          \renewcommand{\mypatha}{(-2.6,6) .. controls +(0,-3.5) and +(-1,1) .. (0,0) .. controls +(1,-1) and +(0,3.5).. (2.6,-6)}
          \renewcommand{\mypathb}{(2.6,6) .. controls +(0,-3.5) and +(1,1) .. (0,0) .. controls +(-1,-1) and +(0,3.5).. (-2.6,-6)}
          \begin{scope}
            \clip (-6,6) -- (0,6) -- (0,-6) -- (6,-6) -- (6,0) -- (-6,0) -- cycle;
            \draw [omon,blue] \mypatha;
          \end{scope}
          \begin{scope}
            \clip (6,6) -- (0,6) -- (0,-6) -- (-6,-6) -- (-6,0) -- (6,0) -- cycle;
            \draw [omon,red] \mypathb;
          \end{scope}
        \end{scope}
        \draw [bound] (-6,-6) rectangle (6,6);
        \node [label,blue,above] at (-2.6,6) {$\mT_2$};
        \node [label,red,above] at (2.6,6) {$\mT_1$};
        \node [label,blue,below] at (2.6,-6) {$\mT_2$};
        \node [label,red,below] at (-2.6,-6) {$\mT_1$};
      \end{tikzpicture}
    }
  \]
  exhibiting $\mT_2$ as a monad lax 1-cell $\mT_1 \to \mT_1$
  \[\preq
    \mathclap
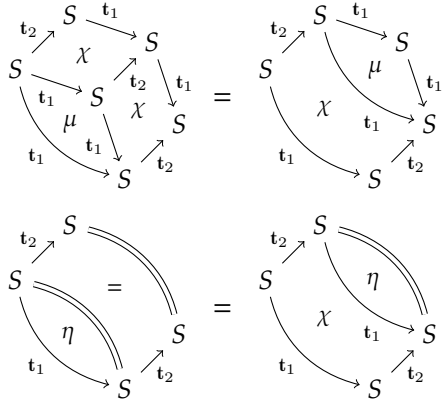
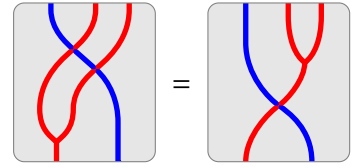
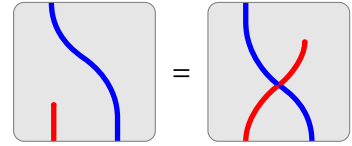
{
      \begin{tikzpicture}[scale=.825]
        \node [ob] (s) at (-7,0) {$\oS$};
        \node [ob] (t) at (10.5,3.5) {$\oS$};
        \node [ob] (u) at (0,7) {$\oS$};
        \node [ob] (d) at (3.5,-3.5) {$\oS$};
        \node [ob] (tt) at (14,-7) {$\oS$};
        \node [ob] (dd) at (7,-14) {$\oS$};
        \draw [a] (s) -- node[arr,al] {$\mT_2$} (u);
        \draw [a] (u) -- node[arr,above] {$\mT_1$} (t);
        \draw [a] (s) -- node[arr,below,pos=.3] {$\mT_1$} (d);
        \draw [a] (d) -- node[arr,br] {$\mT_2$} (t);
        \draw [a] (t) -- node[arr,right] {$\mT_1$} (tt);
        \draw [a] (d) -- node[arr,left,pos=.7] {$\mT_1$} (dd);
        \draw [a] (dd) -- node[arr,br] {$\mT_2$} (tt);
        \draw [a] (s) .. controls +(2,-8) and +(-8,2) .. node[arr,bl] {$\mT_1$} (dd);
        \node [cell] at (1.75,1.75) {$\chi$};
        \node [cell] at (8.75,-5.25) {$\chi$};
        \node [cell] at (0,-7) {$\mu$};
      \end{tikzpicture}
      \hspace{-2.5pt}
      \eqq
      \begin{tikzpicture}[scale=.825]
        \node [ob] (s) at (-7,0) {$\oS$};
        \node [ob] (t) at (10.5,3.5) {$\oS$};
        \node [ob] (u) at (0,7) {$\oS$};
        \node [ob] (tt) at (14,-7) {$\oS$};
        \node [ob] (dd) at (7,-14) {$\oS$};
        \draw [a] (s) -- node[arr,al] {$\mT_2$} (u);
        \draw [a] (u) -- node[arr,above] {$\mT_1$} (t);
        \draw [a] (t) -- node[arr,right] {$\mT_1$} (tt);
        \draw [a] (dd) -- node[arr,br] {$\mT_2$} (tt);
        \draw [a] (u) .. controls +(2,-8) and +(-8,2) .. node[arr,bl,pos=.8] {$\mT_1$} (tt);
        \draw [a] (s) .. controls +(2,-8) and +(-8,2) .. node[arr,bl] {$\mT_1$} (dd);
        \node [cell] at (.5,-5) {$\chi$};
        \node [cell] at (7,0) {$\mu$};
      \end{tikzpicture}
      \hspace{-2.5pt}
    }
    \qqq
    \mathclap{
      \begin{tikzpicture}[xscale=.7894,yscale=.7]
        \begin{scope}
          \clip [boundc] (-10,-14) rectangle (9,10);
          \fill [bg] (-10,-14) rectangle (9,10);
          \renewcommand{\mypath}{(-5,10) -- (-5,9) .. controls +(0,-3) and +(-1,1) .. (-2,3) -- (1,0) .. controls +(2,-2) and +(0,3) .. (4,-8) -- (4,-14)};
          \draw [omon,red] (1,10) -- (1,9) .. controls +(0,-3) and +(1,1) .. (-2,3);
          \draw [omon,red] (5.5,10) -- (5.5,9) .. controls +(0,-5) and +(1,1) .. (1,0);
          \draw [omon,red] (-2,3) .. controls +(-1,-1) and +(0,5) .. (-6.5,-6) .. controls +(0,-3) and +(-1,1) .. (-4.25,-11);
          \draw [omon,red] (1,0) .. controls +(-1,-1) and +(0,3) .. (-2,-6) .. controls +(0,-3) and +(1,1) .. (-4.25,-11);
          \draw [omon,red] (-4.25,-11) -- (-4.25,-14);
          \node [tmon,red] at (-4.25,-11) {};
          \clip (-10,10) -- (-2,10) -- (-2,0) -- (10,0) -- (10,-14) -- (1,-14) -- (1,3) -- (-10,3) -- cycle;
          \draw [omon,blue] (-5,11) -- \mypath -- (5,-15);
        \end{scope}
        \draw [bound] (-10,-14) rectangle (9,10);
      \end{tikzpicture}
      \eqq
      \begin{tikzpicture}[xscale=1,yscale=.7]
        \begin{scope}
          \clip [boundc] (-7.5,-13) rectangle (7.5,11);
          \fill [bg] (-7.5,-14) rectangle (7.5,11);
          \renewcommand{\mypath}{(-3.5,11) -- (-3.5,5) .. controls +(0,-5) and +(-1,1) .. (0,-4.5) .. controls +(1,-1) and +(0,5) .. (3.5,-13) -- (3.5,-14)};
          \draw [omon,red] (1,11) -- (1,9) .. controls +(0,-3.5) and +(-1,1) .. (2.75,2);
          \draw [omon,red] (4.5,11) --(4.5,9) .. controls +(0,-3.5) and +(1,1) .. (2.75,2);
          \draw [omon,red] (2.75,2) .. controls +(0,-2.5) and +(1,1) .. (0,-4.5);
          \node [tmon,red] at (2.75,2) {};
          \draw [omon,red] (0,-4.5) .. controls +(-1,-1) and +(0,4) .. (-3.5,-13) -- (-3.5,-14);
          \clip (-7.5,11) -- (0,11) -- (0,-13) -- (7.5,-13) -- (7.5,-4.5) -- (-7.5,-4.5) -- cycle;
          \draw [omon,blue] (-3.5,12) -- \mypath -- (3.5,-15);
        \end{scope}
        \draw [bound] (-7.5,-13) rectangle (7.5,11);
      \end{tikzpicture}
    }
  \]
  \vspace{-10pt}
  \[\preq
    \mathclap{
      \begin{tikzpicture}[scale=.825]
        \node [ob] (s) at (-7,0) {$\oS$};
        \node [ob] (u) at (0,7) {$\oS$};
        \node [ob] at (0,7) {$\phantom{\oS}$};
        \node [ob] (tt) at (14,-7) {$\oS$};
        \node [ob] (dd) at (7,-14) {$\oS$};
        \draw [a] (s) -- node[arr,al] {$\mT_2$} (u);
        \draw [a] (dd) -- node[arr,br] {$\mT_2$} (tt);
        \draw [a] (s) .. controls +(2,-8) and +(-8,2) .. node[arr,bl] {$\mT_1$} (dd);
        \draw [eq] (u) .. controls +(8,-2) and +(-2,8) .. (tt);
        \draw [eq] (s) .. controls +(8,-2) and +(-2,8) .. (dd);
        \node [cell] at (5.75,-1.25) {$=$};
        \node [cell] at (-.25,-7.25) {$\eta$};
      \end{tikzpicture}
      \eqq
      \begin{tikzpicture}[scale=.825]
        \node [ob] (s) at (-7,0) {$\oS$};
        \node [ob] (u) at (0,7) {$\oS$};
        \node [ob] (tt) at (14,-7) {$\oS$};
        \node [ob] (dd) at (7,-14) {$\oS$};
        \draw [a] (s) -- node[arr,al] {$\mT_2$} (u);
        \draw [a] (dd) -- node[arr,br] {$\mT_2$} (tt);
        \draw [a] (u) .. controls +(2,-8) and +(-8,2) .. node[arr,bl,pos=.8] {$\mT_1$} (tt);
        \draw [eq] (u) .. controls +(8,-2) and +(-2,8) .. (tt);
        \draw [a] (s) .. controls +(2,-8) and +(-8,2) .. node[arr,bl] {$\mT_1$} (dd);
        \node [cell] at (.5,-5) {$\chi$};
        \node [cell] at (6.75,-.25) {$\eta$};
      \end{tikzpicture}
    }
    \qqq
    \mathclap{
      \begin{tikzpicture}[yscale=.7]
        \begin{scope}
          \clip [boundc] (-7.5,-13) rectangle (7.5,8);
          \fill [bg] (-7.5,-14) rectangle (7.5,11);
          \renewcommand{\mypath}{(-3.5,11) -- (-3.5,8) .. controls +(0,-5) and +(-1,1) .. (0,-.5) .. controls +(1,-1) and +(0,5) .. (3.5,-9.5) -- (3.5,-14)};
          \draw [omon,red] (-3.25,-7.5) -- (-3.25,-14);
          \node [tmon,red] at (-3.25,-7.5) {};
          \draw [omon,blue] (-3.5,12) -- \mypath -- (3.5,-15);
        \end{scope}
        \draw [bound] (-7.5,-13) rectangle (7.5,8);
      \end{tikzpicture}
      \eqq
      \begin{tikzpicture}[yscale=.7]
        \begin{scope}
          \clip [boundc] (-7.5,-13) rectangle (7.5,8);
          \fill [bg] (-7.5,-14) rectangle (7.5,11);
          \renewcommand{\mypath}{(-3.5,11) -- (-3.5,5) .. controls +(0,-5) and +(-1,1) .. (0,-4.5) .. controls +(1,-1) and +(0,5) .. (3.5,-13) -- (3.5,-14)};
          \draw [omon,red] (2.75,2) .. controls +(0,-2.5) and +(1,1) .. (0,-4.5);
          \node [tmon,red] at (2.75,2) {};
          \draw [omon,red] (0,-4.5) .. controls +(-1,-1) and +(0,4) .. (-3.5,-13) -- (-3.5,-14);
          \clip (-7.5,8) -- (0,8) -- (0,-13) -- (7.5,-13) -- (7.5,-4.5) -- (-7.5,-4.5) -- cycle; 
          \draw [omon,blue] (-3.5,12) -- \mypath -- (3.5,-15);
        \end{scope}
        \draw [bound] (-7.5,-13) rectangle (7.5,8);
      \end{tikzpicture}
    }
  \]
  and $\mT_1$ as a monad colax 1-cell $\mT_2 \to \mT_2$.
  \[\preq
    \mathclap
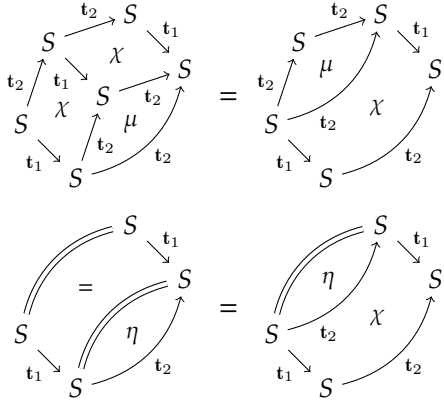
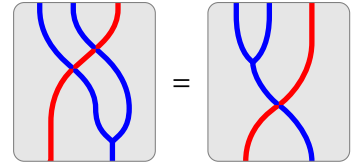
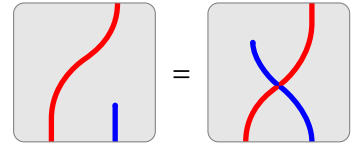
{
      \hspace{-2.5pt}
      \begin{tikzpicture}[scale=.825,xscale=-1]
        \node [ob] (s) at (-7,0) {$\oS$};
        \node [ob] (t) at (10.5,3.5) {$\oS$};
        \node [ob] (u) at (0,7) {$\oS$};
        \node [ob] (d) at (3.5,-3.5) {$\oS$};
        \node [ob] (tt) at (14,-7) {$\oS$};
        \node [ob] (dd) at (7,-14) {$\oS$};
        \draw [ab] (s) -- node[arr,ar] {$\mT_1$} (u);
        \draw [ab] (u) -- node[arr,above] {$\mT_2$} (t);
        \draw [ab] (s) -- node[arr,below,pos=.3] {$\mT_2$} (d);
        \draw [ab] (d) -- node[arr,bl] {$\mT_1$} (t);
        \draw [ab] (t) -- node[arr,left] {$\mT_2$} (tt);
        \draw [ab] (d) -- node[arr,right,pos=.7] {$\mT_2$} (dd);
        \draw [ab] (dd) -- node[arr,bl] {$\mT_1$} (tt);
        \draw [ab] (s) .. controls +(2,-8) and +(-8,2) .. node[arr,br] {$\mT_2$} (dd);
        \node [cell] at (1.75,1.75) {$\chi$};
        \node [cell] at (8.75,-5.25) {$\chi$};
        \node [cell] at (0,-7) {$\mu$};
      \end{tikzpicture}
      \eqq
      \hspace{-2.5pt}
      \begin{tikzpicture}[scale=.825,xscale=-1]
        \node [ob] (s) at (-7,0) {$\oS$};
        \node [ob] (t) at (10.5,3.5) {$\oS$};
        \node [ob] (u) at (0,7) {$\oS$};
        \node [ob] (tt) at (14,-7) {$\oS$};
        \node [ob] (dd) at (7,-14) {$\oS$};
        \draw [ab] (s) -- node[arr,ar] {$\mT_1$} (u);
        \draw [ab] (u) -- node[arr,above] {$\mT_2$} (t);
        \draw [ab] (t) -- node[arr,left] {$\mT_2$} (tt);
        \draw [ab] (dd) -- node[arr,bl] {$\mT_1$} (tt);
        \draw [ab] (u) .. controls +(2,-8) and +(-8,2) .. node[arr,br,pos=.8] {$\mT_2$} (tt);
        \draw [ab] (s) .. controls +(2,-8) and +(-8,2) .. node[arr,br] {$\mT_2$} (dd);
        \node [cell] at (.5,-5) {$\chi$};
        \node [cell] at (7,0) {$\mu$};
      \end{tikzpicture}
    }
    \qqq
    \mathclap{
      \begin{tikzpicture}[xscale=-.7894,yscale=.7]
        \begin{scope}
          \clip [boundc] (-10,-14) rectangle (9,10);
          \fill [bg] (-10,-14) rectangle (9,10);
          \renewcommand{\mypath}{(-5,10) -- (-5,9) .. controls +(0,-3) and +(-1,1) .. (-2,3) -- (1,0) .. controls +(2,-2) and +(0,3) .. (4,-8) -- (4,-14)};
          \draw [omon,blue] (1,10) -- (1,9) .. controls +(0,-3) and +(1,1) .. (-2,3);
          \draw [omon,blue] (5.5,10) -- (5.5,9) .. controls +(0,-5) and +(1,1) .. (1,0);
          \draw [omon,blue] (-2,3) .. controls +(-1,-1) and +(0,5) .. (-6.5,-6) .. controls +(0,-3) and +(-1,1) .. (-4.25,-11);
          \draw [omon,blue] (1,0) .. controls +(-1,-1) and +(0,3) .. (-2,-6) .. controls +(0,-3) and +(1,1) .. (-4.25,-11);
          \draw [omon,blue] (-4.25,-11) -- (-4.25,-14);
          \node [tmon,blue] at (-4.25,-11) {};
          \clip (-10,10) -- (-2,10) -- (-2,0) -- (10,0) -- (10,-14) -- (1,-14) -- (1,3) -- (-10,3) -- cycle;
          \draw [omon,red] (-5,11) -- \mypath -- (5,-15);
        \end{scope}
        \draw [bound] (-10,-14) rectangle (9,10);
      \end{tikzpicture}
      \eqq
      \begin{tikzpicture}[xscale=-1,yscale=.7]
        \begin{scope}
          \clip [boundc] (-7.5,-13) rectangle (7.5,11);
          \fill [bg] (-7.5,-14) rectangle (7.5,11);
          \renewcommand{\mypath}{(-3.5,11) -- (-3.5,5) .. controls +(0,-5) and +(-1,1) .. (0,-4.5) .. controls +(1,-1) and +(0,5) .. (3.5,-13) -- (3.5,-14)};
          \draw [omon,blue] (1,11) -- (1,9) .. controls +(0,-3.5) and +(-1,1) .. (2.75,2);
          \draw [omon,blue] (4.5,11) --(4.5,9) .. controls +(0,-3.5) and +(1,1) .. (2.75,2);
          \draw [omon,blue] (2.75,2) .. controls +(0,-2.5) and +(1,1) .. (0,-4.5);
          \node [tmon,blue] at (2.75,2) {};
          \draw [omon,blue] (0,-4.5) .. controls +(-1,-1) and +(0,4) .. (-3.5,-13) -- (-3.5,-14);
          \clip (-7.5,11) -- (0,11) -- (0,-13) -- (7.5,-13) -- (7.5,-4.5) -- (-7.5,-4.5) -- cycle;
          \draw [omon,red] (-3.5,12) -- \mypath -- (3.5,-15);
        \end{scope}
        \draw [bound] (-7.5,-13) rectangle (7.5,11);
      \end{tikzpicture}
    }
  \]
  \vspace{-10pt}
  \[\preq
    \mathclap{
      \begin{tikzpicture}[scale=.825,xscale=-1]
        \node [ob] (s) at (-7,0) {$\oS$};
        \node [ob] (u) at (0,7) {$\oS$};
        \node [ob] at (0,7) {$\phantom{\oS}$};
        \node [ob] (tt) at (14,-7) {$\oS$};
        \node [ob] (dd) at (7,-14) {$\oS$};
        \draw [ab] (s) -- node[arr,ar] {$\mT_1$} (u);
        \draw [ab] (dd) -- node[arr,bl] {$\mT_1$} (tt);
        \draw [ab] (s) .. controls +(2,-8) and +(-8,2) .. node[arr,br] {$\mT_2$} (dd);
        \draw [eq] (u) .. controls +(8,-2) and +(-2,8) .. (tt);
        \draw [eq] (s) .. controls +(8,-2) and +(-2,8) .. (dd);
        \node [cell] at (5.75,-1.25) {$=$};
        \node [cell] at (-.25,-7.25) {$\eta$};
      \end{tikzpicture}
      \eqq
      \begin{tikzpicture}[scale=.825,xscale=-1]
        \node [ob] (s) at (-7,0) {$\oS$};
        \node [ob] (u) at (0,7) {$\oS$};
        \node [ob] (tt) at (14,-7) {$\oS$};
        \node [ob] (dd) at (7,-14) {$\oS$};
        \draw [ab] (s) -- node[arr,ar] {$\mT_1$} (u);
        \draw [ab] (dd) -- node[arr,bl] {$\mT_1$} (tt);
        \draw [ab] (u) .. controls +(2,-8) and +(-8,2) .. node[arr,br,pos=.8] {$\mT_2$} (tt);
        \draw [eq] (u) .. controls +(8,-2) and +(-2,8) .. (tt);
        \draw [ab] (s) .. controls +(2,-8) and +(-8,2) .. node[arr,br] {$\mT_2$} (dd);
        \node [cell] at (.5,-5) {$\chi$};
        \node [cell] at (6.75,-.25) {$\eta$};
      \end{tikzpicture}
    }
    \qqq
    \mathclap{
      \begin{tikzpicture}[xscale=-1,yscale=.7]
        \begin{scope}
          \clip [boundc] (-7.5,-13) rectangle (7.5,8);
          \fill [bg] (-7.5,-14) rectangle (7.5,11);
          \renewcommand{\mypath}{(-3.5,11) -- (-3.5,8) .. controls +(0,-5) and +(-1,1) .. (0,-.5) .. controls +(1,-1) and +(0,5) .. (3.5,-9.5) -- (3.5,-14)};
          \draw [omon,blue] (-3.25,-7.5) -- (-3.25,-14);
          \node [tmon,blue] at (-3.25,-7.5) {};
          \draw [omon,red] (-3.5,12) -- \mypath -- (3.5,-15);
        \end{scope}
        \draw [bound] (-7.5,-13) rectangle (7.5,8);
      \end{tikzpicture}
      \eqq
      \begin{tikzpicture}[xscale=-1,yscale=.7]
        \begin{scope}
          \clip [boundc] (-7.5,-13) rectangle (7.5,8);
          \fill [bg] (-7.5,-14) rectangle (7.5,11);
          \renewcommand{\mypath}{(-3.5,11) -- (-3.5,5) .. controls +(0,-5) and +(-1,1) .. (0,-4.5) .. controls +(1,-1) and +(0,5) .. (3.5,-13) -- (3.5,-14)};
          \draw [omon,blue] (2.75,2) .. controls +(0,-2.5) and +(1,1) .. (0,-4.5);
          \node [tmon,blue] at (2.75,2) {};
          \draw [omon,blue] (0,-4.5) .. controls +(-1,-1) and +(0,4) .. (-3.5,-13) -- (-3.5,-14);
          \clip (-7.5,8) -- (0,8) -- (0,-13) -- (7.5,-13) -- (7.5,-4.5) -- (-7.5,-4.5) -- cycle; 
          \draw [omon,red] (-3.5,12) -- \mypath -- (3.5,-15);
        \end{scope}
        \draw [bound] (-7.5,-13) rectangle (7.5,8);
      \end{tikzpicture}
    }
  \]
\end{definition}

\begin{remark}\label{rem:mndmnd}
  A distributive law is the same as a monad in the 2-category
  $\Mnd(\tX)$ of monads, monad lax 1-cells, and monad 2-cells. Indeed,
  given that $\mT_2$ is a monad lax 1-cell from $\mT_1$ to $\mT_1$,
  the condition that $\mT_1$ is a monad colax 1-cell is the same as
  the condition that the multiplication and unit of $\mT_2$ are monad
  2-cells $\mT_2 \then \mT_2 \Rightarrow \mT_2$ and
  $\id_{\oS} \Rightarrow \mT_2$.


  Dually, a distributive law may also be viewed as a monad in the
  2-category of monads, monad colax 1-cells, and monad 2-cells.
\end{remark}

\begin{remark}
  We noted in \cref{rem:transfors} that the 2-category $\Mnd(\tX)$ of
  monads in $\tX$, monad lax 1-cells, and monad 2-cells is the
  2-category $[\tDelta, \tX]_{\lax}$ of functors from $\tDelta$ (the
  walking monad) to $\tX$, lax transformations, and
  modifications. This is the internal hom for the \emph{lax Gray
    product} $\grayl$, a monoidal structure on the category of
  2-categories~\cite{gray:formal}.

  We noted in \cref{rem:mndmnd} that a distributive law corresponds to
  a monad in $\Mnd(\tX)$. This in turn corresponds to a functor from
  $\tDelta$ to $\Mnd(\tX) = [\tDelta, \tX]_{\lax}$. This in turn
  corresponds to a functor from $\tDelta \grayl \tDelta$ to $\tX$
  (applying the tensor-hom adjunction, a.k.a.\ currying). Thus, just
  as a monad in $\tX$ amounts to diagram of shape $\tDelta$ in $\tX$,
  a pair of monads with a distributive law in $\tX$ amounts to a
  diagram of shape $\tDelta \grayl \tDelta$ in $\tX$.

  The cells of \cref{def:distributive} give a presentation of
  $\tDelta \grayl \tDelta$. In general, a string diagram for a cell in
  the lax Gray product $\tX \grayl \tY$ consists of a pair of string
  diagrams for cells in $\tX$ and $\tY$ superimposed, with 1-cell
  strings of $\tX$ able to cross NW/SE through NE/SW 1-cell strings of
  $\tY$, both able to slide freely over 2-cell nodes, in accordance
  with the description of lax transformations from
  \cref{rem:transfors}.  See also~\cite{morehouse:two}.
\end{remark}

\begin{proposition}
  Let $\mT_1$ and $\mT_2$ be monads on $\oS$ with distributive law
  $\chi \colon \mT_2 \then \mT_1 \Rightarrow \mT_1 \then \mT_2$.
  \[\preq \mathclap
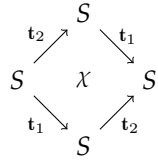
{
      \begin{tikzpicture}
        \node [ob] (s) at (-7,0)
        {$\oS$}; \node [ob] (d) at (0,-7)
        {$\oS$}; \node [ob] (u) at (0,7)
        {$\oS$}; \node [ob] (t) at (7,0)
        {$\oS$}; \draw [a] (s) -- node[arr,al]
        {$\mT_2$} (u); \draw [a] (u) -- node[arr,arcl]
        {$\mT_1$} (t); \draw [a] (s) -- node[arr,bl]
        {$\mT_1$} (d); \draw [a] (d) -- node[arr,br]
        {$\mT_2$} (t); \node [cell] at (0,0) {$\chi$};
      \end{tikzpicture}
    }
    \qqq
    \mathclap
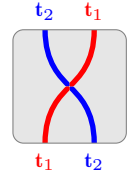
{
      \begin{tikzpicture}
        \begin{scope}
          \clip [boundc] (-6,-6) rectangle (6,6);
          \fill [bg] (-6,-6) rectangle (6,6);
          \renewcommand{\mypatha}{(-2.6,6) .. controls +(0,-3.5) and +(-1,1) .. (0,0) .. controls +(1,-1) and +(0,3.5).. (2.6,-6)}
          \renewcommand{\mypathb}{(2.6,6) .. controls +(0,-3.5) and +(1,1) .. (0,0) .. controls +(-1,-1) and +(0,3.5).. (-2.6,-6)}
          \begin{scope}
            \clip (-6,6) -- (0,6) -- (0,-6) -- (6,-6) -- (6,0) -- (-6,0) -- cycle;
            \draw [omon,blue] \mypatha;
          \end{scope}
          \begin{scope}
            \clip (6,6) -- (0,6) -- (0,-6) -- (-6,-6) -- (-6,0) -- (6,0) -- cycle;
            \draw [omon,red] \mypathb;
          \end{scope}
        \end{scope}
        \draw [bound] (-6,-6) rectangle (6,6);
        \node [label,blue,above] at (-2.6,6) {$\mT_2$};
        \node [label,red,above] at (2.6,6) {$\mT_1$};
        \node [label,blue,below] at (2.6,-6) {$\mT_2$};
        \node [label,red,below] at (-2.6,-6) {$\mT_1$};
      \end{tikzpicture}
    }
  \]
  
  \noindent
  If $\mT_1$ has algebra object $\Alg{\mT_1}{\oS}$, then $\chi$
  induces a monad $\emlift{\mT_2}$ on $\Alg{\mT_1}{\oS}$.
  \vspace{-5pt}
  
  \[\preq
    \mathclap{\emlift{\mT_2} \colon \Alg{\mT_1}{\oS} \to \Alg{\mT_1}{\oS}}
    \qqq
    \mathclap
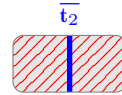
{
      \begin{tikzpicture}[yscale=.5]
        \begin{scope}
          \clip [boundc] (-6,-6) rectangle (6,6);
          \fill [bg] (-6,-6) rectangle (6,6);
          \fill [em=red] (-6,-6) rectangle (6,6);
          \draw [omon,blue] (0,-6) -- (0,6);
        \end{scope}
        \draw [bound] (-6,-6) rectangle (6,6);
        \node [label,blue,above] at (0,6) {$\emlift{\mT_2}$};
        \node [label,blue,below] at (0,-6) {$\phantom{\emlift{\mT_2}}$};
      \end{tikzpicture}
    }
  \]

  \noindent
  Dually, if $\mT_2$ has opalgebra object $\OpAlg{\mT_2}{\oS}$, then
  $\chi$ induces a monad $\kllift{\mT_1}$ on $\OpAlg{\mT_2}{\oS}$.
  \vspace{-5pt}
  
  \[\preq
    \mathclap{\kllift{\mT_1} \colon \OpAlg{\mT_2}{\oS} \to \OpAlg{\mT_2}{\oS}}
    \qqq
    \mathclap{
      \begin{tikzpicture}[yscale=.5]
        \begin{scope}
          \clip [boundc] (-6,-6) rectangle (6,6);
          \fill [bg] (-6,-6) rectangle (6,6);
          \fill [kl=blue] (-6,-6) rectangle (6,6);
          \draw [omon,red] (0,-6) -- (0,6);
        \end{scope}
        \draw [bound] (-6,-6) rectangle (6,6);
        \node [label,red,above] at (0,6) {$\kllift{\mT_1}$};
        \node [label,red,below] at (0,-6) {$\phantom{\kllift{\mT_1}}$};
      \end{tikzpicture}
    }
  \]
\end{proposition}
\begin{proof}
  By \cref{prop:formaltfae} and \cref{rem:emarrow}, assuming the
  relevant algebra objects exist, a functor of 2-categories
  $\tC \to \Mnd(\tX)$ amounts to a pair of functors $\tC \to \tX$ with
  a (strict) natural transformation between them whose components are
  1-cells of the form $\mcarT \colon \AlgTS \to \oS$. In the case that
  $\tC$ is the free 2-category $\tDelta$ containing a monad on an
  object, we find that a distributive law, viewed as a functor
  $\tDelta \to \Mnd(\tX)$ via \cref{rem:mndmnd}, consists of a monad
  $\mT_1$ on an object $\oS$, a monad $\mT_2$ on $\oS$, and a monad
  $\emlift{\mT_2}$ on $\Alg{\mT_1}{\oS}$ lifting $\mT_2$ along
  $\mcar{\mT_1} \colon \Alg{\mT_1}{\oS} \to \oS$.
\end{proof}


\pagebreak

\begin{proposition}\label{prop:dismonad}
  Let $\mT_1$ and $\mT_2$ be monads on $\oS$ with distributive law
  $\chi \colon \mT_2 \then \mT_1 \Rightarrow \mT_1 \then \mT_2$.
  \begin{enumerate}[label=(\roman*)]
  \item \label{item:dismonad}
    The composite $\mT_1 \then \mT_2$ is itself a monad on $\oS$.
    
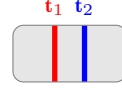
\vspace{-5pt}
    \[\preq
      \hspace{-\leftmargin}
      \mathclap{

      }
    \]
    
  \item\label{item:disthoms}
    The 2-cells $\mT_1 \then \eta \colon \mT_1 \Rightarrow \mT_1 \then
    \mT_2$ and $\eta \then \mT_2 \colon \mT_2 \Rightarrow \mT_1 \then
    \mT_2$ are both monad maps.
    \vspace{-5pt}
    \[\preq
      \hspace{-\leftmargin}
      \mathclap{
        %
      }
    \]
  \end{enumerate}
\end{proposition}
\begin{proof}
  For \labelcref{item:dismonad}, the multiplication is given by
  \vspace{-5pt}
  \[\preq
    \mathclap{
      %
    }
  \]
  and the unit is given by
  \vspace{-5pt}
  \[\preq
    \mathclap{
      %
    }
    \vspace{-5pt}
  \]
  We then have
  \vspace{-15pt}

  \[
    %
    \afterimage
  \]

  \noindent
  as desired. For \labelcref{item:disthoms}, the condition for
  multiplication is similarly straightforward to verify
  \vspace{-7.5pt}
  
  \[\preq
    \mathclap{
      %
    }
  \]
  and the condition for units is trivial.
\end{proof}




\begin{remark}
  The distributive law is recovered from the induced multiplication on
  $\mT_1 \then \mT_2$ by composing with a unit on both sides.
  \[\preq
    \mathclap{

    }
  \]
\end{remark}

\begin{proposition}\label{lem:distmodule}
  Let $\mT_1$ and $\mT_2$ be monads on $\oS$ with distributive law
  $\chi \colon \mT_2 \then \mT_1 \Rightarrow \mT_1 \then \mT_2$ and
  let $\bM \colon \oX \to \oS$ be an arbitrary 1-cell. Giving $\bM$
  the structure of a $(\mT_1 \then \mT_2)$-module
  \vspace{-5pt}
  \[\preq
    \mathclap{
      %
    }
    \vspace{-5pt}
  \]

  \noindent
  is equivalent to giving $\bM$ both the structure of a $\mT_1$-module
  and the structure of a $\mT_2$-module
  \[\preq
    \mathclap{
      %
    }
  \]
  satisfying the distributivity law
  \vspace{-5pt}
  \[\preq
    \mathclap{
      %
    }
  \]
\end{proposition}
\begin{proof}
  Any $(\mT_1 \then \mT_2)$-module structure on $\bM$ induces both
  $\mT_1$-module structure and $\mT_2$-module structure on $\bM$ via
  the monad maps from \cref{prop:dismonad}~\labelcref{item:disthoms}.
  \[\preq
    \hspace{-\leftmargin}
    \mathclap{
      %
    }
  \]
  
  Conversely, given a pair of a $\mT_1$-module structure 2-cell and a
  $\mT_2$-module structure 2-cell we may attempt to define a
  $(\mT_1 \then \mT_2)$-module structure 2-cell as their composite.
  \[\preq
    \hspace{-\leftmargin}
    \mathclap{
      %
    }
  \]

  Translating from such a pair to a putative
  $(\mT_1 \then \mT_2)$-module structure and back again yields the
  original pair by the $\mT_1$-module and $\mT_2$-module unit
  laws. The roundtrip on $(\mT_1 \then \mT_2)$-modules is also
  identity.
  \vspace{-15pt}

  \[
    %
    \hspace{7.5pt}
    \afterimage
  \]

  \noindent
  The $(\mT_1 \then \mT_2)$-module associativity law then corresponds to
  the distributivity law
  
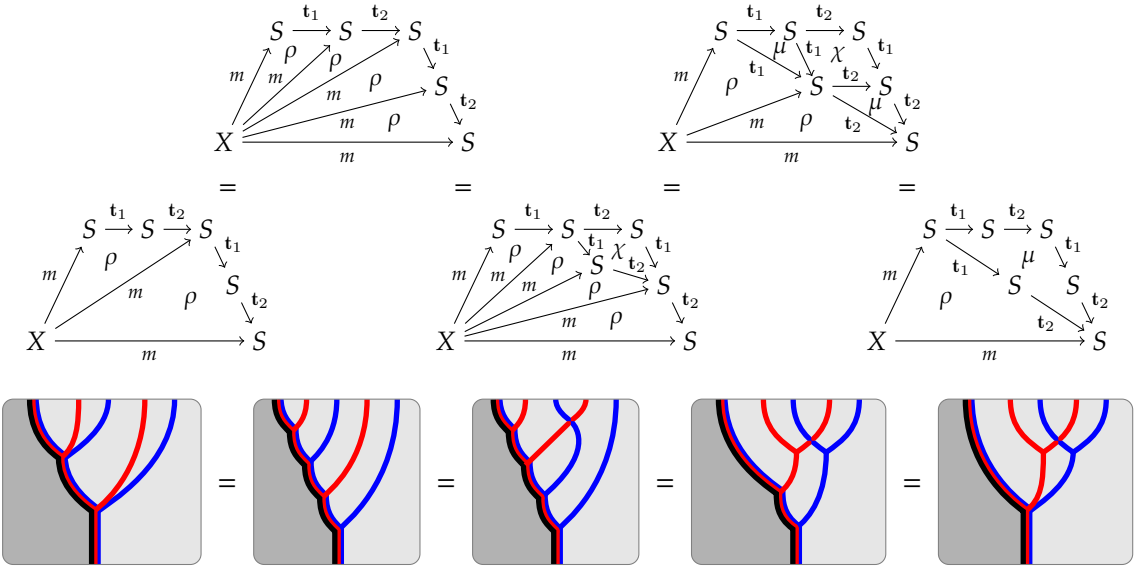
\vspace{-5pt}

  \[
    \mathclap{
      %
    }
    \afterimage
  \]

  \noindent
  and the $(\mT_1 \then \mT_2)$-module unit law holds automatically,
  corresponding to the $\mT_1$-module and $\mT_2$-module unit laws.
\end{proof}

\begin{proposition}\label{prop:distem}
  Given a distributive law of $\mT_1$ over $\mT_2$, if $\mT_1$ has an
  algebra object $\Alg{\mT_1}{\oS}$, then an algebra object
  $\Alg{\emlift{\mT_2}}{(\Alg{\mT_1}{\oS})}$ of $\emlift{\mT_2}$ is
  the same as an algebra object $\Alg{(\mT_1 \then \mT_2)}{\oS}$ of
  $\mT_1 \then \mT_2$.
  \[\preq
    \mathclap{
      \Alg{\emlift{\mT_2}}{(\Alg{\mT_1}{\oS})} \cong \Alg{(\mT_1 \then \mT_2)}{\oS} 
    }
    \qqq
    \mathclap
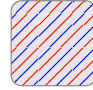
{
      \begin{tikzpicture}[scale=.75]
        \begin{scope}
          \clip [boundc] (-6,-6) rectangle (6,6);
          \fill [bg] (-6,-6) rectangle (6,6);
          \fill [emoff=red] (-6,-6) rectangle (6,6);
          \fill [emon=blue] (-6,-6) rectangle (6,6);
        \end{scope}
        \draw [bound] (-6,-6) rectangle (6,6);
      \end{tikzpicture}
    }
  \]
  In this case $\mcar{\emlift{\mT_2}} \then \mcar{\mT_1}$ is a
  universal $(\mT_1 \then \mT_2)$-module. (In other words, the
  composite of the \emph{monadic} 1-cells $\mcar{\emlift{\mT_2}}$ and
  $\mcar{\mT_1}$ is monadic.)
  
  Dually, if $\mT_2$ has an opalgebra object $\OpAlg{\mT_2}{\oS}$ we
  have
  \[\preq
    \mathclap{
      \OpAlg{\kllift{\mT_1}}{(\OpAlg{\mT_2}{\oS})} \cong \OpAlg{(\mT_1 \then \mT_2)}{\oS} 
    }
    \qqq
    \mathclap
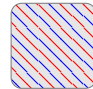
{
      \begin{tikzpicture}[scale=.75]
        \begin{scope}
          \clip [boundc] (-6,-6) rectangle (6,6);
          \fill [bg] (-6,-6) rectangle (6,6);
          \fill [klon=red] (-6,-6) rectangle (6,6);
          \fill [kloff=blue] (-6,-6) rectangle (6,6);
        \end{scope}
        \draw [bound] (-6,-6) rectangle (6,6);
      \end{tikzpicture}
    }
  \]
  assuming either side exists.
\end{proposition}
\begin{proof}
  First, observe that $\mcar{\mT_1}$ is a monad lax 1-cell from
  $\emlift{\mT_2}$ to $\mT_1 \then \mT_2$ with structure 2-cell
  \[\preq
    \mathclap
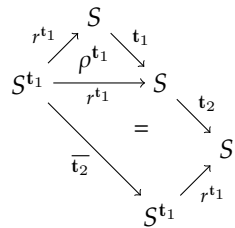
{
      \begin{tikzpicture}
        \node [ob] (s) at (-7,0) {$\Alg{\mT_1}{\oS}$};
        \node [ob] (u) at (0,7) {$\oS$};
        \node [ob] (t) at (7,0) {$\oS$};
        \node [ob] (tt) at (14,-7) {$\oS$};
        \node [ob] (dt) at (7,-14) {$\Alg{\mT_1}{\oS}$};
        \draw [a] (s) -- node[arr,al] {$\mcar{\mT_1}$} (u);
        \draw [a] (u) -- node[arr,ar] {$\mT_1$} (t);
        \draw [a] (s) -- node[arr,bcl] {$\mcar{\mT_1}$} (t);
        \draw [a] (t) -- node[arr,ar] {$\mT_2$} (tt);
        \draw [a] (s) -- node[arr,bl] {$\emlift{\mT_2}$} (dt);
        \draw [a] (dt) -- node[arr,br] {$\mcar{\mT_1}$} (tt);
        \node [cell] at (0,2.5) {$\mcarmod{\mT_1}$};
        \node [cell] at (5,-5) {$=$};
      \end{tikzpicture}
    }
    \qqq
    \mathclap
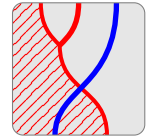
{
      \begin{tikzpicture}
        \begin{scope}
          \clip [boundc] (-7,-7) rectangle (7,7);
          \fill [bg] (-7,-7) rectangle (7,7);
          \renewcommand{\mypath}{(-4,7) .. controls +(0,-2.5) and +(-1,1) .. (-2,2.5) .. controls +(0,-5) and +(0,5) .. (3,-7)}
          \fill [em=red] \mypath -- (-7,-7) -- (-7,7) -- cycle;
          \draw [omon,red] (0,7) .. controls +(0,-2.5) and +(1,1) .. (-2,2.5);
          \draw [oem,red] \mypath;
          \draw [omon,blue] (4,7) .. controls +(0,-9) and +(0,5) .. (-2.5,-7);
        \end{scope}
        \draw [bound] (-7,-7) rectangle (7,7);
      \end{tikzpicture}
    }
  \]

  \noindent
  using the definition of $\emlift{\mT_2}$ and its multiplication
  $\emlift{\mu} \colon \emlift{\mT_2} \then \emlift{\mT_2} \Rightarrow
  \emlift{\mT_2}$ and unit
  $\emlift{\eta} \colon \id_{\Alg{\mT_1}{\oS}} \Rightarrow
  \emlift{\mT_2}$
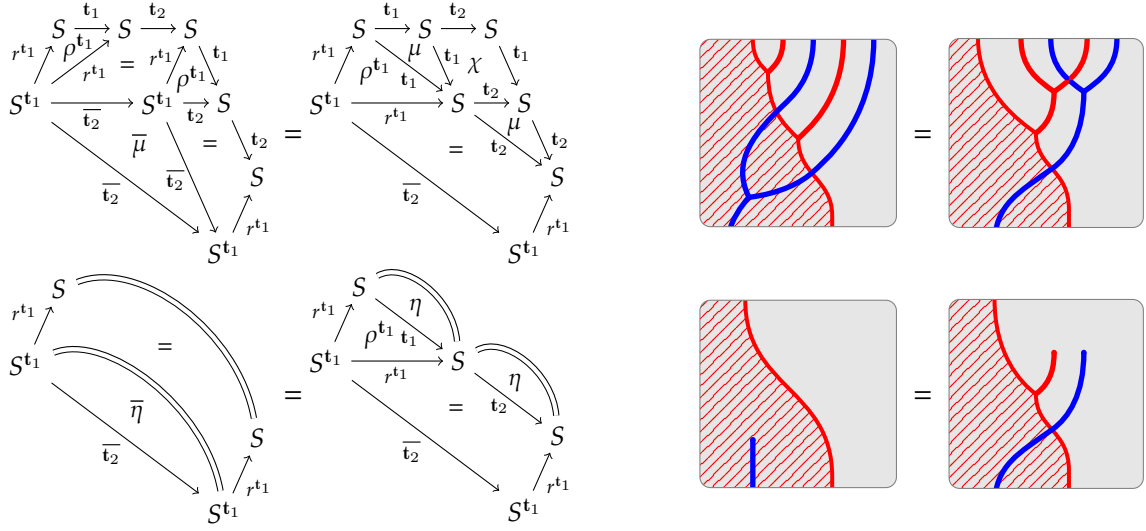
  \[\preq
    \mathclap{
      \begin{tikzpicture}[yscale=1.125,xslant=-.5]
        \node [ob] (s) at (-7,0) {$\Alg{\mT_1}{\oS}$};
        \node [ob] (u) at (0,7) {$\oS$};
        \node [ob] (t) at (7,0) {$\Alg{\mT_1}{\oS}$};
        \node [ob] (tt) at (14,-7) {$\oS$};
        \node [ob] (dt) at (7,-14) {$\Alg{\mT_1}{\oS}$};
        \node [ob] (ut) at (7,7) {$\oS$};
        \node [ob] (utt) at (14,0) {$\oS$};
        \node [ob] (uutt) at (14,7) {$\oS$};
        \draw [a] (s) -- node[arr,al] {$\mcar{\mT_1}$} (u);
        \draw [a] (s) -- node[arr,brcl] {$\mcar{\mT_1}$} (ut);
        \draw [a] (s) -- node[arr,bcl] {$\emlift{\mT_2}$} (t);
        \draw [a] (t) -- node[arr,alclcl] {$\mcar{\mT_1}$} (uutt);
        \draw [a] (t) -- node[arr,blcl,pos=.4] {$\emlift{\mT_2}$} (dt);
        \draw [a] (s) -- node[arr,bl] {$\emlift{\mT_2}$} (dt);
        \draw [a] (dt) -- node[arr,br] {$\mcar{\mT_1}$} (tt);
        \draw [a] (u) -- node[arr,above] {$\mT_1$} (ut);
        \draw [a] (ut) -- node[arr,above] {$\mT_2$} (uutt);
        \draw [a] (uutt) -- node[arr,ar] {$\mT_1$} (utt);
        \draw [a] (utt) -- node[arr,right] {$\mT_2$} (tt);
        \draw [a] (t) -- node[arr,bcl] {$\mT_2$} (utt);
        \node [cell] at (3,-4) {$\emlift{\mu}$};
        \node [cell] at (10.5,-4) {$=$};
        \node [cell] at (5.5,3.5) {$=$};
        \node [cell] at (1.25,5.25) {$\mcarmod{\mT_1}$};
        \node [cell] at (11.75,2.25) {$\mcarmod{\mT_1}$};
      \end{tikzpicture}
      \;\,\eqspace{-7.5pt}\;
      \begin{tikzpicture}[yscale=1.125,xslant=-.5]
        \node [ob] (s) at (-7,0) {$\Alg{\mT_1}{\oS}$};
        \node [ob] (u) at (0,7) {$\oS$};
        \node [ob] (t) at (7,0) {$\oS$};
        \node [ob] (tt) at (14,-7) {$\oS$};
        \node [ob] (dt) at (7,-14) {$\Alg{\mT_1}{\oS}$};
        \node [ob] (ut) at (7,7) {$\oS$};
        \node [ob] (utt) at (14,0) {$\oS$};
        \node [ob] (uutt) at (14,7) {$\oS$};
        \draw [a] (s) -- node[arr,al] {$\mcar{\mT_1}$} (u);
        \draw [a] (u) -- node[arr,blclcl,pos=.65] {$\mT_1$} (t);
        \draw [a] (s) -- node[arr,bcl] {$\mcar{\mT_1}$} (t);
        \draw [a] (t) -- node[arr,blclcl] {$\mT_2$} (tt);
        \draw [a] (s) -- node[arr,bl] {$\emlift{\mT_2}$} (dt);
        \draw [a] (dt) -- node[arr,br] {$\mcar{\mT_1}$} (tt);
        \draw [a] (u) -- node[arr,above] {$\mT_1$} (ut);
        \draw [a] (ut) -- node[arr,above] {$\mT_2$} (uutt);
        \draw [a] (uutt) -- node[arr,ar] {$\mT_1$} (utt);
        \draw [a] (utt) -- node[arr,right] {$\mT_2$} (tt);
        \draw [a] (ut) -- node[arr,arcl] {$\mT_1$} (t);
        \draw [a] (t) -- node[arr,acl] {$\mT_2$} (utt);
        \node [cell] at (-.25,3) {$\mcarmod{\mT_1}$};
        \node [cell] at (4.5,-4.5) {$=$};
        \node [cell] at (10.5,3.5) {$\chi$};
        \node [cell] at (4.75,4.75) {$\mu$};
        \node [cell] at (11.75,-2.25) {$\mu$};
      \end{tikzpicture}
    }
    \qqq
    \mathclap{
      \begin{tikzpicture}[xscale=.8,yscale=.77777]
        \begin{scope} 
          \clip [boundc] (-13,-16.5) rectangle (13,9);
          \fill [bg] (-13,-17) rectangle (13,9);
          \renewcommand{\mypath}{(-6,9) .. controls +(0,-2.5) and +(-1,1) .. (-4,4.5) .. controls +(0,-5) and +(-2,2) .. (0,-4.5) .. controls +(0,-5) and +(0,5) .. (4.5,-15) -- (4.5,-17)}
          \fill [em=red] \mypath -- (-13,-17) -- (-13,9) -- cycle;
          \draw [omon,red] (-2,9) .. controls +(0,-2.5) and +(1,1) .. (-4,4.5);
          \draw [omon,red] (6,9) .. controls +(0,-8.5) and +(1,1) .. (0,-4.5);
          \draw [oem,red] \mypath;
          \coordinate (mypoint) at (-6.5,-12.5);
          \node [tmon,blue] at (mypoint) {};
          \draw [omon,blue] (2,9) .. controls +(0,-5.5) and +(1,1) .. (-2,0) .. controls +(-3.5,-3.5) and +(-2.5,5) .. (mypoint);
          \draw [omon,blue] (10,9) .. controls +(0,-11.5) and +(1,1) .. (2,-9) .. controls +(-1.5,-1.5) and +(4,.5) .. (mypoint);
          \draw [omon,blue] (mypoint) .. controls +(-1,-1) and +(0,1) .. (-9,-17);
        \end{scope}
        \draw [bound] (-13,-16.5) rectangle (13,9);
      \end{tikzpicture}
      \eqq
      \begin{tikzpicture}[xscale=1.75,yscale=1.25]
        \begin{scope}
          \clip [boundc] (-5.25,-9) rectangle (6.5,7);
          \fill [bg] (-7,-9) rectangle (7,7);
          \renewcommand{\mypath}{(-3.5,7) .. controls +(0,-4.5) and +(-1,1) .. (0,-1) .. controls +(0,-3) and +(0,3.5) .. (2,-7.5) -- (2,-9)}
          \fill [em=red] \mypath -- (-7,-9) -- (-7,7) -- cycle;
          \draw [omon,red] (-.9,7) .. controls +(0,-2.5) and +(-1,1) .. (1.1,2.5);
          \draw [omon,red] (1.1,2.5) .. controls +(0,-1) and +(1,1) .. (0,-1);
          \node [tmon,red] at (1.1,2.5) {};
          \draw [oem,red] \mypath;
          \draw [omon,blue] (.9,7) .. controls +(0,-2.5) and +(-1,1) .. (2.9,2.5);
          \draw [omon,blue] (4.9,7) .. controls +(0,-2.5) and +(1,1) .. (2.9,2.5);
          \draw [omon,blue] (2.9,2.5) .. controls +(0,-3.5) and +(1,1) .. (1,-4) .. controls +(-2,-2) and +(0,3.5) .. (-2.5,-10.5);
          \node [tmon,blue] at (2.9,2.5) {};
          \clip (-7,-7) -- (-7,3.48) -- (7,3.48) -- (7,7) -- (2,7) -- (2,-7) -- cycle;
          \draw [omon,red] (3.1,7) .. controls +(0,-2.5) and +(1,1) .. (1.1,2.5);
        \end{scope}
        \draw [bound] (-5.25,-9) rectangle (6.5,7);
      \end{tikzpicture}
    }
    \vspace{-15pt}
  \]

  \[\preq
    \mathclap{
      \begin{tikzpicture}[yscale=1.125,xslant=-.5]
        \node [ob] (s) at (-7,0) {$\Alg{\mT_1}{\oS}$};
        \node [ob] (u) at (0,7) {$\oS$};
        \node [ob] (tt) at (14,-7) {$\oS$};
        \node [ob] (dt) at (7,-14) {$\Alg{\mT_1}{\oS}$};
        \draw [a] (s) -- node[arr,al] {$\mcar{\mT_1}$} (u);
        \draw [a] (s) -- node[arr,bl] {$\emlift{\mT_2}$} (dt);
        \draw [eq] (s) .. controls +(10,3) and +(3,10) .. (dt);
        \draw [a] (dt) -- node[arr,br] {$\mcar{\mT_1}$} (tt);
        \draw [eq] (u) .. controls +(10,3) and +(3,10) .. (tt);
        \node [cell] at (2.5,-4.5) {$\emlift{\eta}$};
        \node [cell] at (8.5,1.5) {$=$};
      \end{tikzpicture}
      \;\,\eqspace{-7.5pt}\;
      \begin{tikzpicture}[yscale=1.125,xslant=-.5]
        \node [ob] (s) at (-7,0) {$\Alg{\mT_1}{\oS}$};
        \node [ob] (u) at (0,7) {$\oS$};
        \node [ob] (t) at (7,0) {$\oS$};
        \node [ob] (tt) at (14,-7) {$\oS$};
        \node [ob] (dt) at (7,-14) {$\Alg{\mT_1}{\oS}$};
        \draw [a] (s) -- node[arr,al] {$\mcar{\mT_1}$} (u);
        \draw [a] (u) -- node[arr,blclcl,pos=.65] {$\mT_1$} (t);
        \draw [a] (s) -- node[arr,bcl] {$\mcar{\mT_1}$} (t);
        \draw [a] (t) -- node[arr,blclcl] {$\mT_2$} (tt);
        \draw [a] (s) -- node[arr,bl] {$\emlift{\mT_2}$} (dt);
        \draw [a] (dt) -- node[arr,br] {$\mcar{\mT_1}$} (tt);
        \draw [eq] (u) .. controls +(6,2.5) and +(2.5,6) .. (t);
        \draw [eq] (t) .. controls +(6,2.5) and +(2.5,6) .. (tt);
        \node [cell] at (0,2.5) {$\mcarmod{\mT_1}$};
        \node [cell] at (4.5,-4.5) {$=$};
        \node [cell] at (5,5) {$\eta$};
        \node [cell] at (12,-2) {$\eta$};
      \end{tikzpicture}
    }
    \qqq
    \mathclap{
      \begin{tikzpicture}[xscale=.8,yscale=.77777]
        \begin{scope} 
          \clip [boundc] (-13,-16.5) rectangle (13,9);
          \fill [bg] (-13,-17) rectangle (13,9);
          \renewcommand{\mypath}{(-7,9) .. controls +(0,-12) and +(0,12) .. (4.5,-15) -- (4.5,-17)}
          \fill [em=red] \mypath -- (-13,-17) -- (-13,9) -- cycle;
          \draw [oem,red] \mypath;
          \node [tmon,blue] at (-6,-10) {};
          \draw [omon,blue] (-6,-10) -- (-6,-17);
        \end{scope}
        \draw [bound] (-13,-16.5) rectangle (13,9);
      \end{tikzpicture}
      \eqq
      \begin{tikzpicture}[xscale=1.75,yscale=1.25]
        \begin{scope}
          \clip [boundc] (-5.25,-9) rectangle (6.5,7);
          \fill [bg] (-7,-9) rectangle (7,7);
          \renewcommand{\mypath}{(-2.5,7) .. controls +(0,-4.5) and +(-1,1) .. (0,-1) .. controls +(0,-3) and +(0,3.5) .. (2,-7.5) -- (2,-9)}
          \fill [em=red] \mypath -- (-7,-9) -- (-7,7) -- cycle;
          \draw [omon,red] (1.1,2.5) .. controls +(0,-1) and +(1,1) .. (0,-1);
          \node [tmon,red] at (1.1,2.5) {};
          \draw [oem,red] \mypath;
          \draw [omon,blue] (2.9,2.5) .. controls +(0,-3.5) and +(1,1) .. (1,-4) .. controls +(-2,-2) and +(0,3.5) .. (-2.5,-10.5);
          \node [tmon,blue] at (2.9,2.5) {};
        \end{scope}
        \draw [bound] (-5.25,-9) rectangle (6.5,7);
      \end{tikzpicture}
    }
  \]

  We would like to show that the induced mapping from
  $\emlift{\mT_2}$-modules to $(\mT_1 \then \mT_2)$-modules is a
  natural isomorphism, and thus $\mcar{(\mT_1 \then\mT_2)}$ and
  $\mcar{\emlift{\mT_2}} \then \mcar{\mT_1}$ satisfy the same
  universal property.

  Given a $(\mT_1 \then\mT_2)$-module $\bM \colon \oX \to \oS$, we
  obtain $\mT_1$-module and $\mT_2$-module structures on $\bM$ as in
  \cref{lem:distmodule}. The composite $\bM \then \mT_2$ is also a
  $\mT_1$-module since $\mT_2$ is a monad lax 1-cell, and the
  distributivity law from \cref{lem:distmodule} says precisely that
  $\rho \colon \bM \then \mT_2 \Rightarrow \bM$ is a $\mT_1$-module
  map. We therefore obtain a 1-cell
  $\lift{\bM} \colon \oX \to \Alg{\mT_1}{\oS}$ and a 2-cell
  $\lift{\bM \then \mT_2} = \emlift{\bM} \then \emlift{\mT_2}
  \Rightarrow \emlift{\bM}$ using the universal property of the
  algebra object $\Alg{\mT_1}{\oS}$. The $\emlift{\mT_2}$-module laws
  for $\emlift{\bM}$ follow from the $\mT_2$-module laws for $\bM$,
  since $\mcar{\mT_1}$ may be cancelled on the right of 2-cells (i.e.\
  $\mcar{\mT_1}$ is faithful) by its universal property. Thus every
  $(\mT_1 \then \mT_2)$-module $\bM$ arises from a
  $\emlift{\mT_2}$-module, which by its definition is the unique one
  that gives rise to $\bM$.

  The corresponding fact about opmodules
  $\OpAlg{\kllift{\mT_1}}{(\OpAlg{\mT_2}{\oS})} \cong \OpAlg{(\mT_1
    \then \mT_2)}{\oS}$ is dual.
\end{proof}

\begin{remark}\label{rem:distemkl}
  In the following, assume algebra objects and opalgebra objects exist
  as necessary. By \cref{prop:formaltfae} and \cref{con:expand}, the
  fact that $\mT_2$ is a monad lax 1-cell $\mT_1 \to \mT_1$ allows us
  to expand the distributive law structure 2-cell as
  \[\preq
    \mathclap{

    }
  \]

  \noindent
  The multiplication $\mT_2 \then \mT_2 \Rightarrow \mT_2$ and
  unit $\id_{\oS} \Rightarrow \mT_2$, being monad 2-cells, also lift
  as\[\preq
    \mathclap{
      %
    }
  \]
  
  \noindent
  Here we see $\mcar{\mT_1}$ is a monad colax 1-cell
  $\emlift{\mT_2} \to \mT_2$ (as well as lax, since its structure
  2-cell is invertible). Similarly $\mlcar{\mT_1}$ is a monad colax
  1-cell $\mT_2 \to \emlift{\mT_2}$ (with structure 2-cell mate to the
  inverse of the previous), allowing the further
  expansion\footnote{This arrangement of adjoints is different from
    the \emph{distributive adjoint situations} considered
    in~\cite{beck:distributive}, which involve only algebra objects
    and no opalgebra objects.}  
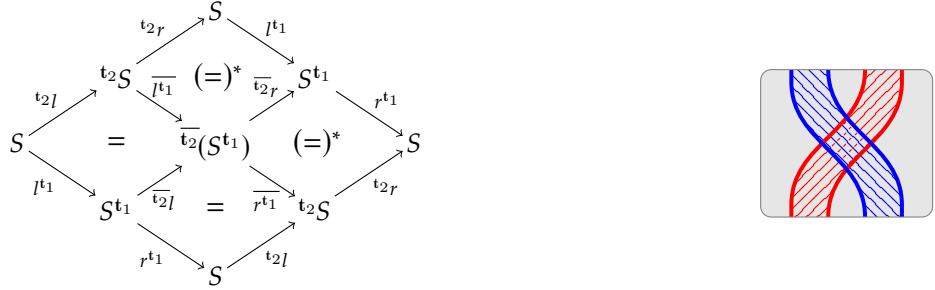
\vspace{-5pt}
  \[\preq
    \mathclap{
      %
    }
  \]

  Alternatively, we may proceed in the other order, writing
  
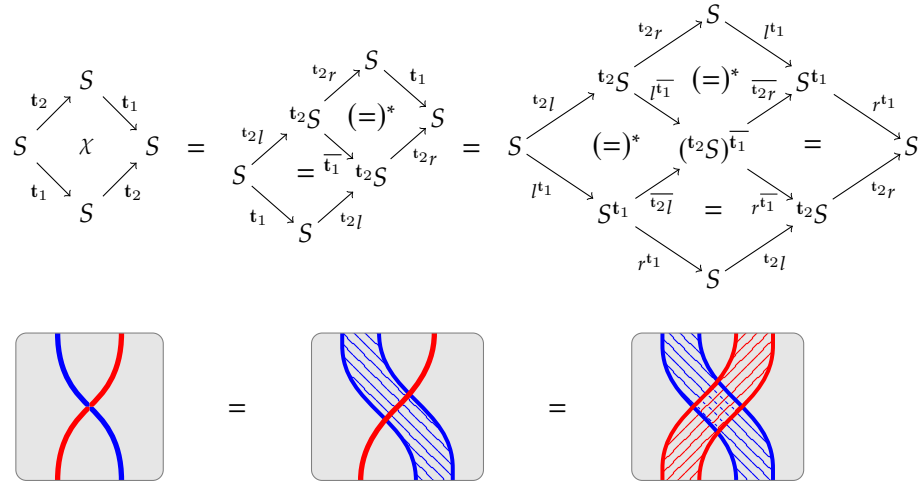
\vspace{-5pt}
  \[
    \hspace{20pt}
    %
    \afterimage
    \vspace{7.5pt}
  \]

  The objects $\OpAlg{\emlift{\mT_2}}{(\Alg{\mT_1}{\oS})}$ and
  $\Alg{\kllift{\mT_1}}{(\OpAlg{\mT_2}{\oS})}$ are not equivalent in
  general, but there is a 1-cell
  between them, given by applying \cref{prop:bimtfae} to the following
  $(\emlift{\mT_2}, \kllift{\mT_1})$-bimodule.
  \[\preq
    \mathclap{
      %
    }
  \]
\end{remark}

\pagebreak
\chapter{Codensity monads}

\begin{definition}\label{def:codensitymonad}
  A \emph{codensity monad}\footnote{The definition of codensity monad
    makes sense in a virtual 2-category.} (or
  \emph{endomorphism monad}) of a 1-cell
  \[\preq
    \mathclap{\fX\colon \oX \to \oS}
    \qqq
    \mathclap{
      %
    }
  \]

  \noindent
  That is, $\wmX$ is a \emph{right extension of $\fX$ along itself}.
  Explicitly, this means that for any cells \vspace{-10pt}
  \[\preq
    \mathclap{\fY\colon \oS \to \oS}
    \qqq
    \mathclap{
      %
    }
  \]
\end{definition}

\begin{convention}
  In string diagrams, we write the codensity monad $\wmX$ of
  $\fX \colon \oX \to \oS$ as a string the color of $\oX$ surrounded
  on both sides by a lip of $\fX$, and we write the universal 2-cell
  $\mXmod$ as a splitting of $\wmX$ into strings of $\fX$, to suggest
  the universal property via peeling up and pulling down the lip on
  the left side.

  As always, these are a priori merely textures we are using to
  display the cells, and only the manipulations given in the
  definition are allowed.
  %
  %

\end{convention}

\pagebreak

\begin{proposition}\label{prop:cismonad}
  Let $\fX \colon \oX \to \oS$ be a 1-cell with codensity monad
  $\wmX$.
  \begin{enumerate}[label=(\roman*)]
  \item\label{item:ismonad} The codensity monad $\wmX$ is a monad. That
    is, there exist canonical 2-cells
    \[\preq
      \hspace{-\leftmargin}
      \mathclap{
        %
      }
    \]
    satisfying the monad laws.
    
  \item\label{item:ismodule} The universal 2-cell
    \[\preq
      \hspace{-\leftmargin}
      \mathclap{
        %
      }
    \]
    exhibits $\fX$ as a $\wmX$-module.
    
  \item\label{item:unimonad} The codensity monad $\wmX$ is a
    \emph{universal monad for which $\fX$ is a module}. Explicitly,
    for any monad
    
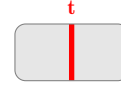
\vspace{-5pt}
    \[\preq
      \hspace{-\leftmargin}
      \mathclap{\mT \colon \oS \to \oS}
      \qqq
      \mathclap{
        %
      }
      \vspace{-5pt}
    \]
    and $\mT$-module structure on $\fX$
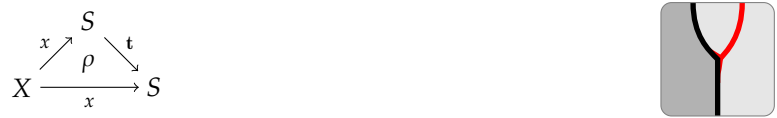
    \[\preq
      \hspace{-\leftmargin}
      \mathclap{
        %
      }
    \]
    the 2-cell induced by the universal property of $\wmX$
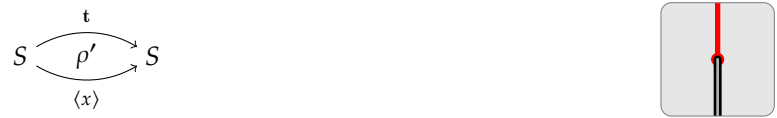
    \[\preq
      \hspace{-\leftmargin}
      \mathclap{
        %
      }
    \]
    is a monad map $\mT \to \wmX$.
  \end{enumerate}
\end{proposition}
\begin{proof}
  By the universal property of the codensity monad, there are unique
  2-cells as in \labelcref{item:ismonad} satisfying
  \[\preq
    \mathclap{
      %
    }
    \afterimage
    \vspace{-5pt}
  \]
  which are the additional required laws for
  \labelcref{item:ismodule}. Using these definitions we obtain the
  monad laws

  \[
    \mathclap{
      %
    \afterimage
    \vspace{5pt}
  \]

  \noindent
  cancelling $\mXmod$ by the uniqueness condition of the universal
  property.
  And using the module laws we obtain
  \labelcref{item:unimonad}
  \vspace{-25pt}

  \[
    \mathclap{
      \hspace{-2.5pt}
      %
    \afterimage
  \]

  \noindent
  again cancelling $\mXmod$.
\end{proof}

\begin{remark}
  More abstractly, for a fixed 1-cell $\fX$ one can define, as in~\cite[Proposition
  1.1]{mateo}, a monoidal category whose objects are 2-cells of the
  form
  
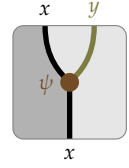
\vspace{-7.5pt}
  \[\preq
    \mathclap{
      %
    }
    \vspace{-2.5pt}
  \]
  with a forgetful monoidal functor to the monoidal category of
  1-cells $\oS \to \oS$. A monoid in the former monoidal category is
  the same as the structure on $\fX$ of a $\mT$-module for some monad
  $\mT$, and a terminal object is the same as a codensity monad
  $\wmX$. The statements in \cref{prop:cismonad} then follow by noting
  that a terminal object in any monoidal category admits a unique
  monoid structure, which is moreover the terminal monoid.
\end{remark}

\begin{remark}
  Codensity monads generalize monads induced by adjunctions. If $\fL$
  is a 1-cell left adjoint to the 1-cell $\fR$, then $\fL \then \fR$
  is a codensity monad $\wcd{\fR}$ of $\fR$.
  \vspace{-5pt}
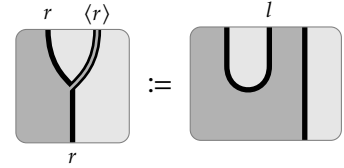
  \[\preq
    \mathclap{
      \begin{tikzpicture}
        \node [ob] (s) at (-7,0) {$\oX$};
        \node [ob] (m) at (0,7) {$\oS$};
        \node [ob] (t) at (7,0) {$\oS$};
        \draw [a] (s) -- node[arr,al] {$\fR$} (m);
        \draw [a] (m) -- node[arr,ar] {$\wcd{\fR}$} (t);
        \draw [a] (s) -- node[arr,below] {$\fR$} (t);
        \node [cell] at (0,2.5) {$\cdmod{\fR}$};
      \end{tikzpicture}
      \eqquad\coloneqq\eqquad
      \begin{tikzpicture}[xscale=.75]
        \node [ob] (s) at (-7,0) {$\oX$};
        \node [ob] (m) at (0,7) {$\oS$};
        \node [ob] (t) at (7,0) {$\oX$};
        \node [ob] (tt) at (17.5,0) {$\oS$};
        \draw [a] (s) -- node[arr,al] {$\fR$} (m);
        \draw [eq] (s) -- (t);
        \draw [a] (m) -- node[arr,ar] {$\fL$} (t);
        \draw [a] (t) -- node[arr,above] {$\fR$} (tt);
        \node [cell] at (0,2.75) {$\epsilon$};
      \end{tikzpicture}
    }
    \qqq
    \mathclap{
      \begin{tikzpicture}
        \begin{scope}
          \clip [boundc] (-6,-6) rectangle (6,6);
          \fill [bg] (-6,-6) rectangle (6,6);
          \renewcommand{\mypath}{(-2.6,6) .. controls +(0,-3.5) and +(-1,1) .. (0,0) -- (0,-6)}
          \fill [bgd] \mypath -- (-6,-6) -- (-6,6) -- cycle;
          \draw [o] \mypath;
          \renewcommand{\mypatha}{(2.6,6) .. controls +(0,-3.5) and +(1,1) .. (0,0)}
          \draw [ocdb] \mypatha;
          \draw [ocd] \mypatha -- (-.5,-.5);
        \end{scope}
        \draw [bound] (-6,-6) rectangle (6,6);
        \node [label,above] at (-2.6,6) {$\fR$};
        \node [label,acl] at (2.6,6) {$\wcd{\fR}$};
        \node [label,below] at (0,-6) {$\fR$};
        \node [label,bcl] at (2.6,6) {$\phantom{\wcd{\fR}}$};
      \end{tikzpicture}
      \eqquad\coloneqq\eqquad
      \begin{tikzpicture}
        \begin{scope}
          \clip [boundc] (-8,-6) rectangle (8,6);
          \fill [bg] (-8,-6) rectangle (8,6);
          \begin{scope}[shift={(1.125,0)}]
            \renewcommand{\mypatha}{(-5.25,6) -- (-5.25,1.875) .. controls +(0,-3.375) and +(0,-3.375) .. (-.75,1.875) -- (-.75,6)}
            \renewcommand{\mypathc}{(3,6) -- (3,-6)}
            \fill [bgd] (-10,-6) -- (-10,6) -- \mypatha -- \mypathc -- cycle;
            \draw [o] \mypatha;
            \draw [o] \mypathc;
          \end{scope}
        \end{scope}
        \draw [bound] (-8,-6) rectangle (8,6);
        \node [label,above] at (.375,6) {$\fL$};
        \node [label,below] at (-.75,-6) {$\phantom{\fL}$};
      \end{tikzpicture}
    }
    \vspace{-2.5pt}
  \]
\end{remark}

The following is a generalization of \cref{prop:formaltfae} for
codensity monads, where the algebra object $\Alg{\mT_1}{\oS_1}$ and
universal $\mT$-module $\mcar{\mT_1}$ are replaced by $\oX$ and the
1-cell $\fX$.

\begin{proposition}\label{prop:cformaltfae}
  Let $\fX \colon \oX \to \oS_1$ be a 1-cell with codensity monad
  $\wmX$, and let $\mT_2$ be a monad on $\oS_2$ with algebra object
  $\Alg{\mT_2}{\oS_2}$. Let
  \vspace{-12.5pt}
  \[\preq
    \mathclap{\fF \colon \oS_1 \to \oS_2}
    \qqq
    \mathclap
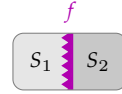
{
      \begin{tikzpicture}[yscale=.5]
        \begin{scope}
          \clip [boundc] (-6,-6) rectangle (6,6);
          \fill [bg] (-6,-6) rectangle (6,6);
          \fill [bgm] (0,-6) rectangle (6,6);
          \draw [olax,purp] (0,7) -- (0,-9);
        \end{scope}
        \draw [bound] (-6,-6) rectangle (6,6);
        \node [label] at (-3,0) {$\oS_1$};
        \node [label] at (3,0) {$\oS_2$};
        \node [label,purp,above] at (0,6) {$\fF$};
        \node [label,below] at (0,-6) {$\phantom{\fF}$};
      \end{tikzpicture}
    }
    \vspace{-10pt}
  \]
  be any 1-cell that \emph{preserves the right extension $\wmX$}. That
  is, for any cells \vspace{-12.5pt}
  \[\preq
    \mathclap{\fY\colon \oS_1 \to \oS_2}
    \qqq
    \mathclap{
      \begin{tikzpicture}[yscale=.5]
        \begin{scope}
          \clip [boundc] (-6,-6) rectangle (6,6);
          \fill [bg] (-6,-6) rectangle (6,6);
          \fill [bgm] (0,-6) rectangle (6,6);
          \draw [o,beige] (0,6) -- (0,-6);
        \end{scope}
        \draw [bound] (-6,-6) rectangle (6,6);
        \node [label,beige,above] at (0,6) {$\fY$};
        \node [label,below] at (0,-6) {$\phantom{\fY}$};
      \end{tikzpicture}
    }
    \vspace{-20pt}
  \]
  \[\preq
    \mathclap{
      \begin{tikzpicture}[yscale=.5]
        \node [ob] (s) at (-8.5,7) {$\oX$};
        \node [ob] (u) at (0,7) {$\oS_1$};
        \node [ob] (d) at (0,-7) {$\oS_1$};
        \node [ob] (t) at (8.5,7) {$\oS_2$};
        \draw [a] (s) -- node[arr,above] {$\fX$} (u);
        \draw [a] (u) -- node[arr,acl] {$\fY$} (t);
        \draw [a] (s) -- node[arr,bl] {$\fX$} (d);
        \draw [a] (d) -- node[arr,br] {$\fF$} (t);
        \node [cell] at (0,1) {$\psi$};
      \end{tikzpicture}
    }
    \qqq
    \mathclap{
      \begin{tikzpicture}
        \begin{scope}
          \clip [boundc] (-6,-6) rectangle (6,6);
          \begin{scope}[shift={(-.75,0)}]
            \fill [bg] (-6,-6) rectangle (8,6);
            \renewcommand{\mypath}{(-2.6,6) .. controls +(0,-3.5) and +(-1,1) .. (0,0) -- (0,-6)}
            \renewcommand{\mypatha}{(2.6,6) .. controls +(0,-3.5) and +(1,1) .. (0,0)}
            \renewcommand{\mypathb}{(0,0) .. controls +(2.5,-1) and +(0,3) .. (4.5,-7) -- (4.5,-8)}
            \fill [bgd] \mypath -- (-6,-6) -- (-6,6) -- cycle;
            \fill [bgm] \mypatha -- \mypathb -- (8,-6) -- (8,6) -- cycle;
            \draw [o] \mypath;
            \draw [o,beige] \mypatha;
            \draw [olax,purp] \mypathb;
            \node [tcirc=choc] at (0,0) {};
          \end{scope}
        \end{scope}
        \draw [bound] (-6,-6) rectangle (6,6);
        \begin{scope}[shift={(-.75,0)}]
          \node [label] at (-2.9,-2) {$\oX$};
          \node [label,choc,right=2] at (.5,.75) {$\psi$};
          \node [label,above] at (-2.6,6) {$\fX$};
          \node [label,below] at (0,-6) {$\fX$};
        \end{scope}
      \end{tikzpicture}
    }
    \vspace{-7.5pt}
  \]
  
  \noindent
  there is a unique 2-cell
  \vspace{-5pt}
  \[\preq
    \mathclap{
      \begin{tikzpicture}
        \node [ob] (s) at (-7,0) {$\oS_1$};
        \node [ob] (d) at (0,-7) {$\oS_1$};
        \node [ob] (t) at (7,0) {$\oS_2$};
        \draw [a] (s) -- node[arr,above] {$\fY$} (t);
        \draw [a] (s) -- node[arr,bl] {$\mX$} (d);
        \draw [a] (d) -- node[arr,br] {$\fF$} (t);
        \node [cell] at (0,-2.5) {$\factor{\psi}$};
      \end{tikzpicture}
    }
    \qqq
    \mathclap{
      \begin{tikzpicture}
        \begin{scope}
          \clip [boundc] (-6,-6) rectangle (6,6);
          \fill [bg] (-6,-6) rectangle (6,6);
          \renewcommand{\mypath}{(-1.5,0) .. controls +(2.5,-1) and +(0,3) .. (3,-7) -- (3,-8)}
          \fill [bgm] (-1.5,6) -- \mypath -- (6,-6) -- (6,6) -- cycle;
          \draw [olax,purp] \mypath;
          \draw [o,beige] (-1.5,7) -- (-1.5,0);
          \draw [ocdb] (-1.5,0) -- (-1.5,-7);
          \draw [ocd] (-1.5,0) -- (-1.5,-7);
          \node [tcirc=choc] at (-1.5,0) {};
        \end{scope}
        \draw [bound] (-6,-6) rectangle (6,6);
        \node [label,choc,above=4.5,right=2] at (-1.5,0) {$\factor{\psi}$};
      \end{tikzpicture}
    }
    \vspace{-5pt}
  \]
  
  \noindent
  satisfying
  \[\preq
    \mathclap{
      \begin{tikzpicture}
        \node [ob] (s) at (-7,0) {$\oX$};
        \node [ob] (m) at (0,7) {$\oS_1$};
        \node [ob] (t) at (7,0) {$\oS_1$};
        \node [ob] (tt) at (14,7) {$\oS_2$};
        \draw [a] (s) -- node[arr,al] {$\fX$} (m);
        \draw [a] (m) -- node[arr,blcl,pos=.6] {$\mX$} (t);
        \draw [a] (s) -- node[arr,below] {$\fX$} (t);
        \draw [a] (m) -- node[arr,above] {$\fY$} (tt);
        \draw [a] (t) -- node[arr,br] {$\fF$} (tt);
        \node [cell] at (-.75,2.5) {$\mXmod$};
        \node [cell] at (7,4.5) {$\factor{\psi}$};
      \end{tikzpicture}
      \eqq
      \begin{tikzpicture}[yscale=.5]
        \node [ob] (s) at (-8.5,7) {$\oX$};
        \node [ob] (u) at (0,7) {$\oS_1$};
        \node [ob] (d) at (0,-7) {$\oS_1$};
        \node [ob] (t) at (8.5,7) {$\oS_2$};
        \draw [a] (s) -- node[arr,above] {$\fX$} (u);
        \draw [a] (u) -- node[arr,acl] {$\fY$} (t);
        \draw [a] (s) -- node[arr,bl] {$\fX$} (d);
        \draw [a] (d) -- node[arr,br] {$\fF$} (t);
        \node [cell] at (0,1) {$\psi$};
      \end{tikzpicture}
    }
    \qqq
    \mathclap{
      \begin{tikzpicture}[scale=.8571]
        \begin{scope}
          \clip [boundc] (-7,-7) rectangle (7,7);
          \begin{scope}[shift={(-1.5,0)}]
            \fill [bg] (-7,-7) rectangle (9,7);
            \renewcommand{\mypatha}{(-2.6,8) -- (-2.6,2.5) .. controls +(0,-3.5) and +(-1,1) .. (0,-2.5) -- (0,-7)}
            \renewcommand{\mypathb}{(2.6,2.5) .. controls +(2.5,-1) and +(0,3) .. (6,-4.5) -- (6,-8.5)}
            \fill [bgd] \mypatha -- (-7,-7) -- (-7,7) -- cycle;
            \fill [bgm] (2.6,8) -- \mypathb -- (9,-7) -- (9,7) -- cycle;
            \draw [o,beige] (2.6,8) -- (2.6,2.5);
            \draw [o] \mypatha;
            \draw [olax,purp] \mypathb;
            \draw [ocdb] (2.6,2.5) .. controls +(0,-3.5) and +(1,1) .. (0,-2.5);
            \draw [ocd] (2.6,2.5) .. controls +(0,-3.5) and +(1,1) .. (0,-2.5) --  (-.5,-3);
            \node [tcirc=choc] at (2.6,2.5) {};
          \end{scope}
        \end{scope}
        \draw [bound] (-7,-7) rectangle (7,7);
      \end{tikzpicture}
      \eqq
      \begin{tikzpicture}
        \begin{scope}
          \clip [boundc] (-6,-6) rectangle (6,6);
          \begin{scope}[shift={(-.75,0)}]
            \fill [bg] (-6,-6) rectangle (8,6);
            \renewcommand{\mypath}{(-2.6,6) .. controls +(0,-3.5) and +(-1,1) .. (0,0) -- (0,-6)}
            \renewcommand{\mypatha}{(2.6,6) .. controls +(0,-3.5) and +(1,1) .. (0,0)}
            \renewcommand{\mypathb}{(0,0) .. controls +(2.5,-1) and +(0,3) .. (4.5,-7) -- (4.5,-8)}
            \fill [bgd] \mypath -- (-6,-6) -- (-6,6) -- cycle;
            \fill [bgm] \mypatha -- \mypathb -- (8,-6) -- (8,6) -- cycle;
            \draw [o] \mypath;
            \draw [o,beige] \mypatha;
            \draw [olax,purp] \mypathb;
            \node [tcirc=choc] at (0,0) {};
          \end{scope}
        \end{scope}
        \draw [bound] (-6,-6) rectangle (6,6);
      \end{tikzpicture}
    }
    \afterimage
  \]
  
  \noindent
  (For instance, any right adjoint 1-cell $\fF$ satisfies this
  property, and if $\fX$ is a right adjoint then any 1-cell $\fF$
  satisfies this property.)
  
  \begin{enumerate}[label=(\roman*)]
  \item\label{item:cmonadone}
    Giving $\fF$ the structure of a monad lax 1-cell $\wmX \to \mT_2$
    \vspace{-5pt}
    \[\preq
      \hspace{-\leftmargin}
      \mathclap{
        \begin{tikzpicture}[yscale=.875,xscale=.875]
          \node [ob] (s) at (-7,0) {$\oS_1$};
          \node [ob] (t) at (7,0) {$\oS_2$};
          \node [ob] (u) at (0,7) {$\oS_2$};
          \node [ob] (d) at (0,-7) {$\oS_1$};
          \draw [a] (s) -- node[arr,al] {$\fF$} (u);
          \draw [a] (u) -- node[arr,ar] {$\mT_2$} (t);
          \draw [a] (s) -- node[arr,bl] {$\mX$} (d);
          \draw [a] (d) -- node[arr,br] {$\fF$} (t);
          \node [cell] at (0,0) {$\chi$};
        \end{tikzpicture}
      }
      \qqq
      \mathclap{
        \begin{tikzpicture}
          \begin{scope}
            \clip [boundc] (-6,-6) rectangle (6,6);
            \fill [bg] (-6,-6) rectangle (6,6);
            \renewcommand{\mypath}{(-3,6) .. controls +(0,-3) and +(-1,1) .. (0,0) .. controls +(1,-1) and +(0,3) .. (3,-6)};
            \fill [bgm] \mypath -- (6,-6) -- (6,6) -- cycle;
            \draw [omon,blue] (3,6) .. controls +(0,-3) and +(1,1) .. (0,0);
            \renewcommand{\mypatha}{(0,0) .. controls +(-1,-1) and +(0,3) .. (-3,-6)};
            \draw [ocdb] \mypatha;
            \draw [ocd] \mypatha;
            \draw [olax,purp] (-3,7) -- \mypath -- (3,-8);
          \end{scope}
          \draw [bound] (-6,-6) rectangle (6,6);
          \node [label,blue,above] at (3,6) {$\mT_2$};
          \node [label,black,below] at (-3,-6) {$\mX$};
        \end{tikzpicture}
      }
      \vspace{-5pt}
    \]
    is equivalent to giving a 1-cell
    \vspace{-5pt}
    \[\preq
      \hspace{-\leftmargin}
      \mathclap{\cdlift{\fF} \colon \oX \to \Alg{\mT_2}{\oS_2}}
      \qqq
      \mathclap{
        \begin{tikzpicture}[yscale=.5]
          \begin{scope}
            \clip [boundc] (-6,-6) rectangle (6,6);
            \fill [bg] (-6,-6) rectangle (6,6);
            \fill [bgm] (0,-6) rectangle (6,6);
            \fill [bgd] (-6,-6) rectangle (0,6);
            \fill [em=blue] (0,-6) rectangle (6,6);
            \draw [olax,purp] (0,7) -- (0,-9);
          \end{scope}
          \draw [bound] (-6,-6) rectangle (6,6);
          \node [label,purp,above] at (0,6) {$\cdlift{\fF}$};
          \node [label,below] at (0,-6) {$\phantom{\cdlift{\fF}}$};
        \end{tikzpicture}
      }
      \vspace{-5pt}
    \]
    satisfying
    \vspace{-5pt}
    
    \[\preq
      \hspace{-\leftmargin}
      \mathclap{
        \begin{tikzpicture}[yscale=.875]
          \node [ob] (s) at (-7,0) {$\oX$};
          \node [ob] (u) at (0,7) {$\oS_1$};
          \node [ob] (t) at (7,0) {$\oS_2$};
          \node [ob] (d) at (0,-7) {$\Alg{\mT_2}{\oS_2}$};
          \draw [a,long] (s) -- node[arr,al] {$\fX$} (u);
          \draw [a,long] (u) -- node[arr,ar] {$\fF$} (t);
          \draw [a,long] (s) -- node[arr,bl] {$\cdlift{\fF}$} (d);
          \draw [a,long] (d) -- node[arr,br] {$\mcar{\mT_2}$} (t);
          \node at (0,0) {$=$};
        \end{tikzpicture}
      }
      \qqq
      \mathclap{
        \begin{tikzpicture}[yscale=.5]
          \begin{scope}
            \clip [boundc] (-8,-6) rectangle (8,6);
            \fill [bg] (-8,-6) rectangle (8,6);
            \fill [bgm] (3,-6) rectangle (8,6);
            \fill [bgd] (-8,-6) rectangle (-3,6);
            \draw [o,black] (-3,6) -- (-3,-6);
            \draw [olax,purp] (3,7) -- (3,-9);
          \end{scope}
          \draw [bound] (-8,-6) rectangle (8,6);
          \node [label,black,above] at (-3,6) {$\fX$};
          \node [label,black,below] at (-3,-6) {$\phantom{\fX}$};
        \end{tikzpicture}
        \eqq
        \begin{tikzpicture}[yscale=.5]
          \begin{scope}
            \clip [boundc] (-8,-6) rectangle (8,6);
            \fill [bg] (-8,-6) rectangle (8,6);
            \fill [bgm] (-3,-6) rectangle (8,6);
            \fill [bgd] (-8,-6) rectangle (-3,6);
            \fill [em=blue] (-3,-6) rectangle (3,6);
            \draw [olax,purp] (-3,7) -- (-3,-9);
            \draw [oem,blue] (3,7) -- (3,-7);
          \end{scope}
          \draw [bound] (-8,-6) rectangle (8,6);
          \node [label,blue,above] at (3,6) {$\mcar{\mT_2}$};
          \node [label,blue,below] at (3,-6) {$\phantom{\mcar{\mT_2}}$};
        \end{tikzpicture}
      }
    \]
    \afterimage

    More specifically, the canonical map, sending a monad lax 1-cell
    $\wmX \to \mT_2$ carried by $\fF$ to the 1-cell
    $\cdlift{\fF} \colon \oX \to \Alg{\mT_2}{\oS_2}$ corresponding to
    the $\mT_2$-module $\fX \then \fF$, is bijective.
  \end{enumerate}

  \noindent
  Now let $\fF_1, \fF_2 \colon S_1 \to S_2$ be monad lax 1-cells
  $\wmX \to \mT_2$ with $\fF_2$ preserving the right extension $\wmX$
  as above. Let
  $\cdlift{\fF_1},\cdlift{\fF_2} \colon \oX \to \Alg{\mT_2}{\oS_2}$
  denote the 1-cells corresponding to the $\mT_2$-modules
  $\fX \then \fF_1$ and $\fX \then \fF_2$ respectively, as above.
  Let $\gamma \colon \fF_1 \Rightarrow \fF_2$
  be an arbitrary 2-cell.
  
  \begin{enumerate}[label=(\roman*)]
    \setcounter{enumi}{1}
  \item\label{item:cmonadtwo} We have that $\gamma$ is a monad 2-cell
    $\fF_1 \Rightarrow \fF_2$ if and only if there is a (unique) 2-cell
    ${\cdlift{\gamma} \colon \cdlift{\fF_1} \Rightarrow
      \cdlift{\fF_2}}$ \vspace{-15pt}
    \[\preq
      \hspace{-\leftmargin}
      \mathclap{
        \begin{tikzpicture}
          \node [ob] (s) at (-7,0) {$\oX$};
          \node [ob] (t) at (7,0) {$\Alg{\mT_2}{S_2}$};
          \draw [a] (s) .. controls +(5,3) and +(-5,3) .. node[arr,above] {$\cdlift{\fF_1}$} (t);
          \draw [a] (s) .. controls +(5,-3) and +(-5,-3) .. node[arr,below] {$\cdlift{\fF_2}$} (t);
          \node [cell] at (0,0) {$\cdlift{\gamma}$};
        \end{tikzpicture}
      }
      \qqq
      \mathclap{
        \begin{tikzpicture}
          \begin{scope}
            \clip [boundc] (-6,-6) rectangle (6,6);
            \fill [bg] (-6,-6) rectangle (6,6);
            \fill [bgm] (0,-6) rectangle (6,6);
            \fill [bgd] (-6,-6) rectangle (0,6);
            \fill [em=blue] (0,-6) rectangle (6,6);
            \draw [olax,purp] (0,7) -- (0,0);
            \draw [olax,dpurp] (0,0) -- (0,-7);
            \node [tspec] at (0,0) {};
          \end{scope}
          \draw [bound] (-6,-6) rectangle (6,6);
          \node [label,purp,above] at (0,6) {$\cdlift{\fF_1}$};
          \node [label,purp,below] at (0,-6) {$\cdlift{\fF_2}$};
          \node [label,left=3.5] at (0,0) {$\cdlift{\gamma}$};
        \end{tikzpicture}
      }
      \vspace{-15pt}
    \]
    satisfying
    \vspace{-5pt}
    
    \[\preq
      \hspace{-\leftmargin}
      \mathclap{
        \begin{tikzpicture}[xscale=.625,yscale=1.125]
          \node [ob] (s) at (-7,0) {$\oS_1$};
          \node [ob] (t) at (7,0) {$\oS_2$};
          \draw [a] (s) .. controls +(5,3) and +(-5,3) .. node[arr,above] {$\fF_1$} (t);
          \draw [a] (s) .. controls +(5,-3) and +(-5,-3) .. node[arr,below] {$\fF_2$} (t);
          \node [cell] at (0,0) {$\gamma$};
          \node [ob] (ss) at (-21,0) {$\oX$};
          \draw [a] (ss) -- node[arr,above] {$\fX$} (s);
        \end{tikzpicture}
        \!\eqq\!
        \begin{tikzpicture}[xscale=.625,yscale=1.125]
          \node [ob] (s) at (-7,0) {$\oX$};
          \node [ob] (t) at (7,0) {$\Alg{\mT_2}{\oS_2}$};
          \draw [a,long] (s) .. controls +(5,3) and +(-5,3) .. node[arr,above] {$\cdlift{\fF_1}$} (t);
          \draw [a,long] (s) .. controls +(5,-3) and +(-5,-3) .. node[arr,below] {$\cdlift{\fF_2}$} (t);
          \node [cell] at (0,0) {$\cdlift{\gamma}$};
          \node [ob] (tt) at (21,0) {$\oS_2$};
          \draw [a] (t) -- node[arr,above] {$\mcar{\mT_2}$} (tt);
        \end{tikzpicture}
      }
      \qqq
      \mathclap{
        \begin{tikzpicture}
          \begin{scope}
            \clip [boundc] (-8,-6) rectangle (8,6);
            \fill [bg] (-8,-6) rectangle (8,6);
            \fill [bgm] (3,-6) rectangle (8,6);
            \fill [bgd] (-8,-6) rectangle (-3,6);
            \draw [o,black] (-3,6) -- (-3,-6);
            \draw [olax,purp] (3,7) -- (3,0);
            \draw [olax,dpurp] (3,0) -- (3,-7);
            \node [tmhomi] at (3,0) {};
          \end{scope}
          \draw [bound] (-8,-6) rectangle (8,6);
          \node [label,left=3.5] at (3,0) {$\gamma$};
        \end{tikzpicture}
        \eqq
        \begin{tikzpicture}
          \begin{scope}
            \clip [boundc] (-8,-6) rectangle (8,6);
            \fill [bg] (-8,-6) rectangle (8,6);
            \fill [bgm] (-3,-6) rectangle (8,6);
            \fill [bgd] (-8,-6) rectangle (-3,6);
            \fill [em=blue] (-3,-6) rectangle (3,6);
            \draw [olax,purp] (-3,7) -- (-3,0);
            \draw [olax,dpurp] (-3,0) -- (-3,-7);
            \draw [oem,blue] (3,6) -- (3,-6);
            \node [tmhomi] at (-3,0) {};
          \end{scope}
          \draw [bound] (-8,-6) rectangle (8,6);
        \end{tikzpicture}
      }
      \afterimage
    \]
    
  \item\label{item:cmonadspec} A monad specialization $\sigma$ from
    $\fF_1$ to $\fF_2$
    \[\preq
      \hspace{-\leftmargin}
      \mathclap{
        \begin{tikzpicture}
          \node [ob] (s) at (-7,0) {$\oS_1$};
          \node [ob] (m) at (0,-7) {$\oS_1$};
          \node [ob] (t) at (7,0) {$\oS_2$};
          \draw [a] (s) -- node[arr,above] {$\fF_1$} (t);
          \draw [a] (s) -- node[arr,bl] {$\mX$} (m);
          \draw [a] (m) -- node[arr,br] {$\fF_2$} (t);
          \node [cell] at (0,-2.5) {$\sigma$};
        \end{tikzpicture}
      }
      \qqq
      \mathclap{
        \begin{tikzpicture}
          \begin{scope}
            \clip [boundc] (-6,-6) rectangle (6,6);
            \fill [bg] (-6,-6) rectangle (6,6);
            \begin{scope}
              \clip (1.5,6) -- (1.5,-6) -- (6,-6) -- (6,6) -- cycle;
              \fill [bgm] (-6,-6) rectangle (6,6);
              \node [tspech,blue] at (1.5,0) {};
            \end{scope}
            \renewcommand{\mypath}{(1.5,0) .. controls +(-2.5,-1) and +(0,3) .. (-3,-7)}
            \draw [ocdb] \mypath;
            \begin{scope}
              \clip (1.5,6) -- (1.5,-6) -- (-6,-6) -- (-6,6) -- cycle;
              \node [tspechsb] at (1.5,0) {};
              \node [tspechs,bgd] at (1.5,0) {};
            \end{scope}
            \draw [ocd] \mypath;
            \draw [olax,purp] (1.5,7) -- (1.5,0);
            \draw [olax,dpurp] (1.5,0) -- (1.5,-7);
            \node [tspec] at (1.5,0) {};
          \end{scope}
          \draw [bound] (-6,-6) rectangle (6,6);
          \node [label,above=4.5,left=5] at (1.5,0) {$\sigma$};
        \end{tikzpicture}
      }
    \]
    is equivalent to an arbitrary 2-cell
    $\cdlift{\sigma} \colon \cdlift{\fF_1} \Rightarrow \cdlift{\fF_2}$.
    \[\preq
      \hspace{-\leftmargin}
      \mathclap{
        \begin{tikzpicture}
          \node [ob] (s) at (-7,0) {$\oX$};
          \node [ob] (t) at (7,0) {$\Alg{\mT_2}{\oS_2}$};
          \draw [a] (s) .. controls +(5,3) and +(-5,3) .. node[arr,above] {$\cdlift{\fF_1}$} (t);
          \draw [a] (s) .. controls +(5,-3) and +(-5,-3) .. node[arr,below] {$\cdlift{\fF_2}$} (t);
          \node [cell] at (0,0) {$\cdlift{\sigma}$};
        \end{tikzpicture}
      }
      \qqq
      \mathclap{
        \begin{tikzpicture}
          \begin{scope}
            \clip [boundc] (-6,-6) rectangle (6,6);
            \fill [bg] (-6,-6) rectangle (6,6);
            \fill [bgm] (0,-6) rectangle (6,6);
            \fill [bgd] (-6,-6) rectangle (0,6);
            \fill [em=blue] (0,-6) rectangle (6,6);
            \draw [olax,purp] (0,7) -- (0,0);
            \draw [olax,dpurp] (0,0) -- (0,-7);
            \node [tspec] at (0,0) {};
          \end{scope}
          \draw [bound] (-6,-6) rectangle (6,6);
          \node [label,left=3.5] at (0,0) {$\cdlift{\sigma}$};
        \end{tikzpicture}
      }
      \afterimage
    \]
  \end{enumerate}
\end{proposition}

The proof is the same as that of \cref{prop:formaltfae}, in the sense
that we are stripping the assumptions down to just what is necessary
to carry it through.  The content of that proof was to observe that
the definitions of (i) monad lax 1-cell structure on $\fF$, (ii) monad
2-cell structure on $\gamma$, and (iii) monad specialization are
\emph{mate} to the definitions of (i) module structure on
$\mcarT \then \fF$, (ii) module map structure on
$\mcarT \then \gamma$, and (iii) arbitrary module map.

The assumption that $\fF$ preserves the right extension $\wmX$
enables only the specific form of mate correspondence
needed.

\begin{proof}
  Giving $\mT_2$-module structure $\rho$ on $\fX \then \fF$ is
  equivalent to giving monad lax 1-cell structure
  $\chi \coloneqq \factor{\rho}$ on $\fF$
  \[\preq
    \mathclap{

    }
    \afterimage
    \vspace{5pt}
  \]
  since the module laws

  \[\preq
    \mathclap{
      %
    }
    \vspace{-10pt}
  \]
  \afterimage

  \noindent
  translate to the monad lax 1-cell laws under the correspondence
  induced by the universal property of the right extension.  Likewise,
  a $\mT_2$-module map
  $\phi \colon \fX \then \fF_1 \Rightarrow \fX \then \fF_2$. is
  equivalent to a monad specialization
  $\sigma \coloneqq \factor{\phi} \colon \fF_1 \Rightarrow \fF_2$
  \[\preq
    \mathclap{
      %
    }
    \vspace{-10pt}
  \]
  \afterimage

  \noindent
  and to give a monad 2-cell $\gamma$ is to give such a $\mT_2$-module
  map of the form $\fX \then \gamma$.
  \[\preq
    \mathclap{
      %
    }
  \]

  The result then follows from the universal property of the algebra
  object $\Alg{\mT_2}{\oS_2}$.
\end{proof}

\pagebreak
\chapter{Pushforward monads}

Codensity monads are a special case of the following more general
construction, introduced in~\cite{street:monads} and further studied
in~\cite{mateo}.  This section simply retraces the content of the
previous section in greater generality.

\begin{definition}
  Let
  
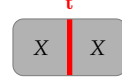
\vspace{-5pt}
  \[\preq
    \mathclap{\mT \colon \oX \to \oX}
    \qqq
    \mathclap{
      %
    }
    \vspace{-5pt}
  \]
  be a monad on $\oX$. A \emph{pushforward monad}\footnote{The
    definition of pushforward monad makes sense in an implicit
    2-category.} (a.k.a.\ \emph{extension}~\cite{street:monads}) of
  $\mT$ along a 1-cell 
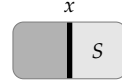
\vspace{-5pt}
  \[\preq
    \mathclap{\fX\colon \oX \to \oS}
    \qqq
    \mathclap{
      %
    }
    \vspace{-5pt}
  \]
  is a universal 1-cell $\mTX\colon \oS \to \oS$ equipped with a
  2-cell
  $\mTXmod \colon \fX \then \mTX \Rightarrow \mT \then
  \fX$.\footnote{To match our convention of diagrammatic order we are
    using the notation $\mTX$ reverse to $\fX_\#\mT$
    from~\cite{mateo}.}  
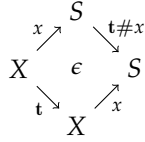
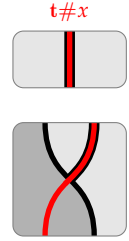
\vspace{-5pt}
  \[\preq
    \mathclap{\mTX\colon \oS \to \oS}
    \qqq
    \mathclap{
      %
    }
    \vspace{-5pt}
  \]
  \afterimage
  
  \noindent
  That is, $\mTX$ is a \emph{right extension of $\mT \then \fX$
    along $\fX$}.  Explicitly, this means that for any cells
  \vspace{-5pt}
  \[\preq
    \mathclap{\fY\colon \oS \to \oS}
    \qqq
    \mathclap{
      %
    }
  \]
\end{definition}

\begin{proposition}\label{prop:pfismonad}
  Let $\mT$ be a monad on $\oX$ with pushforward monad $\mTX$ along
  $\fX \colon \oX \to \oS$.
  \begin{enumerate}[label=(\roman*)]
  \item\label{item:pfismonad} The pushforward monad $\mTX$ is a
    monad. That is, there exist canonical 2-cells
    \[\preq
      \hspace{-\leftmargin}
      \mathclap{
        %
      }
    \]
    satisfying the monad laws.
    
  \item\label{item:pfismodule} The 2-cell
    \[\preq
      \hspace{-\leftmargin}
      \mathclap{
        %
      }
    \]
    exhibits $\fX$ as a monad lax 1-cell $\mT \to \mTX$.
    
  \item\label{item:pfunimonad} The pushforward monad $\mTX$ is a
    \emph{universal monad for which $\fX$ is monad lax
      1-cell}. Explicitly, for any monad
    
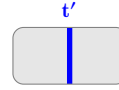
\vspace{-13.75pt}
    \[\preq
      \hspace{-\leftmargin}
      \mathclap{\mT' \colon \oS \to \oS}
      \qqq
      \mathclap{
        %
      }
      \vspace{-11.25pt}
    \]
    and monad lax 1-cell structure on $\fX$
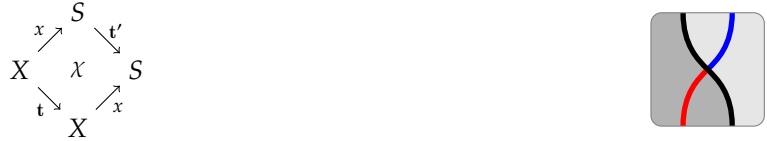
    \[\preq
      \hspace{-\leftmargin}
      \mathclap{
        %
      }
    \]
    the 2-cell induced by the universal property of $\mTX$
    
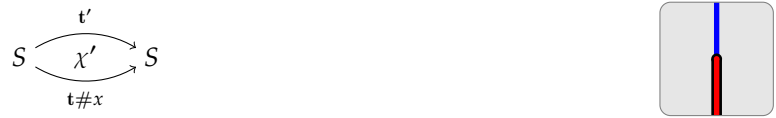
\vspace{-5pt}
    \[\preq
      \hspace{-\leftmargin}
      \mathclap{
        %
      }
      \vspace{-5pt}
    \]
    is a monad map $\mT' \to \mTX$.
  \end{enumerate}
\end{proposition}
\begin{proof}
  By the universal property of the pushforward monad, there are unique
  2-cells as in \labelcref{item:pfismonad} satisfying
  
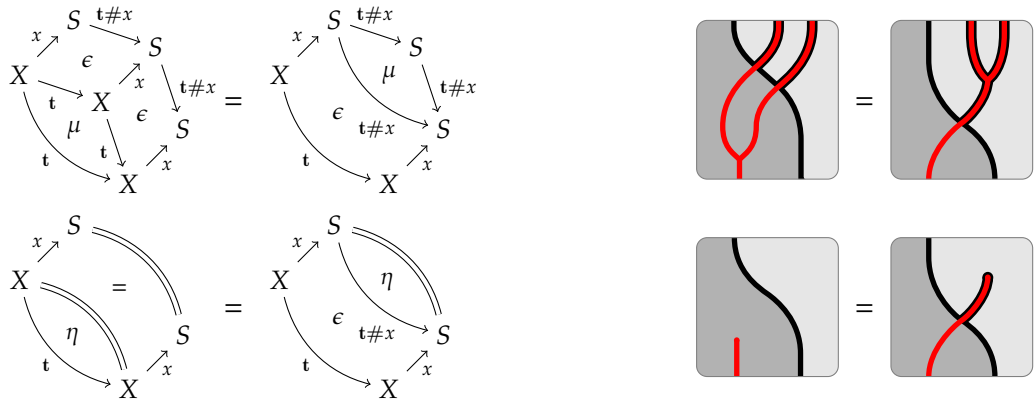
\vspace{-5pt}
  \[\preq
    \mathclap{
      %
    }
  \]

  \noindent
  which are the additional required laws for
  \labelcref{item:pfismodule}. Using these definitions we obtain the
  monad laws

  \[
    \mathclap{
      %
    \afterimage
  \]

  \noindent
  cancelling $\mTXmod$ by the uniqueness condition of the universal
  property.  And using the module laws we obtain
  \labelcref{item:pfunimonad} 
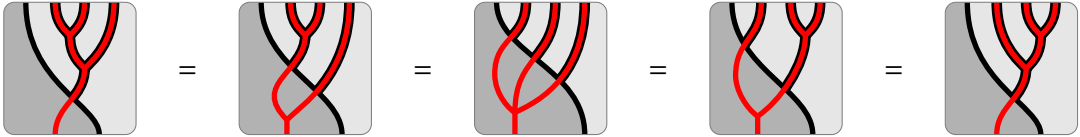
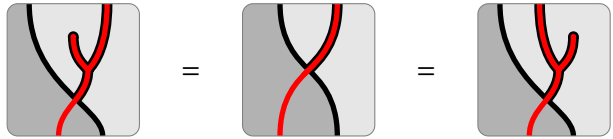
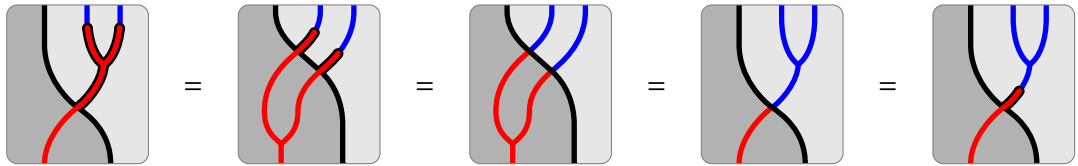
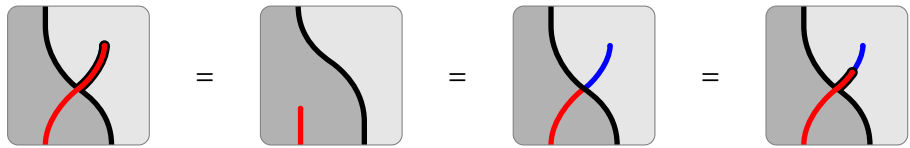
\vspace{-10pt}

  \[
    \mathclap{
      %
    \afterimage
  \]

  \noindent
  again cancelling $\mXmod$.
\end{proof}

\begin{remark}
  More abstractly, for a fixed $\mT$ and $\fX$ one can define, as
  in~\cite[Proposition 1.1]{mateo}, a monoidal category whose objects
  are 2-cells of the form
  \[\preq
    \mathclap{
      %
    }
  \]
  with a forgetful monoidal functor to the monoidal category of
  1-cells $\oS \to \oS$. A monoid in the former monoidal category is
  the same as the structure on $\fX$ of a monad lax 1-cell to some
  monad $\mT$, and a terminal object is the same as a pushforward
  monad $\mTX$. The statements in \cref{prop:pfismonad} then follow by
  noting that a terminal object in any monoidal category admits a
  unique monoid structure, which is moreover the terminal monoid.
\end{remark}
\begin{remark}
  Pushforward monads generalize monads transported along
  adjunctions. If $\mT$ is a monad on $\oX$ and
  $\fL \colon \oS \to \oX$ is a 1-cell left adjoint to the 1-cell
  $\fR \colon \oX \to \oS$, then $\fL \then \mT \then \fR$ is a
  pushforward monad $\pf{\mT}{\fR}$ of $\mT$ along $\fR$.
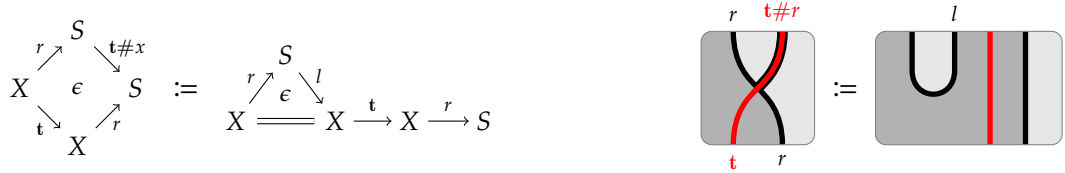
  \[\preq
    \mathclap{
      \begin{tikzpicture}[scale=.875]
        \node [ob] (s) at (-7,0) {$\oX$};
        \node [ob] (d) at (0,-7) {$\oX$};
        \node [ob] (u) at (0,7) {$\oS$};
        \node [ob] (t) at (7,0) {$\oS$};
        \draw [a] (s) -- node[arr,al] {$\fR$} (u);
        \draw [a] (u) -- node[arr,arcl] {$\mTX$} (t);
        \draw [a] (s) -- node[arr,bl] {$\mT$} (d);
        \draw [a] (d) -- node[arr,br] {$\fR$} (t);
        \node [cell] at (0,0) {$\mTXmod$};
      \end{tikzpicture}
      \eqquad\coloneqq\eqquad
      \begin{tikzpicture}[xscale=.75]
        \node [ob] (s) at (-7,0) {$\oX$};
        \node [ob] (m) at (0,7) {$\oS$};
        \node [ob] (t) at (7,0) {$\oX$};
        \node [ob] (tt) at (17.5,0) {$\oX$};
        \node [ob] (ttt) at (28,0) {$\oS$};
        \draw [a] (s) -- node[arr,al] {$\fR$} (m);
        \draw [eq] (s) -- (t);
        \draw [a] (m) -- node[arr,ar] {$\fL$} (t);
        \draw [a] (t) -- node[arr,above] {$\mT$} (tt);
        \draw [a] (tt) -- node[arr,above] {$\fR$} (ttt);
        \node [cell] at (0,2.75) {$\epsilon$};
      \end{tikzpicture}
    }
    \qqq
    \mathclap{
      \begin{tikzpicture}
        \begin{scope}
          \clip [boundc] (-6,-6) rectangle (6,6);
          \fill [bg] (-6,-6) rectangle (6,6);
          \renewcommand{\mypath}{(-2.6,6) .. controls +(0,-3.5) and +(-1,1) .. (0,0) .. controls +(1,-1) and +(0,3.5).. (2.6,-6)}
          \fill [bgd] \mypath -- (-6,-6) -- (-6,6) -- cycle;
          \draw [o] \mypath;
          \draw [omon,red] (0,0) .. controls +(-1,-1) and +(0,3.5).. (-2.6,-6);
          \renewcommand{\mypatha}{(2.6,6) .. controls +(0,-3.5) and +(1,1) .. (0,0)}
          \draw [opfb] \mypatha;
          \draw [omon,red] \mypatha -- (-.5,-.5);
        \end{scope}
        \draw [bound] (-6,-6) rectangle (6,6);
        \node [label,above] at (-2.6,6) {$\fR$};
        \node [label,below,red] at (-2.6,-6) {$\mT$};
        \node [label,above,red] at (2.6,6) {$\pf{\mT}{\fR}$};
        \node [label,below] at (2.6,-6) {$\fR$};
        \node [label,below,red] at (2.6,-6) {$\phantom{\pf{\mT}{\fR}}$};
      \end{tikzpicture}
      \eqquad\coloneqq\eqquad
      \begin{tikzpicture}
        \begin{scope}
          \clip [boundc] (-8,-6) rectangle (11.75,6);
          \fill [bg] (-8,-6) rectangle (11.75,6);
          \begin{scope}[shift={(1.125,0)}]
            \renewcommand{\mypatha}{(-5.25,6) -- (-5.25,1.875) .. controls +(0,-3.375) and +(0,-3.375) .. (-.75,1.875) -- (-.75,6)}
            \renewcommand{\mypathc}{(6.75,6) -- (6.75,-6)}
            \fill [bgd] (-10,-6) -- (-10,6) -- \mypatha -- \mypathc -- cycle;
            \draw [o] \mypatha;
            \draw [o] \mypathc;
            \draw [omon,red] (3,6) -- (3,-6);
          \end{scope}
        \end{scope}
        \draw [bound] (-8,-6) rectangle (11.75,6);
        \node [label,above] at (.375,6) {$\fL$};
        \node [label,below] at (.375,-6) {$\phantom{\fL}$};
      \end{tikzpicture}
    }
  \]
\end{remark}

\begin{remark}
  Pushforward monads generalize codensity monads. A codensity monad
  $\wmX$ of a 1-cell $\fX \colon \oX \to \oS$ is the same as a
  pushforward monad $\pf{\id_\oS}{\fX}$ of the identity monad
  $\id_\oS$ on $\oS$ along $\fX$.

  But also, it is conversely almost the case that codensity monads
  generalize pushforward monads. For any monad $\mT$ induced by a
  right adjoint 1-cell $\fR$, a pushforward monad $\mTX$ of $\mT$
  along a 1-cell $\fX$ is the same as a codensity monad
  $\wcd{\fR \then \fX}$ of $\fR \then \fX$. Indeed, as observed
  in~\cite{mateo}, by general properties of extensions, if $\fG$ is
  any functor that preserves the right extension $\pf{\mT}{\fF}$, then
  a pushforward monad $\pf{(\pf{\mT}{\fF})}{\fG}$ of $\pf{\mT}{\fF}$
  along $\fG$ is the same as a pushforward monad
  $\pf{\mT}{(\fF \then \fG)}$ of $\mT$ along $\fF \then
  \fG$. Therefore, since all functors $\fX$ preserve right extensions
  along right adjoint 1-cells $\fR$, we get
  \[\pf{\mT}{\fX} = \pf{(\pf{\id_\oS}{\fR})}{\fX} = \pf{\id_\oS}{(\fR \then \fX)} = \wcd{\fR \then \fX}.\]
  
  In particular, whenever $\mT$ admits an algebra object, $\mTX$ can
  be described as a codensity monad. \cref{prop:cformaltfae} then
  gives a description of monad lax 1-cells from $\mTX$ to monads with
  algebra objects, as well as monad 2-cells and monad specializations
  between them, in terms of cells between algebra objects.
\end{remark}

We end with our most general form of the correspondence used in the
proof of \cref{prop:formaltfae} and \cref{prop:cformaltfae}, which
still applies in the case that algebra objects do not exist.

\pagebreak

\begin{proposition}\label{prop:pfformaltfae}
  Let $\mT$ be a monad on $\oX$ with pushforward monad $\mTX$ along
  $\fX \colon \oX \to \oS_1$, and let $\mT_2$ be a monad on
  $\oS_2$. Let $\fF \colon \oS_1 \to \oS_2$ be any 1-cell that
  preserves the right extension $\mTX$.
  \begin{enumerate}[label=(\roman*)]
  \item\label{item:pfmonadone} Giving $\fF$ the structure of a monad
    lax 1-cell $\mTX \to \mT_2$ is equivalent to giving
    $\fX \then \fF$ the structure of a monad lax 1-cell
    $\mT \to \mT_2$.
  \end{enumerate}

  \noindent
  Now let $\fF_1, \fF_2 \colon \oS_1 \to \oS_2$ be monad lax 1-cells
  $\mTX \to \mT_2$ with $\fF_2$ preserving the right extension $\mTX$
  as above. Let $\gamma \colon \fF_1 \Rightarrow \fF_2$ be an
  arbitrary 2-cell.
  
  \begin{enumerate}[label=(\roman*)]
    \setcounter{enumi}{1}
  \item\label{item:pfmonadtwo} We have that $\gamma$ satisfies the law
    of a monad 2-cell $\fF_1 \Rightarrow \fF_2$ if and only if
    $\fX \then \gamma$ satisfies the law of a monad 2-cell
    $\fX \then \fF_1 \Rightarrow \fX \then \fF_2$.
  \item\label{item:pfmonadspec} Giving a monad specialization
    $\fF_1 \Rightarrow \fF_2$ is equivalent to giving a monad
    specialization $\fX \then \fF_1 \Rightarrow \fX \then \fF_2$.
  \end{enumerate}
\end{proposition}
\begin{proof}
  For \labelcref{item:pfmonadone}, by our assumption that $\fF$
  preserves the right extension $\mTX$, any 2-cell
  $\chi \colon \fX \then \fF \then \mT_2 \Rightarrow \mT \then \fX
  \then \fF$
  factors as
  \vspace{-7.5pt}
  \[\preq
    \mathclap{

    }
    \vspace{-5pt}
  \]

  \noindent
  and the monad lax 1-cell laws for $\fX \then \fF$
  \vspace{-7.5pt}
  \[\preq
    \mathclap{
      %
    }
    \vspace{-2.5pt}
  \]

  \noindent
  translate to the monad lax 1-cell laws for $\fF$ under the
  correspondence induced by universal property of the right extension.

  Likewise for \labelcref{item:pfmonadtwo}, given a 2-cell
  $\gamma \colon \fF_1 \Rightarrow \fF_2$, the monad 2-cell laws for
  $\fX \then \gamma$
  \vspace{-5pt}
  \[\preq
    \mathclap{
      %
    }
    \vspace{-5pt}
  \]

  \noindent
  correspond to the monad 2-cell laws for $\gamma$.

  Finally for \labelcref{item:pfmonadspec}, any 2-cell
  $\sigma \colon \fX \then \fF_1 \Rightarrow \mT \then \fX \then
  \fF_2$ factors as
  \[\preq
    \mathclap{
      %
    }
  \]

  \noindent
  and the monad specialization laws for $\sigma$
  \[\preq
    \mathclap{
      %
    }
  \]

  \noindent
  correspond to the monad specialization laws for $\factor{\sigma}$.
\end{proof}


\bibliographystyle{alpha}
\bibliography{references}

\end{document}